\documentclass[reqno]{amsart}

\usepackage[a4paper]{geometry}

\usepackage[T2A]{fontenc}
\usepackage[utf8]{inputenc}
\usepackage[ukrainian,russian,english]{babel}
\usepackage{comment}

\usepackage{amssymb}

\usepackage[mathscr]{eucal}
\usepackage{tikz}
\usetikzlibrary{positioning,calc,decorations.markings}
\usetikzlibrary{patterns}
\usetikzlibrary{calc}
\usepackage{mathtools}
\usepackage{enumerate}
\usepackage{accents}
\usepackage{dsfont}
\usepackage{multicol}

\usepackage{color}
\usepackage[colorlinks=true]{hyperref}
\hypersetup{urlcolor=blue,citecolor=red,linkcolor=blue}
\usepackage[initials]{amsrefs}
\definecolor{darkgreen}{rgb}{0.13, 0.55, 0.13}
\numberwithin{equation}{section}
\theoremstyle{plain}
\newtheorem{theorem}{Theorem}[section]
\newtheorem{proposition}[theorem]{Proposition}

\theoremstyle{definition}
\newtheorem*{rh-pb*}{Basic RH problem}
\newtheorem*{rh-I*}{RH problem normalized as $I$ at $\infty$}
\newtheorem*{sol-rh-pb*}{Soliton RH problem}
\newtheorem*{data*}{Data of this RH problem associated with $\BS{u_0(x)}$}
\theoremstyle{remark}
\newtheorem{remark}[theorem]{Remark}
\newtheorem*{notations*}{Notations}
\providecommand{\BS}[1]{\boldsymbol{#1}}  
\providecommand{\D}[1]{\mathbb{#1}}
\newcommand{\dd}{\mathrm{d}}
\newcommand{\eul}{\mathrm{e}}
\newcommand{\ii}{\mathrm{i}}
\newlength{\dhatheight}
\newcommand{\doublehat}[1]{%
    \settoheight{\dhatheight}{\ensuremath{\hat{#1}}}%
    \addtolength{\dhatheight}{-0.3ex}%
    \hat{\vphantom{\rule{1pt}{\dhatheight}}%
    \smash{\hat{#1}}}}
\renewcommand{\Im}{\operatorname{Im}}

\newcommand{\ord}{\mathrm{O}}
\DeclareMathOperator{\Res}{Res}
\newif\ifshort
\shorttrue
\title{The modified Camassa-Holm equation on the half line: a Riemann--Hilbert approach}
\author[Iryna Karpenko]{Iryna Karpenko}
\address{IK: Faculty of Mathematics\\ University of Vienna\\
Oskar-Morgenstern-Platz 1\\ 1090 Wien\\ Austria\\ and B. Verkin Institute for Low Temperature Physics and Engineering\\ 47, Nauky ave\\ 61103 Kharkiv\\ Ukraine}
\email{\href{mailto:iryna.karpenko@univie.ac.at}{iryna.karpenko@univie.ac.at}}

\author[Dmitry Shepelsky]{Dmitry Shepelsky}
\address{DS: B.Verkin Institute for Low Temperature Physics and Engineering, Kharkiv, Ukraine\\ and V.N.Karazin Kharkiv National University, Ukraine}
\email{\href{mailto:shepelsky@yahoo.com }{shepelsky@yahoo.com}}

\begin{document}

\begin{abstract}

We consider the initial-boundary value (IBV) problem for the modified Camassa--Holm (mCH) equation 
\[
\tilde m_t+\left((\tilde u^2-\tilde u_x^2+2\tilde u)\tilde m\right)_x = 0, \qquad \tilde m:=\tilde u-\tilde u_{xx}+1
\]
on the half-line
 $x \geq 0$. We  provide a characterization of the solution of the IBV
 problem in terms of the solution of a matrix Riemann–Hilbert (RH) factorization
 problem in the complex plane of the spectral parameter. The data of this RH problem
 are determined in terms of spectral functions associated with the initial and boundary values
 of the solution, whose compatibility is characterized in spectral terms.

\end{abstract}

\maketitle

\section{Introduction}

Nonlinear dispersive equations are partial differential equations that naturally arise in physical settings where dispersion dominates dissipation, notably hydrodynamics, nonlinear optics, plasma physics and Bose–Einstein condensates. 
The original Camassa--Holm (CH) equation  
\begin{equation}\label{CH}
u_t - u_{xxt}+ 3u u_x - 2u_x u_{xx} - u u_{xxx} = 0,
\end{equation}
which can also be written in terms of the momentum variable
\begin{equation}\label{CH-1}
m_t+\left(u m\right)_x + u_x m = 0,\quad m:= u-u_{xx},
\end{equation}
has been studied intensively over the last 30 years, due to its rich mathematical structure. It is a model for the unidirectional propagation of shallow water waves over a flat bottom, it is bi-Hamiltonian, and it is completely integrable with algebro-geometric solutions. The local and global well-posedness of the Cauchy problem for the CH equation has been studied extensively. In particular, it has both globally strong solutions  and blow-up solutions at finite time, and also it has globally weak solutions in $H^1(\mathbb{R})$.

In the case of the CH equation, the Riemann–Hilbert problem formalism for the initial boundary value problem was developed in \cite{BS08}, and the long-time behavior of the solutions of this problem was studied in~\cite{BS09}. 

Over the last few years, various modifications and generalizations of the CH equation have been introduced. Novikov \cite{N09} applied the perturbative symmetry approach in order to classify integrable equations of the form 
\[
\left(1-\partial_x^2\right) u_t = F(u, u_x, u_{xx}, u_{xxx}, \dots),\qquad u=u(x,t), \quad \partial_x=\partial/\partial x,
\]
assuming that $F$ is a homogeneous differential polynomial over $\mathbb{C}$, quadratic or cubic in $u$ and its $x$-derivatives. Such equations are sometimes called \textit{Camassa--Holm--type} equations.
   
In the list of equations presented in \cite{N09}, there is an equation with \emph{cubic} nonlinearity:
\begin{equation}\label{mCH-1}
m_t+\left((u^2-u_x^2)m\right)_x = 0, \quad m:= u-u_{xx}.
\end{equation}
This equation was first proposed by Fokas \cite{F95} and Olver and Rosenau independently in 1996 as a formal integrable equation, and in 2006 it was reintroduced be Qiao \cite{Q06}, who proposed a Lax pair representation for it. So it is sometimes referred to as the Fokas--Olver--Rosenau--Qiao (FORQ) equation, but is mostly known as the \textit{modified Camassa--Holm} (mCH) equation.

The mCH equation was recently considered by Chen, Hu and Liu (see \cite{CHL}) as a model of unidirectional propagation for shallow waves of moderate amplitude over a flat bottom, where the solution is related to the horizontal velocity at a certain water level. Notably, by applying a scale transformation corresponding to the short-wavelength limit, the mCH equation can be reduced to the Short Pulse Equation, which is an interesting model itself: it was proposed by Schäfer and Wayne in 2004 as an alternative (to the nonlinear Schrödinger equation) model for approximating the evolution of ultrashort intense infrared pulses in silica optics.

The mCH equation belongs to the class of peakon equations. The dynamical stability of peakons is discussed in \cite{QLL13}. The local well-posedness and wave-breaking mechanisms for the mCH equation and its generalizations, particularly, the mCH equation  with linear dispersion, are discussed in \cites{GLOQ13,CLQZ15}.  The local well-posedness for classical solutions and global weak solutions to \eqref{mCH-1} in Lagrangian coordinates are discussed in \cite{GL18}. In \cite{CS18} the authors discuss multipeakon solutions developing the inverse spectral method for the associated peakon system of ordinary differential equations. The Hamilton structure and  Liouville integrability of peakon systems are discussed in \cite{AK18}. The B\"acklund transformation for the mCH equation and the related nonlinear superposition formula are presented in \cite{WLM20}.

The inverse scattering method for the Cauchy problem for the mCH can be developed when equation \eqref{mCH-1} is considered on a non-zero background. A non-zero background provides that 
 the spectral problem in the associated Lax pair equations has a continuous spectrum, which allows  formulating the inverse spectral problem as a Riemann--Hilbert factorization problem with jump conditions across the real
axis (constituting the continuous spectrum). The Riemann--Hilbert formalism for this problem on a constant non-zero background was developed in \cite{BKS20}, and  the long-time behavior of its solutions was established in~\cite{BKS22}.

Following \cite{BKS20}, we introduce a new function $\tilde u$ by 
\begin{equation}\label{utilde}
u(x,t)=\tilde u(x-t,t)+1,
\end{equation}
in terms of  which the mCH equation reads as
\begin{equation}\label{mCH2_}
\tilde m_t+\left((\tilde u^2-\tilde u_x^2+2\tilde u)\tilde m\right)_x = 0,\quad
\tilde m\coloneqq\tilde u-\tilde u_{xx}+1.
\end{equation}
Introducing $\tilde\omega\coloneqq\tilde u^2-\tilde u_x^2+2\tilde u$, equations\eqref{mCH2_}  takes the form
\begin{subequations}\label{mCH2}
\begin{align}\label{mCH-2}
&\tilde m_t+\left(\tilde\omega\tilde m\right)_x = 0,\\
\label{tm}
&\tilde m\coloneqq\tilde u-\tilde u_{xx}+1,\\
\label{tom}
&\tilde\omega\coloneqq\tilde u^2-\tilde u_x^2+2\tilde u.
\end{align}
\end{subequations}

In the present paper, we consider the initial-boundary value (IBV) problem for  \eqref{mCH2}
on the half-line $x\ge 0$,
with the boundary conditions
\begin{equation}\label{boundary}
    \tilde u(0,t) = v_0(t), \quad \tilde u_x(0,t) = v_1(t), \quad 
    \tilde u_{xx}(0,t) =  v_2(t), \qquad
    0\leq t \leq T<\infty,
\end{equation}
and the initial condition
\begin{subequations}\label{ic}
\begin{equation}
    \tilde u(x,0) = \tilde u_0(x), \quad x \geq 0
\end{equation}
with $\tilde u_0(x)\to 0$ as $x\to\infty$.
Consequently, 
\begin{equation}
     \tilde m(x,0) = \tilde m_0(x), \quad x \geq 0,
\end{equation}
where $\tilde m_0(x):=\tilde u_0(x)-\tilde u_{0xx}(x)+1$.
\end{subequations}

We additionally assume that (i) $\tilde m_0(x)> 0$ for all $x\geq 0$ , (ii) $v_0(t)-v_2(t)+1>0$ 
(so that $\tilde m(0,t)> 0$) for all $t\geq 0$, and (iii) $\tilde u_0(0)=v_0(0)$. Observe that the condition $\tilde m(x,0)> 0$ ensures that $\tilde m(x,t)>0$ for all $t$ as long as solution exists (cf. \cite{KST23}, for the original Camassa--Holm equation see \cite{CE98}). This positivity is crucial for our analysis, as it provides analytical properties 
of Jost solutions to the associated Lax equations that are suitable for constructing an associated Riemann--Hilbert problem.

A natural motivation for studying a boundary value problem for an equation describing the wave propagation is the interpretation of the boundary conditions at $x = 0$ as the action of a wave maker located at this point, making the values of the solution and/or its certain derivatives prescribed, as functions of $t$, for $x$ = 0.

Notice that, under the transformation \eqref{utilde}, 
the boundary $x=0$ for the function $\tilde u$ corresponds to the 
characteristic line $x=t$ for the original variable $u$. 
Hence the boundary conditions \eqref{boundary} imply
\[
u(t,t)=v_0(t)+1,\qquad 
u_x(t,t)=v_1(t),\qquad 
u_{xx}(t,t)=v_2(t).
\]
Similarly, the initial condition \eqref{ic} reduces
\[
u(x,0)=\tilde u_0(x)+1,\qquad x\ge 0,
\]
so that the decay assumption $\tilde u_0(x)\to 0$ as $x\to\infty$ 
corresponds to $u(x,0)\to 1$.
Thus, the IBVP for $\tilde u$ on the half-line is equivalent, via 
\eqref{utilde}, to an IBVP for $u$ with boundary data prescribed 
along the moving boundary $x=t$ and initial data approaching the 
constant background $1$.

To investigate our problem, we employ a generalization of the inverse scattering transform
method -- specifically, its Riemann–Hilbert (RH) formulation \cite{trogdon2015riemann}) -- to the setting of initial boundary value (IBV) problems, which was introduced by Fokas (see \cite{F02}).

In this approach, the solution is reconstructed from the solution of a matrix RH problem formulated in the complex plane of the spectral parameter. 
the data for which can be calculated from the boundary and initial data. The method relies on the simultaneous spectral analysis of  both  equations of the Lax pair in the domain $0 \leq x < +\infty$, $0 \leq t \leq T$. On the ``parabolic boundary'' of this domain, the analysis leads to separate spectral problems for the equations of the Lax pair, the $x$-equation (for $t = 0$) and the $t$-equation (for $x = 0$).

In the general case, the construction of the RH problem requires the knowledge
of spectral functions associated with a “full” set of boundary values: for the mCH equation, they are $\tilde u(0,t)$, $\tilde u_x(0,t)$, and $\tilde u_{xx}(0,t)$. On the other hand, they cannot be all prescribed in the framework of a well-posed initial boundary value problem. On of out result is  that the compatibility of these boundary values can  be characterized in spectral terms, by the so-called “global relation”.

Similarly to the case of original CH equation, the analytic properties of distinguished solutions of the Lax 
pair (used in the construction of the associated RH problem) depends significantly 
on the sign of the boundary values $\tilde \omega(0,t)$. In what follows, we focus on the case $\tilde \omega(0,t)\leq 0$ whereas the case $\tilde \omega(0,t)\geq 0$ is briefly discussed in Appendix \ref{sec:geq} (cf. \cite{BS08}). If $\tilde \omega(0,t)$ is allowed to change the sign, 
the analysis can be performed step by step for half-bands $x\ge 0,\; t\in[T_j,T_{j+1}]$, within which the boundary values retain a fixed sign; in each step, the initial data for the next interval are obtained from the solution evaluated along the upper boundary of the preceding half-strip.

The paper is organized as follows. 

In Section \ref{sec:2}, we introduce appropriate transformations of the Lax pair equations
and the associated Jost solutions (``eigenfunctions'') and present detailed analytic properties of the eigenfunctions and the corresponding spectral functions (scattering coefficients). Furthermore, we discuss the compatibility conditions for the initial and boundary values (``global relations''). 

In Section \ref{sec:3}, we discuss the direct spectral problems.  

In Section \ref{sec:4}, we describe the inverse spectral problems for the initial and boundary values in terms of solutions of the associated RH problems. 

In Section \ref{sec:5}, assuming that the solution of \eqref{mCH-2}--\eqref{ic} $\tilde u(x,t)$ exist, we construct the  master RH problem and derive a representation of $\tilde u(x,t)$ in terms of this Riemann--Hilbert problem (see Proposition \ref{Prop:rep}).

In Section \ref{sec:mCHinyt}, we first discuss 
the relationship between solutions of the CH equation in the original variables $(x,t)$
and solutions of the CH equation in variables $(y,t)$ suitable for the RH formalism.
Then we discuss 
how a function solving, locally, the CH equation 
in $(y,t)$ variables appears from the solution of a Riemann-Hilbert
problem parametrized by $y$ and $t$.

In Section \ref{sec:7}, we characterize the solutions of the IBV problems \eqref{mCH-2}--\eqref{ic} in terms of solutions of associated Riemann--Hilbert problems (see Proposition \ref{prop:7}).

Throughout the text, we emphasize the differences in the implementation of the RH approach to the CH and mCH equations.

%-------------------%
\begin{notations*}
In what follows, $\sigma_1\coloneqq\left(\begin{smallmatrix}0&1\\1&0\end{smallmatrix}\right)$, $\sigma_2\coloneqq\left(\begin{smallmatrix}0&-\ii\\\ii&0\end{smallmatrix}\right)$, and $\sigma_3\coloneqq\left(\begin{smallmatrix}1&0\\0&-1\end{smallmatrix}\right)$ denote the standard Pauli matrices, 
$e^{\hat\sigma_3}A:=e^{\sigma_3}A e^{-\sigma_3}$,
$\mathbb{C}^+:=\{\lambda\in\mathbb{C}|\Im(\lambda)> 0\}$,
$\mathbb{C}^-:=\{\lambda\in\mathbb{C}|\Im(\lambda)< 0\}$, and $A \,\Delta\, B = (A \setminus B) \cup (B \setminus A)$. We also let $f^*(k)\coloneqq\overline{f(\bar k)}$ denote the Schwarz conjugate of a function $f(k)$, $k\in\D{C}$.
\end{notations*}
%-------------------%

\section{Eigenfunctions and Spectral Functions}\label{sec:2}

The modified Camassa--Holm equation \eqref{mCH2} is the compatibility condition for the pair of equations (see \cite{BKS20})
\begin{subequations}\label{Lax}
\begin{align}\label{Lax-x}
\Phi_x&=U\Phi,\\
\label{Lax-t}
\Phi_t&=V\Phi,
\end{align}
\end{subequations}
where the coefficients $U\equiv U(x,t,\lambda)$ and $V\equiv V(x,t,\lambda)$ are defined by
\begin{subequations}\label{Lax-UV}
\begin{align}\label{Lax-U}
U&=\frac{1}{2}\begin{pmatrix} -1 & \lambda \tilde m \\
-\lambda \tilde m & 1 \end{pmatrix},\\
\label{Lax-V}
V&=\begin{pmatrix}\lambda^{-2}+\frac{\tilde\omega}{2} &
-\lambda^{-1}(\tilde u-\tilde u_x+1)-\frac{\lambda\tilde\omega\tilde m}{2}\\
\lambda^{-1}(\tilde u+\tilde u_x+1)+\frac{\lambda\tilde\omega\tilde m}{2} & -\lambda^{-2}-\frac{\tilde\omega}{2}\end{pmatrix}.
\end{align}
\end{subequations}
The RH formalism for integrable nonlinear equations is based on using appropriately defined eigenfunctions, i.e., solutions of the Lax pair, whose behavior as functions of the spectral parameter is well-controlled in the extended complex plane.
Notice that the coefficient matrices $U$ and $V$ are traceless, which provides that the determinant of a matrix solution to \eqref{Lax} (composed from two vector solutions) is independent of $x$ and $t$.

Since the matrices $U$ and $V$ possess singularities in the extended complex
$\lambda$–plane at $\lambda = 0$ and $\lambda = \infty$,
one has to be careful when specifying the asymptotic behavior of solutions
at these points. On the other hand, 
 as $\lambda \to 0$ the matrix $U$ becomes independent of
$u$, a feature shared
by many Camassa--Holm–type systems.  
This suggests that the limit $\lambda \to 0$ can be employed to recover the
solution to the nonlinear equation through the associated Riemann–Hilbert
problem (cf. \cite{BS08}).

Assuming that we are given a solution $\tilde u(x,t)$ of the mCH equation in the domain $(x,t)\in(0,\infty)\times (0,T)$ such that $\tilde u(x,t)\to 0$ as $x\to\infty$ for all $t$, $\tilde m(x,0)>0$ and $\tilde \omega(0,t)\leq 0$, we analyze the analytic properties, in the complex plane of the spectral parameter, of the associated solutions of the Lax pair equations (eigenfunctions), in view of finding such a relation among them that can be viewed as 
a factorization Riemann--Hilbert problem.

\subsection{Eigenfunctions near $\lambda=\infty$}

In order to control the large $\lambda$ behavior, we introduce 
\[
\tilde \Phi_\infty(x,t,\lambda):= D(\lambda)\Phi(x,t,\lambda),
\]
where
\[
D(\lambda):=\begin{pmatrix}
1 & - \frac{\lambda}{1+\sqrt{1-\lambda^2}}  \\
- \frac{\lambda}{1+\sqrt{1-\lambda^2}} & 1 \\
\end{pmatrix},
\]
and the square root is chosen so that $\sqrt{1-\lambda^2}\sim\ii\lambda$ as $\lambda\to\infty$.
This transforms \eqref{Lax} into

\begin{subequations}\label{Lax-Q-form}
\begin{align}
&\tilde \Phi_{\infty x}+\frac{\tilde m\sqrt{1-\lambda^2}}{2}\sigma_3\tilde \Phi_\infty = U_\infty \tilde \Phi_\infty,\\
&\tilde \Phi_{\infty t} +\sqrt{1-\lambda^2}
\left(-\frac{1}{2}\tilde m\tilde\omega-\frac{1}{\lambda^2}\right)\sigma_3\tilde \Phi_\infty = V_\infty \tilde \Phi_\infty,
\end{align}
\end{subequations}
where $U_\infty\equiv U_\infty(x,t,\lambda)$ and $V_\infty\equiv V_\infty(x,t,\lambda)$ are given by
\begin{subequations}\label{Lax-1}
\begin{equation}\label{U-hat}
U_\infty=\frac{\lambda(\tilde m-1)}{2\sqrt{1-\lambda^2}}
\begin{pmatrix}
0 & 1 \\
-1 & 0 \\
\end{pmatrix}
+\frac{\tilde m-1}{2\sqrt{1-\lambda^2}}\sigma_3
\end{equation}
and 
\begin{equation}\label{hat-V}
\begin{aligned}
 V_\infty&=
\frac{1}{2 \sqrt{1-\lambda^2}}\left(\lambda\tilde\omega(\tilde m - 1)
+\frac{2 \tilde u}{\lambda}\right)
\begin{pmatrix}
0 & -1 \\
1 & 0
\end{pmatrix}
+\frac{ \tilde u_x}{\lambda}  \begin{pmatrix}
0 & 1 \\
1 & 0
\end{pmatrix}\\
&\quad-\frac{1}{\sqrt{1-\lambda^2}}\left(\tilde u +\frac{1}{2}(\tilde m-1)\tilde\omega\right) \sigma_3.
\end{aligned}
\end{equation}
\end{subequations}

Define $Q$ by 
\begin{subequations}\label{Qp}
\begin{equation}\label{Q}
Q(x,t,\lambda):= p(x,t,\lambda)\sigma_3, 
\end{equation}
with
\begin{equation}\label{p}
p(x,t,\lambda):= \sqrt{1-\lambda^2}\left(\frac{1}{2}\int_0^{x} \tilde m(\xi,t)\dd\xi-\frac{1}{2}\int_0^{t} (\tilde m \tilde \omega)(0,\tau)\dd\tau-\frac{t}{\lambda^2}\right).
\end{equation}
\end{subequations}
Then 
\[
p_x=\frac{\tilde m\sqrt{1-\lambda^2}}{2},\qquad 
p_t=\sqrt{1-\lambda^2}\left(-\frac{1}{2}\tilde m\tilde\omega-\frac{1}{\lambda^2}\right),\qquad 
p(0,0,\lambda)=0,
\]
and 
equations in \eqref{Lax-Q-form} take the form
\begin{subequations}\label{phi-tilde}
   \begin{align}
&\tilde \Phi_{\infty x}+Q_x\tilde \Phi_\infty = U_\infty \tilde \Phi_\infty,\\
&\tilde \Phi_{\infty t} +Q_t\tilde \Phi_\infty = V_\infty \tilde \Phi_\infty,
\end{align} 
\end{subequations}
which is appropriate to control the large $\lambda$ behavior of 
its solutions.

\begin{remark}
    Notice that \eqref{mCH2} implies the ``conservation law''
\begin{equation} \label{cons_law}
\nu(t)=\nu(0)-\eta(t),  
\end{equation}
where

\[
\nu(t):=\int_0^\infty (\tilde m -1)(\xi,t)\dd \xi,\qquad 
    \eta(t):=-\int_0^t (\tilde m \tilde \omega)(0,\tau)\dd \tau.
\]
    
\end{remark}

Introduce the solutions $\tilde\Phi_{\infty j}$, $j=1,2,3$ to \eqref{phi-tilde} through the solutions $\Phi_{\infty j}$ of the integral equations, which are fixed by the associated
initial 
integration point. Namely, $\tilde \Phi_{\infty j }(x,t,\lambda)=\Phi_{\infty j }(x,t,\lambda)\eul^{-p(x,t,\lambda)\sigma_3}$, where $\Phi_{\infty j}$, $j=1,2,3$ satisfy
the integral equation 
\begin{equation}\label{inteq_inf}
\Phi_{\infty}(x,t,\lambda)=I+\int_{(x^*,t^*)}^{(x,t)}
	\eul^{(p(y,\tau,\lambda)-p(x,t,\lambda))\hat\sigma_3}( U_\infty\Phi_\infty \dd y+V_\infty\Phi_\infty \dd \tau)(y,\tau,\lambda),
\end{equation}
where $(x^*,t^*)$ is taken respectively as $(0,T)$, $(0,0)$ and $(+\infty,t)$, and the integration paths are as in Figure \ref{fig:integration-paths} (due to the compatibily of the Lax pair equations, the integrals don't depend on integration paths). More precisely,
\begin{equation}\label{inteq_inf1}
    \begin{aligned}
\Phi_{\infty 1}(x,t,&\lambda)=I+
\int_{0}^{x}
	\eul^{-\frac{\sqrt{1-\lambda^2}}{2}\int_y^x\tilde m(\xi,t)\dd \xi\hat\sigma_3}( U_\infty\Phi_{\infty1})(y,t,\lambda) \dd y   \\
&-\eul^{-\frac{\sqrt{1-\lambda^2}}{2}\int_0^x\tilde m(\xi,t)\dd \xi\hat\sigma_3}
\int_{t}^{T}
	\eul^{\left(-\frac{\sqrt{1-\lambda^2}}{2}\int_t^\tau (\tilde m\tilde \omega)(0,s)\dd s-\frac{\sqrt{1-\lambda^2}}{\lambda^2}(\tau-t)\right)\hat\sigma_3}( V_\infty\Phi_{\infty1})(0,\tau,\lambda) \dd \tau,
\end{aligned}
\end{equation}
\begin{equation}\label{inteq_inf2}
    \begin{aligned}
\Phi_{\infty 2}(x,t,&\lambda)=I+
\int_{0}^{x}
	\eul^{-\frac{\sqrt{1-\lambda^2}}{2}\int_y^x\tilde m(\xi,t)\dd \xi\hat\sigma_3}( U_\infty\Phi_{\infty2})(y,t,\lambda) \dd y   \\
&+\eul^{-\frac{\sqrt{1-\lambda^2}}{2}\int_0^x\tilde m(\xi,t)\dd \xi\hat\sigma_3}
\int_{0}^{t}
	\eul^{\left(-\frac{\sqrt{1-\lambda^2}}{2}\int_t^\tau (\tilde m\tilde \omega)(0,s)\dd s-\frac{\sqrt{1-\lambda^2}}{\lambda^2}(\tau-t)\right)\hat\sigma_3}( V_\infty\Phi_{\infty2})(0,\tau,\lambda) \dd \tau,
\end{aligned}
\end{equation}    

\begin{equation}\label{inteq_inf3}
    \begin{aligned}
\Phi_{\infty 3}(x,t,\lambda)=&I-
\int_{x}^{+\infty}
	\eul^{\frac{\sqrt{1-\lambda^2}}{2}\int_x^y\tilde m(\xi,t)\dd \xi\hat\sigma_3}( U_\infty\Phi_{\infty3})(y,t,\lambda) \dd y.
\end{aligned}
\end{equation}

\begin{figure}[h]
\centering
\begin{tikzpicture}[scale=1]
  % Tunables
  \def\x{2.2}  % x-location of the evaluation point
  \def\t{1.3}  % t level
  \def\T{2.6}  % mark on tau-axis
  \def\W{3.2}  % panel width

  % ---------- Panel 1 ----------
  \begin{scope}
    % axes
    \draw[->] (0,0) -- (\W,0) node[below right] {$y$};
    \draw[->] (0,0) -- (0,\W) node[above left] {$\tau$};

    % labels
    \node[below left] at (0,0) {$0$};
    \node[below]      at (\x,0) {$x$};
    \node[left]       at (0,\t) {$t$};

    % T mark
    \fill (0,\T) circle (1.2pt) node[left] {$T$};

    % dashed guide
    \draw[dashed] (\x,0) -- (\x,\t);

    % path: down then right (arrow in the middle)
    \draw[very thick,postaction={decorate},
          decoration={markings, mark=at position 0.5 with {\arrow{latex}}}] 
          (0,\T) -- (0,\t);
    \draw[very thick,postaction={decorate},
          decoration={markings, mark=at position 0.5 with {\arrow{latex}}}] 
          (0,\t) -- (\x,\t);

    % evaluation point
    \fill (\x,\t) circle (1.4pt);
  \end{scope}

  % ---------- Panel 2 ----------
  \begin{scope}[xshift=5.2cm]
    \draw[->] (0,0) -- (\W,0) node[below right] {$y$};
    \draw[->] (0,0) -- (0,\W) node[above left] {$\tau$};

    \node[below left] at (0,0) {$0$};
    \node[below]      at (\x,0) {$x$};
    \node[left]       at (0,\t) {$t$};

    \fill (0,\T) circle (1.2pt) node[left] {$T$};

    \draw[dashed] (\x,0) -- (\x,\t);

    % path: up then right (arrow in the middle)
    \draw[very thick,postaction={decorate},
          decoration={markings, mark=at position 0.5 with {\arrow{latex}}}] 
          (0,0) -- (0,\t);
    \draw[very thick,postaction={decorate},
          decoration={markings, mark=at position 0.5 with {\arrow{latex}}}] 
          (0,\t) -- (\x,\t);

    \fill (\x,\t) circle (1.4pt);
  \end{scope}

  % ---------- Panel 3 ----------
  \begin{scope}[xshift=10.4cm]
    \draw[->] (0,0) -- (\W,0) node[below right] {$y$};
    \draw[->] (0,0) -- (0,\W) node[above left] {$\tau$};

    \node[below left] at (0,0) {$0$};
    \node[below]      at (\x,0) {$x$};
    \node[left]       at (0,\t) {$t$};

    \fill (0,\T) circle (1.2pt) node[left] {$T$};

    \draw[dashed] (\x,0) -- (\x,\t);
    \draw[dashed] (0,\t) -- (\x,\t);

    % path: from right to (x,t), arrow in the middle pointing left
    \draw[very thick,postaction={decorate},
          decoration={markings, mark=at position 0.5 with {\arrow{latex}}}] 
          (\W,\t) -- (\x,\t);

    \fill (\x,\t) circle (1.4pt);
  \end{scope}
\end{tikzpicture}

\caption{Paths of integration for $\Phi_{01}, \Phi_{02},$ and $\Phi_{03}$ 
($\Phi_{\infty1}, \Phi_{\infty2},$ and $\Phi_{\infty3}$).}
\label{fig:integration-paths}
\end{figure}
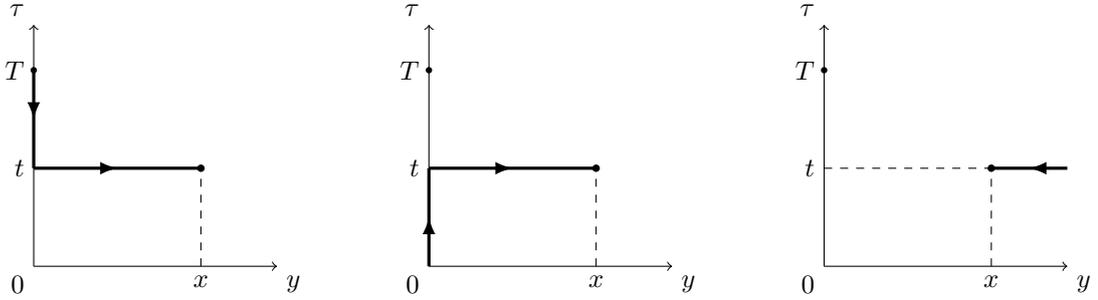

As in \cite{BKS20}, we introduce a new (uniformizing) spectral parameter $\mu$ such that both $\lambda$ and 
$\sqrt{1-\lambda^2}$ are rational w.r.t.~$\mu$:
\begin{equation}\label{la-k-mu}
\lambda=-\frac{1}{2}\left(\mu+\frac{1}{\mu}\right), \qquad
k = \frac{1}{4}\left(\mu-\frac{1}{\mu}\right),
\end{equation}
so that  $\sqrt{1-\lambda^2}=2\ii k$.
 Particularly, $p$ in terms of $\mu$ becomes
\begin{equation}\label{p_mu}
p(x,t,\mu)=\frac{\ii(\mu^2-1)}{2\mu}\left(\frac{1}{2}\int_0^{x} \tilde m(\xi,t)\dd\xi-\frac{1}{2}\int_0^t(\tilde m \tilde \omega)(0,\tau)d\tau-\frac{4\mu^2}{(\mu^2+1)^2}t\right).
\end{equation}

The domains where the exponentials are bounded in the complex $\mu$-plane are separated by the contours 
in Figure \ref{fig:sign}.
Maintaining the same analytic properties of $\Phi_{\infty 1}$ and $\Phi_{\infty 2}$ for all $x\ge 0$, $t\ge 0$ requires that
$(\tilde m \tilde \omega)(0,t)$ keeps the same sign for all $t$.
With this respect, the  cases
\begin{enumerate}
    \item $\tilde \omega(0,t)\leq 0$ for all $0\leq t<T$;
    \item $\tilde \omega(0,t)\geq 0$ for all $0\leq t<T$
\end{enumerate}
are distinguished by the analytic properties of $\Phi_{\infty 1}$ and $\Phi_{\infty 2}$.
In what follows, we consider the case $\tilde \omega(0,t)\leq 0$
(as menstioned above, the case  $\tilde \omega(0,t)\geq 0$ is briefly discussed in Appendix \ref{sec:geq}).

\begin{figure}[ht]
    \centering
\begin{tikzpicture}[scale=1.2]
    % Title
    \node at (0.75,-2) {$\mathrm{sign}\,\mathrm{Im} \frac{\mu(\mu^2 - 1)}{(\mu^2 + 1)^2}$};
    
    % Left circle
    \draw (0,0) circle (1);
    % Right circle (center moved further to reduce intersection)
    \draw (1.5,0) circle (1);

    \node at (1,0.7) {\textcolor{red}{$i$}};
    \fill[red] (0.75, 0.66) circle(0.05);
    
    % Horizontal axis
    \draw[thick] (-1.5,0) -- (3,0);
    
    % Signs in left circle
    \node at (-0.5,0.5) {+};
    \node at (-0.5,-0.5) {$-$};
    
    % Signs in right circle
    \node at (2,0.5) {+};
    \node at (2,-0.5) {$-$};
    
    % Overlap signs (adjusted for new spacing)
    \node at (0.75,0.3) {$-$};
    \node at (0.75,-0.3) {+};

        % outside signs 
    \node at (0.75,1.1) {$-$};
    \node at (0.75,-1.1) {+};

    % Title
\node at (6.5,-2) {$\mathrm{sign}\,\Im\frac{\mu^2 - 1}{\mu}$};
    
    % Horizontal axis
    \draw[thick] (4.5,0) -- (8.5,0);
    
    % Signs
    \node at (6.5,0.5) {+};
    \node at (6.5,-0.5) {$-$};
\end{tikzpicture}
    \caption{Signature tables for $p(x,t,\mu)$  }
    \label{fig:sign}
\end{figure}
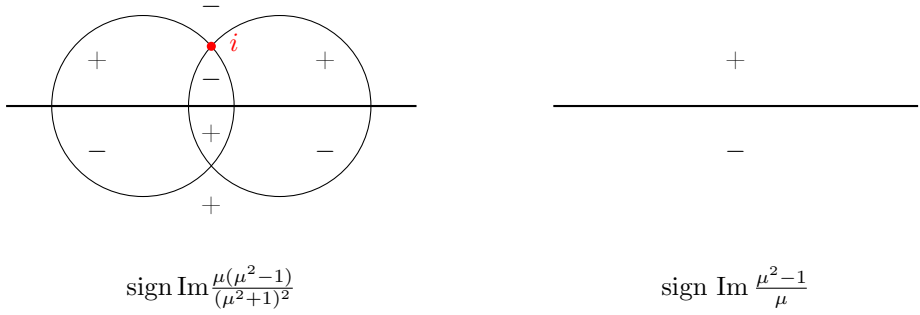

\begin{notations*}
We introduce the following notations for various domains of the complex plane
(see Figure \ref{fig:contour_RH_xt_}): 
\begin{itemize}
    \item $\mathcal{D}_+:=\left(\left(\{|\mu+1| < \sqrt{2}\}\setminus\{|\mu-1| < \sqrt{2}\}\right)\cup\left(\{|\mu-1| < \sqrt{2}\}\setminus\{|\mu+1| < \sqrt{2}\}\right)\right)
    \cap {\mathbb C}^+$;

        \item $\mathcal{C}_-:={\mathbb C}^+ \setminus\mathcal{D}_+$;

    \item $\mathcal{D}_-:=\left(\left(\{|\mu+1| < \sqrt{2}\}\setminus\{|\mu-1| < \sqrt{2}\}\right)\cup\left(\{|\mu-1| < \sqrt{2}\}\setminus\{|\mu+1| < \sqrt{2}\}\right)\right)
    \cap {\mathbb C}^-$;

            \item $\mathcal{C}_+:={\mathbb C}^-\setminus\mathcal{D}_-$.
\end{itemize}
\end{notations*}

Let $A^{(1)}$ and $A^{(2)}$ denote the columns of a $2\times 2$ matrix $A = \bigl( A^{(1)}\ \ A^{(2)} \bigr)$. With this notation, using Neumann series expansions, we obtain the following properties of $\Phi_{\infty i}^{(j)}(x,t,\mu)$ (recall that we consider the case $\tilde \omega(0,t) \leq 0$):

\begin{enumerate}
    \item $\Phi_{\infty 1}^{(1)}(x,t,\mu)$ is analytic in $\mathbb{C}\setminus\{\pm1,\pm\ii\}$ and continuous up to the boundary except at $\mu=\pm 1, \pm\ii$.
 Moreover, $\Phi_{\infty 1}^{(1)}(0,t,\mu)$=$\begin{pmatrix}
        1\\0
    \end{pmatrix}+O(\frac{1}{\mu})$ as $\mu\to\infty$ in $\mathbb{C}^-$.

    \item $\Phi_{\infty 1}^{(2)}(x,t,\mu)$ is analytic in $\mathbb{C}\setminus\{\pm1,\pm\ii\}$ and continuous up to the boundary except at $\mu=\pm 1, \pm\ii$.
 Moreover, $\Phi_{\infty 1}^{(2)}(0,t,\mu)$=$\begin{pmatrix}
        0\\1
    \end{pmatrix}+O(\frac{1}{\mu})$ as $\mu\to\infty$ in $\mathbb{C}^+$.

     \item $\Phi_{\infty 2}^{(1)}(x,t,\mu)$ is analytic in $\mathbb{C}\setminus\{\pm1,\pm\ii\}$ and $\Phi_{\infty 2}^{(1)}(x,0,\mu)$ is analytic in $\mathbb{C}\setminus\{\pm1\}$.
 Moreover, $\Phi_{\infty 2}^{(1)}(x,t,\mu)$=$\begin{pmatrix}
        1\\0
    \end{pmatrix}+O(\frac{1}{\mu})$ as $\mu\to\infty$ in $\mathbb{C}^+$.

    \item $\Phi_{\infty 2}^{(2)}(x,t,\mu)$ is analytic in $\mathbb{C}\setminus\{\pm1,\pm\ii\}$ and $\Phi_{\infty 2}^{(2)}(x,0,\mu)$ is analytic in $\mathbb{C}\setminus\{\pm1\}$.
 Moreover, $\Phi_{\infty 2}^{(2)}(x,t,\mu)$=$\begin{pmatrix}
        0\\1
    \end{pmatrix}+O(\frac{1}{\mu})$ as $\mu\to\infty$ in $\mathbb{C}^-$.

      \item $\Phi_{\infty 3}^{(1)}(x,t,\mu)$ is analytic in $\mathbb{C}^-\setminus\{-\ii\}$ and continuous up to the boundary except at $\mu=\pm 1, -\ii$.
 Moreover, $\Phi_{\infty 3}^{(1)}(x,t,\mu)$=$\begin{pmatrix}
        1\\0
    \end{pmatrix}+O(\frac{1}{\mu})$ as $\mu\to\infty$.

    \item $\Phi_{\infty 3}^{(2)}(x,t,\mu)$ is analytic in $\mathbb{C}^+\setminus\{\ii\}$ and continuous up to the boundary except at $\mu=\pm 1, \ii$.
 Moreover, $\Phi_{\infty 3}^{(2)}(x,t,\mu)$=$\begin{pmatrix}
        0\\1
    \end{pmatrix}+O(\frac{1}{\mu})$ as $\mu\to\infty$.   

    \item In the corresponding half-planes,

    \begin{subequations}\label{Phi_inf_pm1}
        \begin{align}
            \Phi_{\infty j}&=\frac{\ii}{2(\mu-1)}\alpha_j(x,t)
            \begin{pmatrix}
                -1&1\\-1&1
            \end{pmatrix}+O(1),\quad \mu\to 1\\
             \Phi_{\infty j}&=-\frac{\ii}{2(\mu+1)}\alpha_j(x,t)
            \begin{pmatrix}
                1&1\\-1&-1
            \end{pmatrix}+O(1),\quad \mu\to -1.
        \end{align}
    \end{subequations}

\end{enumerate}

Further, observing  that $U_\infty(\mu)\equiv U_\infty(x,t,\mu)$ and  $V_\infty(\mu)\equiv V_\infty(x,t,\mu)$ satisfy the symmetries
\begin{subequations}\label{sym-UV}
\begin{alignat}{3}\label{sym-U}
U_\infty(\bar\mu)&=\sigma_1\overline{U_\infty(\mu)}\sigma_1,&\qquad&U_\infty(-\mu)=\sigma_2U_\infty(\mu)\sigma_2,&\qquad&U_\infty(\mu^{-1})=\sigma_1U_\infty(\mu)\sigma_1,\\
V_\infty(\bar\mu)&=\sigma_1\overline{V_\infty(\mu)}\sigma_1,&&V_\infty(-\mu)=\sigma_2V_\infty(\mu)\sigma_2,&&V_\infty(\mu^{-1})=\sigma_1 V_\infty(\mu)\sigma_1,
\end{alignat}
\end{subequations}
and that  $p(\mu)\equiv p(x,t,\mu)$ satisfies the symmetries
\begin{equation}\label{sym-p}
p^*(\mu)=-p(\mu)=p(-\mu)=p(\mu^{-1}),
\end{equation}
it follows that
$\Phi_{\infty j}$ also satisfy the same symmetries as in \eqref{sym-U}:
\begin{equation}\label{sym-Phi}
\Phi_{\infty j}(\bar\mu)=\sigma_1\overline{\Phi_{\infty j}(\mu)}\sigma_1,\quad\Phi_{\infty j}(-\mu)=\sigma_2\Phi_{\infty j}(\mu)\sigma_2,\quad\Phi_{\infty j}(\mu^{-1})=\sigma_1\Phi_{\infty j}(\mu)\sigma_1.
\end{equation}

Also, since  the coefficients in  \eqref{Lax-1} are traceless matrices, 
it follows that
$\det \Phi_{\infty j}\equiv 1$.

\subsection{Eigenfunctions near $\lambda=0$}

In the case of the Camassa--Holm equation \cite{BS08} as well as other CH-type nonlinear integrable equations studied so far, see, e.g., \cites{BS15}, the analysis of the behavior of the respective Jost solutions at a dedicated point in the complex plane of the spectral parameter (in our case, at $\lambda=0$) requires a dedicated gauge transformation of the Lax pair equations.

It is remarkable that in the case of the mCH equation, in order to control the 
behavior of the eigenfunctions at $\lambda=0$, we don't need to introduce an additional transformation; all we need is to regroup the terms in the Lax pair \eqref{Lax-1}.

Namely, we rewrite \eqref{Lax-Q-form} as follows:

\begin{subequations}\label{Lax-2}
\begin{equation}\label{Lax-2-x}
\tilde\Phi_{0x}+\frac{\ii(\mu^2-1)}{4\mu }\sigma_3\tilde\Phi_0 =  U_0 \tilde\Phi_0,
\end{equation}
where
\begin{equation}\label{U0-hat}
 U_0(x,t,\mu):=\frac{\ii(\mu^2+1)(\tilde m - 1)}{2 (\mu^2-1)}
\begin{pmatrix}
0 & 1 \\
-1 & 0 \\
\end{pmatrix}
-\left(\frac{\ii\mu(\tilde m-1)}{\mu^2-1}+\frac{\ii(\mu^2-1)\tilde m}{4\mu}-\frac{\ii(\mu^2-1)}{4\mu}\right)\sigma_3,
\end{equation}
so that $ U_0(x,t,\pm\ii)\equiv 0$. Accordingly, 
\begin{equation}\label{Lax-2-t}
\tilde\Phi_{0t}-\frac{2\ii(\mu^2-1)\mu}{(\mu^2+1)^2}\sigma_3\tilde\Phi_0=V_0 \tilde\Phi_0,
\end{equation}
where
\begin{equation}\label{hat-V0}
 V_0(x,t,\mu):=\frac{\ii(\mu^2-1)}{4\mu}(\tilde{u}^2-\tilde{u}_x^2+2\tilde{u})\tilde{m}\sigma_3 + V_\infty(x,t,\mu).
\end{equation}
\end{subequations}
Introduce (compare with \eqref{p_mu})
\begin{equation}\label{p_0mu}
p_0(x,t,\mu):=\frac{\ii(\mu^2-1)}{2\mu}\left(\frac{1}{2}x-\frac{4\mu^2}{(\mu^2+1)^2}t\right)
\end{equation}
and $Q_0:= p_0\sigma_3$.
Then introduce the solutions $\tilde \Phi_{0j}(x,t,\mu )$, 
$j=1,2,3$ to equations 
\eqref{Lax-2} similarly to $\tilde \Phi_{\infty j}(x,t,\mu )$:
$\tilde\Phi_{0 j}(x,t,\mu):= \Phi_{0 j}(x,t,\mu )\eul^{-p_0(x.t.\mu)\sigma_3}$, where $\{\Phi_{0 j}(x,t,\mu )\}$ are the solutions of the integral equations:
\begin{equation}\label{inteq_01}
    \begin{aligned}
\Phi_{0 1}(x,t,\mu )=&I+
\int_{0}^{x}
	\eul^{-\frac{i(\mu^2-1)}{4\mu}(x-y)\hat\sigma_3}( U_0\Phi_{01})(y,t,\mu ) \dd y  - \\
&\eul^{-\frac{i(\mu^2-1)}{4\mu}x\hat\sigma_3}
\int_{t}^{T}
	\eul^{-\frac{2i\mu(\mu^2-1)}{(\mu^2+1)^2}(\tau-t)\hat\sigma_3}( V_0\Phi_{01})(0,\tau,\mu ) \dd \tau,
\end{aligned}
\end{equation}

\begin{equation}\label{inteq_02}
  \begin{aligned}
\Phi_{02}(x,t,\mu )=&I+
\int_{0}^{x}
	\eul^{-\frac{i(\mu^2-1)}{4\mu}(x-y)\hat\sigma_3}( U_0\Phi_{02})(y,t,\mu ) \dd y  + \\
&\eul^{-\frac{i(\mu^2-1)}{4\mu}x\hat\sigma_3}
\int_{0}^{t}
	\eul^{\frac{2i\mu(\mu^2-1)}{(\mu^2+1)^2}(t-\tau)\hat\sigma_3}( V_0\Phi_{02})(0,\tau,\mu ) \dd \tau,
\end{aligned}  
\end{equation}    

\begin{equation}\label{inteq_03}
  \begin{aligned}
\Phi_{0 3}(x,t,\mu )=&I-
\int_{x}^{+\infty}
	\eul^{\frac{i(\mu^2-1)}{4\mu}(y-x)\hat\sigma_3}( U_0\Phi_{03})(y,t,\mu ) \dd y.
\end{aligned}  
\end{equation}    

Notice that properties of $\Phi_{0}(x,t,\mu )$ are independent of sign of $\tilde \omega(0,t)$:

\begin{enumerate}
    \item $\Phi_{0 1}^{(1)}(x,t,\mu)$ is analytic in $\mathbb{C}\setminus\{\pm1,\pm\ii\}$ and continuous up to the boundary except at $\mu=\pm 1, \ii$.
 Moreover, $\Phi_{0 1}^{(1)}(x,t,\mu)=\begin{pmatrix}
        1\\0
    \end{pmatrix}+O((\mu\pm\ii))$ as $\mu\to\mp\ii$ in $\mathcal{D}_+\cup\mathcal{C}_+$.

    \item $\Phi_{0 1}^{(2)}(x,t,\mu)$ is analytic in $\mathbb{C}\setminus\{\pm1,\pm\ii\}$ and continuous up to the boundary except at $\mu=\pm 1, -\ii$.
 Moreover, $\Phi_{0 1}^{(2)}(x,t,\mu)=\begin{pmatrix}
        0\\1
    \end{pmatrix}+O((\mu\pm\ii))$ as $\mu\to\mp\ii$ in $\mathcal{D}_-\cup\mathcal{C}_-$.

     \item $\Phi_{0 2}^{(1)}(x,t,\mu)$ is analytic in $\mathbb{C}\setminus\{\pm1,\pm\ii\}$ and $\Phi_{0 2}^{(1)}(x,0,\mu)$ is analytic in $\mathbb{C}\setminus\{\pm1\}$.
 Moreover, $\Phi_{0 2}^{(1)}(x,t,\mu)=\begin{pmatrix}
        1\\0
    \end{pmatrix}+O((\mu\pm\ii))$ as $\mu\to\pm\ii$ in $\mathcal{D}_-\cup\mathcal{C}_-$.

    \item $\Phi_{0 2}^{(2)}(x,t,\mu)$ is analytic in $\mathbb{C}\setminus\{\pm1,\pm\ii\}$ and $\Phi_{0 2}^{(2)}(x,0,\mu)$ is analytic in $\mathbb{C}\setminus\{\pm1\}$.
 Moreover, $\Phi_{0 2}^{(2)}(x,t,\mu)=\begin{pmatrix}
        0\\1
    \end{pmatrix}+O((\mu\pm\ii))$ as  $\mu\to\pm\ii$ in $\mathcal{D}_+\cup\mathcal{C}_+$.

      \item $\Phi_{0 3}^{(1)}(x,t,\mu)$ is analytic in $\mathbb{C}^-$ and continuous up to the boundary except at $\mu=\pm 1$.
 Moreover, $\Phi_{0 3}^{(1)}(x,t,\mu)$=$\begin{pmatrix}
        1\\0
    \end{pmatrix}+O((\mu+\ii))$ as $\mu\to-\ii$ in $\mathbb{C}^-$.

    \item $\Phi_{0 3}^{(2)}(x,t,\mu)$ is analytic in $\mathbb{C}^+$ and continuous up to the boundary except at $\mu=\pm 1$.
 Moreover, $\Phi_{03}^{(2)}(x,t,\mu)=\begin{pmatrix}
        0\\1
    \end{pmatrix}+O((\mu-\ii))$ as $\mu\to\ii$ in $\mathbb{C}^+$.
\end{enumerate}

Since $U_0(x,t,\mu)$, $ V_0(x,t,\mu)$ satisfy the symmetries \eqref{sym-UV} and $p_0(x,t,\mu)$ satisfy the symmetries \eqref{sym-p}, it follows that $\Phi_{0 j}(x,t,\mu)$ satisfy the symmetries \eqref{sym-Phi}. 

\subsection{Spectral functions}

The functions $\Phi_{\infty j}$,
being the solutions of both equations of \eqref{Lax-Q-form}, are related (where they are defined) by matrices independent of $x$ and $t$:

\begin{subequations}\label{rel_inf}
    \begin{align}
        &\Phi_{\infty 3}(x,t,\mu)=\Phi_{\infty 2}(x,t,\mu)\eul^{-p(x,t,\mu)\sigma_3}s(\mu)\eul^{p(x,t,\mu)\sigma_3},\qquad \mu\in {\mathbb C}^+,\\
        &\Phi_{\infty 1}(x,t,\mu)=\Phi_{\infty 2}(x,t,\mu)\eul^{-p(x,t,\mu)\sigma_3}S(\mu)\eul^{p(x,t,\mu)\sigma_3}.
    \end{align}
\end{subequations}
Due to the symmetries of $U_\infty$ and $V_\infty$,  $s(\mu)$ and $S(\mu)$
can be written as 
        \begin{equation}
        \label{s}
s(\mu)=\Phi_{\infty 3}(0,0,\mu)=\begin{pmatrix}
        a^*(\mu) & b(\mu)\\
        b^*( \mu) & a(\mu)
    \end{pmatrix},\qquad \mu\in {\mathbb C}^+
    \end{equation}
and
              \begin{equation}
        \label{S}
S(\mu)=\Phi_{\infty 1}(0,0,\mu)=\begin{pmatrix}
        A^*( \mu) & B(\mu)\\
        B^*( \mu) & A(\mu)
    \end{pmatrix},
    \end{equation}
    with some $a(\mu)$, $b(\mu)$, $A(\mu)$, and  $B(\mu)$.

Similarly, the functions $\Phi_{0 j}$ are related (in the domain where they are defined) by matrices independent of $x$ and $t$:

\begin{subequations}\label{rel_0}
    \begin{align}\label{rel_0-a}
        &\Phi_{0 3}(x,t,\mu)=\Phi_{0 2}(x,t,\mu)\eul^{-p_0(x,t,\mu)\sigma_3}\tilde s(\mu)\eul^{p_0(x,t,\mu)\sigma_3},\qquad \mu\in {\mathbb C}^+,\\ \label{rel_0-b}
        &\Phi_{0 1}(x,t,\mu)=\Phi_{0 2}(x,t,\mu)\eul^{-p_0(x,t,\mu)\sigma_3}\tilde S(\mu)\eul^{p_0(x,t,\mu)\sigma_3}.
    \end{align}
\end{subequations}
In turn,  symmetries of $U_0$ and $V_0$ imply
    \begin{equation}
        \label{tils}
        \tilde s(\mu)=\Phi_{0 3}(0,0,\mu)=\begin{pmatrix}
        \overline{\tilde a(\bar \mu)} & \tilde b(\mu)\\
        \overline{\tilde b(\bar \mu)} & \tilde a(\mu)
    \end{pmatrix}
    \end{equation}
and
          \begin{equation}
        \label{tilS}
\tilde S(\mu)=\Phi_{0 1}(0,0,\mu)=\begin{pmatrix}
        \overline{\tilde A(\bar \mu)} &\tilde  B(\mu)\\
        \overline{\tilde B(\bar \mu)} &\tilde  A(\mu)
    \end{pmatrix},
    \end{equation}
with some $\tilde a(\mu)$, $\tilde b(\mu)$, $\tilde A(\mu)$, and  $\tilde B(\mu)$.    

Since the integral equations for $\Phi_{01}(0,t,\mu)$ and $\Phi_{\infty 1}(0,t,\mu)$ involve only the boundary values $v_j(t)$, it follows that the direct $t$-spectral mapping
 \begin{equation*}
     \{v_j(t)\}_0^2 \longrightarrow \{A(\mu), B(\mu)\}
 \end{equation*}
can be defined via the solution of the integral equation \eqref{inteq_inf1} taken at $x=0$.
Alternatively, 
 the direct $t$-spectral mapping
\begin{equation*}
     \{v_j(t)\}_0^2 \longrightarrow \{\tilde A(\mu),\tilde B(\mu)\}
 \end{equation*}
can be  defined via the solution of the integral equation \eqref{inteq_01} taken at $x=0$.

Similarly, since the integral equations for $\Phi_{03}(x,0,\mu)$ and $\Phi_{\infty 3}(x,0,\mu)$ involve only the
 initial condition $\tilde m_0(x)$, it follows that the direct $x$-spectral mapping
 \begin{equation*}
     \{\tilde m_0(x)\} \longrightarrow \{a(\mu), b(\mu)\}
 \end{equation*}
can be defined via the solution of the integral equation \eqref{inteq_inf3} taken at $t=0$
or it can be defined by  
\begin{equation*}
     \{\tilde m_0(x)\} \longrightarrow \{\tilde a(\mu),\tilde b(\mu)\},
 \end{equation*}
in terms of  the solution of the integral equation \eqref{inteq_03} taken at $t=0$.

Since $\Phi_0$ and $\Phi_\infty$ 
come from the 
 same system of ODEs \ref{Lax}, they are related as follows:

\begin{subequations}\label{Phi_0_inf_rel}
\begin{equation}\label{Phi_0_inf_rel_}
 \Phi_{\infty j}(x,t,\mu)=\Phi_{0 j}(x,t,\mu)\eul^{-p_0(x,t,\mu)\sigma_3}C_j(\mu)\eul^{p(x,t,\mu)\sigma_3}   
\end{equation}
with $C_j(\mu)$ independent of $x$ and $t$ given by:
\begin{align}\label{Phi_0_inf_coeff}
    C_1&=\eul^{\frac{\ii (\mu^2-1) }{4\mu}\int_0^T(\tilde m\tilde \omega)(0,\tau)\dd\tau\sigma_3},\\
    C_2&=I,\\\label{Phi_0_inf_coeff3}
    C_3&=\eul^{-\frac{\ii (\mu^2-1) }{4\mu}\int_0^\infty(\tilde m-1)(\xi,0)\dd\xi\sigma_3}.
\end{align}
\end{subequations}
Particularly, substituting $(x,t)=(0,0)$ in \ref{Phi_0_inf_rel}, we obtain 
    \begin{align}\label{s_via_til_s}
      &s(\mu)=\tilde s(\mu)\eul^{-\frac{\ii (\mu^2-1) }{4\mu}\nu(0)\sigma_3}, \\ \label{S_via_til_S}
      &S(\mu)=\tilde S(\mu)\eul^{-\frac{\ii (\mu^2-1) }{4\mu}\eta(T)\sigma_3}
    \end{align}
and thus
\begin{equation}\label{tilde-a--a}
    \tilde a(\mu)=a(\mu)\eul^{-\frac{\ii (\mu^2-1) }{4\mu}\nu(0)}, \qquad
    \tilde b(\mu)=b(\mu)\eul^{-\frac{\ii (\mu^2-1) }{4\mu}\nu(0)}
\end{equation}
and 
\begin{equation}\label{tilde-A--A}
   \tilde A(\mu)= A(\mu)\eul^{-\frac{\ii (\mu^2-1) }{4\mu}\eta(T)}, \qquad
   \tilde  B(\mu)= B(\mu)\eul^{-\frac{\ii (\mu^2-1) }{4\mu}\eta(T)}.
\end{equation}

\begin{remark}
    An important difference from the original CH equation is that we do not need another pair of Lax equations in order to 
    control solutions $\lambda=0$. Consequently, in the r.h.s. of \eqref{Phi_0_inf_rel} there is
no left factor (cf. \cite{BS08}) and thus the relations between the spectral functions 
associated with $\Phi_{\infty j}$ and $\Phi_{0 j}$ take simple forms 
\eqref{tilde-a--a} and \eqref{tilde-A--A} (without mixing up different spectral functions).
\end{remark}

\subsection{Compatibility of initial and boundary values}

In this section, we discuss the relations among the spectral functions.  

Assuming that $\Phi_{\infty 3}$ is the solution of the Lax pair equations \eqref{Lax-Q-form}, in which
 $\tilde u(x,t)$ satisfies the mCH equation, evaluating it at $x = 0$, $t = T$, and taking into account
 \eqref{rel_inf} gives
    \begin{equation*}
        \Phi_{\infty3}(0,T,\mu)=\eul^{-p(0,T,\mu)\sigma_3}S^{-1}(\mu)s(\mu)\eul^{p(0,T,\mu)\sigma_3}.
    \end{equation*}
In particular, its $(12)$ entry reads as
\begin{equation*}
    \Phi^{(12)}_{\infty3}(0,T,\mu)=(A(\mu)b(\mu)-B(\mu)a(\mu))\eul^{\frac{\ii}{2}\frac{\mu^2-1}{\mu}\int_0^T(\tilde\omega\tilde m)(0,\tau)\dd\tau}\eul^{\frac{4\ii\mu(\mu^2-1)}{(\mu^2+1)^2}T}.
\end{equation*}
Now, recall that the properties of $\Phi_{\infty 3}$ imply
 $ \Phi^{(12)}_{\infty3}=(\frac{1}{\mu})$ as $\mu\to\infty$ in $\mathbb{C}^+$.
Thus, we arrive at the following asymptotic relationship among the spectral functions:

\begin{equation}
\label{relations_Phi_inf3}
       (A(\mu)b(\mu)-B(\mu)a(\mu))\eul^{\frac{\ii}{2}\frac{\mu^2-1}{\mu}\int_0^T(\tilde\omega\tilde m)(0,\tau)\dd\tau}=O(\frac{1}{\mu}),\quad\mu\to\infty, \quad\mu\in\mathbb{C}^+.
    \end{equation}

Similarly, assuming that $\Phi_{0 3}$ is the solution of the Lax pair equations \eqref{Lax-2}, in which
 $\tilde u(x,t)$ satisfies the mCH equation, evaluating it at $x = 0$, $t = T$, and taking into account
 \eqref{rel_0} gives
    \begin{equation*}
        \Phi_{03}(0,T,\mu)=\eul^{-p_0(0,T,\mu)\sigma_3}\tilde S^{-1}(\mu)\tilde s(\mu)\eul^{p_0(0,T,\mu)\sigma_3}.
    \end{equation*}
In particular, its $(12)$ entry reads as
\begin{equation*}
    \Phi^{(12)}_{03}(0,T,\mu)=(\tilde A(\mu)\tilde b(\mu)-\tilde B(\mu)\tilde a(\mu))\eul^{\frac{4\ii\mu(\mu^2-1)}{(\mu^2+1)^2}T}.
\end{equation*}
Recalling that the properties of $\Phi_{0 3}$ imply
 $ \Phi^{(12)}_{03}=O(\mu-\ii)$ as $\mu\to \ii$, we obtain 
 \[
 (\tilde A(\mu)\tilde b(\mu)-\tilde B(\mu)\tilde a(\mu))\eul^{\frac{4\ii\mu(\mu^2-1)}{(\mu^2+1)^2}T}=O(\mu-\ii),\quad\mu\to\ii,
 \]
which due to \eqref{tilde-a--a} and \eqref{tilde-A--A} reads as
\begin{equation}\label{relations_Phi_03_i}
     \eul^{\frac{\eta(T)+\nu(0)}{2}} ( A(\mu)b(\mu)- B(\mu) a(\mu))\eul^{\frac{4\ii\mu(\mu^2-1)}{(\mu^2+1)^2}T}=O(\mu-\ii),\quad\mu\to\ii,
\end{equation}

Asymptotic formulas \eqref{relations_Phi_inf3},   \eqref{relations_Phi_03_i} constitute the so-called
\emph{Global Relations} characterizing the compatibility of initial and boundary values in spectral terms,
namely, in terms of $\{a(\mu), b(\mu)\}$ and $\{A(\mu), B(\mu)\}$.

\section{Spectral Mappings: Direct Problems}\label{sec:3}

We already noticed that restricting the eigenfunctions to the boundary lines $t=0$ and $x=0$ yields the 
scattering problems for respectively the $x$- and $t$-equations of the Lax pair, with 
potentials supported on $x>0$ and $0<t<T$.
In the context of the development of the representation of the 
solution of initial boundary value problem for the mCH equation,
it is important to have a detailed characterization of both
the direct and inverse parts of these scattering problems.
In this subsection, we discuss the direct parts.

\subsection{The direct $x$-spectral problem \texorpdfstring{$\{ \tilde m_0(x) \} \longrightarrow \{ a(\mu),b(\mu) \}$}{m_0(x) -> (a(mu), b(mu))}}

Consider equation \eqref{inteq_inf3} for $t=0$:
\begin{equation}\label{inteq_inf3_t_0}    
\Phi_{\infty 3}(x,0,\mu)=I-
\int_{x}^{+\infty}
	\eul^{\frac{\ii(\mu^2-1)}{4\mu}\int_x^y\tilde m(\xi,0)\dd \xi\hat\sigma_3}( U_\infty\Phi_{\infty3})(y,0,\mu) \dd y
\end{equation}
with $U_\infty$ given via \eqref{U-hat}  with $\tilde m$ replaced by $\tilde m_0(x)$ such that $\tilde m_0(x)>0$ and $\tilde m_0(x)-1\to 0$ as $x\to\infty$ (now we do not require $\tilde m_0(x)$ to be the initial value of a solution of the mCH equation). Then the solution of \eqref{inteq_inf3_t_0} evaluated at $x=0$
and \eqref{s} determine the eigenfunctions and the corresponding spectral functions  ${a}(\mu)$ and $ {b}(\mu)$  associated with $\tilde m_0(x)$.

 Analyzing the Volterra integral equation \eqref{inteq_03_t_0} yields the following properties of $a(\mu)$ and $b(\mu)$:

\begin{enumerate}
    \item $ a(\mu)$ and $ b(\mu)$ are analytic in $\mathbb{C}^+$. Moreover, $a(0)=1$, $b(0)=0$, $ a(\mu)=1+O(\frac{1}{\mu})$ and $ b(\mu)=O(\frac{1}{\mu})$ as $\mu\to\infty$. 

    \item Determinant relation (since the matrices $U_\infty$ and $V_\infty$ are traceless):
        \begin{equation}\label{detrel_ab}
           a(\mu) a^*( \mu)-  b(\mu) b^*( \mu)=1  
        \end{equation}

    \item Symmetries:
    \begin{subequations}\label{sym_a_}
        \begin{align}
    &\overline{a(\bar \mu)}=a(-\mu)=a(\frac{1}{\mu}),\\
    &\overline{b(\bar \mu)}=-b(-\mu)=b(\frac{1}{\mu}).
\end{align}
    \end{subequations}

\item Behavior at $\mu=\pm1$:
\begin{align}
    \label{ab_at_1}
    &a(\mu)=\gamma\frac{\ii}{2(\mu-1)}+O(1),\quad b(\mu)=\gamma\frac{\ii}{2(\mu-1)}+O(1)\text{ as }\mu\to 1\text{ in }\mathbb{C}^+,\\\label{ab_at_-1}
    &a(\mu)=\gamma\frac{\ii}{2(\mu+1)}+O(1),\quad b(\mu)=-\gamma\frac{\ii}{2(\mu+1)}+O(1)\text{ as }\mu\to -1\text{ in }\mathbb{C}^+
\end{align}

\begin{remark}\label{rem:sing}
The case $\gamma \neq 0$ is generic. On the other hand, in the non-generic case $\gamma =0$, we  have $a(\pm 1)=a_1$ and $b(\pm 1)=\pm b_1$ with some $a_1\in\mathbb{R}$ and $b_1\in\mathbb{R}$ such that $a_1^2=1+b_1^2$. It then follows from \eqref{rel_inf} that the coefficients $\alpha_3(x,t)$ and $\alpha_2(x,t)$ appearing in the expansions of $\Phi_{\infty j}$ at $\mu=\pm1$,
see \eqref{Phi_inf_pm1}, are related by
\begin{equation}\label{a1}
\alpha_3(x,t)=(a_1-b_1)\alpha_2(x,t).
\end{equation}
\end{remark}

\item The zeros of $a(\mu)$ are simple and lie on the unit circle (cf. \cite{KST23})

\end{enumerate}

In order to specify the behavior of $ a(\mu)$ and $ b(\mu)$ as $\mu\to\ii$, it is convenient to consider  equation \eqref{inteq_03} for $t=0$:
\begin{equation}\label{inteq_03_t_0}    
\Phi_{0 3}(x,0,\mu)=I-
\int_{x}^{+\infty}
	\eul^{\frac{i(\mu^2-1)}{4\mu}(y-x)\sigma_3}( U_0\Phi_{03})(y,0,\mu) \dd y
\end{equation}
with $U_0$ given via \eqref{U0-hat} with $\tilde m$ replaced by $\tilde m_0(x)$.
Namely, considering the solution of \eqref{inteq_03_t_0} 
at $x=0$ and $\mu=i$ and taking into account \eqref{tils}
and \eqref{tilde-a--a}, it follows that  $\tilde a(i)=1$ and $\tilde b(i)=0$
and thus we have 
\begin{enumerate}[(6)]
    \item Behavior at $\mu=\ii$: 
\begin{equation}\label{beh_a_b_at_i}
    a(i)=\eul^{-\frac{\nu(0)}{2}}, \qquad b(i)=0.
\end{equation}
\end{enumerate}

\subsection{The direct $t$-spectral problem \texorpdfstring{$\{v_j(t) \}_{j=0}^2 \longrightarrow \{{A}(\mu), {B}(\mu) \}$}{vj(x) -> (A(mu), B(mu))}}

Consider equation \eqref{inteq_inf1} for $x=0$
\begin{equation}\label{inteq_inf1_x_0}    
\Phi_{\infty 1}(0,t,\mu)=I-\int_{t}^{T}
	\eul^{\left(-\frac{\ii(\mu^2-1)}{4\mu}\int_t^\tau (\tilde m\tilde \omega)(0,s)\dd s-\frac{2\ii\mu(\mu^2-1)}{(\mu^2+1)^2}(\tau-t)\right)\hat\sigma_3}( V_\infty\Phi_{\infty1})(0,\tau,\mu) \dd \tau
\end{equation}
with $V_\infty$ given by \eqref{hat-V} with $\tilde u$, $\tilde u_x$, $\tilde u_{xx}$ replaced by $v_0(t)$, $v_1(t)$, $v_2(t)$, respectively, such that $v_0(t)-v_2(t)+1>0$ for all $t$.
Evaluating \eqref{inteq_01_x_0} at $t=0$
and \eqref{S} allows determining the eigenfunctions and the corresponding spectral functions $A(\mu)$ and $B(\mu)$  associated with $\{v_j(t)\}_0^2$.

Here we do not require $\{v_j(t)\}_0^2$ to be the boundary values of a solution of the mCH equation, but if $\omega_0(t)=v_0^2(t)-v_1^2(t)+2v_0(t)\equiv 0$ on an interval $[T_1,T_2]$, then we require that $v_{0t}(t)-v_{2t}(t)+2v_1(t)(v_0(t)-v_2(t)+1)^2\equiv 0$ on this interval, which is consistent with the mCH equation with $\tilde \omega\equiv 0$.

The spectral functions  $A(\mu)$ and $ B(\mu)$ satisfy the following properties (notice that 
the properties of $A(\mu)$ and $ B(\mu)$ depend on the sign of $\tilde \omega(0,t)$; recall that we are dealing with the case $\tilde \omega(0,t)\leq0$):

\begin{enumerate}
    \item $ A(\mu)$ and $ B(\mu)$ are analytic in $\mathbb{C}\setminus\{\pm 1,\pm\ii\}$. Moreover, $A(0)=1$, $B(0)=0$, $ A(\mu)=1+O(\frac{1}{\mu})$ and $ B(\mu)=O(\frac{1}{\mu})$ as $\mu\to\infty$ in $\mathbb{C}^+$.

    \item Determinant relation:
        \begin{equation}\label{detrel_AB}
           A(\mu) A^*( \mu)-  B(\mu) B^*( \mu)=1  
        \end{equation}

    \item Symmetries:

    \begin{subequations}\label{sym_A_}
        \begin{align}
    &\overline{A(\bar \mu)}=A(-\mu)=A(\frac{1}{\mu}),\\
    &\overline{B(\bar \mu)}=-B(-\mu)=B(\frac{1}{\mu}).
\end{align}
    \end{subequations}

\item Behavior at $\mu=\pm1$:
\begin{align}
    \label{AB_at_1}
    &A(\mu)=\Gamma\frac{\ii}{2(\mu-1)}+O(1),\quad B(\mu)=\Gamma\frac{\ii}{2(\mu-1)}+O(1)\text{ as }\mu\to 1\text{ in }\mathbb{C}^-,\\\label{AB_at_-1}
    &A(\mu)=\Gamma\frac{\ii}{2(\mu+1)}+O(1),\quad B(\mu)=-\Gamma\frac{\ii}{2(\mu+1)}+O(1)\text{ as }\mu\to -1\text{ in }\mathbb{C}^-.
\end{align}

\begin{remark}\label{rem:sing_G}
The case $\Gamma \neq 0$ is generic. On the other hand, in the non-generic case $\Gamma =0$, we have $A(\pm 1)=A_1$ and $B(\pm 1)=\pm B_1$ with some $A_1\in\mathbb{R}$ and $B_1\in\mathbb{R}$ such that $A_1^2=1+B_1^2$. It then follows from \eqref{rel_inf} that the coefficients $\alpha_1(x,t)$ and $\alpha_2(x,t)$ appearing in the expansions of $\Phi_j$ at $\mu=\pm1$ are related by
\begin{equation}\label{A1}
\alpha_1(x,t)=(A_1-B_1)\alpha_2(x,t).
\end{equation}
\end{remark}

\end{enumerate}

In order to specify the behavior of $ A(\mu)$ and $ B(\mu)$ as $\mu\to\ii$, it is convenient to consider  equation \eqref{inteq_01} for $x=0$:
\begin{equation}\label{inteq_01_x_0}    
\Phi_{0 1}(0,t,\mu)=I-
\int_{t}^{T}
	\eul^{-\frac{2i\mu(\mu^2-1)}{(\mu^2+1)^2}(\tau-t)\hat\sigma_3}( V_0\Phi_{01})(0,\tau,\mu ) \dd \tau.
\end{equation}
with $V_0$ given by \eqref{hat-V0} with $\tilde u$, $\tilde u_x$, $\tilde u_{xx}$ replaced by $v_0(t)$, $v_1(t)$, $v_2(t)$, respectively. 

Analyzing the Volterra integral equation \eqref{inteq_01_x_0}, it follows that $\tilde A(\mu)$ and $\tilde B(\mu)$ are analytic in $\mathcal{C}_-\cup\mathcal{D}_-$. Furthermore, 
\[\tilde A(\mu)=1+O(\mu-\ii),~ \tilde B(\mu)=O(\mu-\ii),\text{ as }\mu\to\ii, ~\mu\in\mathcal{C}_-. \]

Taking into account \eqref{tilde-A--A}, we have

\begin{enumerate}[(5)]
    \item Behavior at $\mu=\ii$:
    \begin{equation}
    \label{beh_A_B_at_i}
    A(\mu)=\eul^{-\frac{1 }{2}\eta(T)}+O(\mu-\ii),~  B(\mu)=O(\mu-\ii),\text{ as }\mu\to\ii, ~\mu\in\mathcal{C}_-.
\end{equation}

\end{enumerate}

\section{Inverse Spectral Mappings}\label{sec:4}

We now describe the inverse spectral mappings. These mappings are expressed in terms of solutions of the associated Riemann–Hilbert problems, whose jump matrices are constructed from the corresponding spectral functions.

\subsection{The Inverse $x$-Spectral Mapping}

 The inverse $x$-spectral mapping
 \[ \{ a(\mu),b(\mu) \}\longrightarrow \{ \tilde m_0(x) \}\]
 is described in terms of the solution of the RH problem, whose jump matrix is determined by $a(\mu)$ and $b(\mu)$. The construction of the RH problem follows from the relation
between the eigenfunctions $\Phi_{\infty2}(x,0,\mu)$ and $\Phi_{\infty3}(x,0,\mu)$ in the direct problem setting.

 Let us define a matrix-valued function $M^{(x)}(x,k)$ with $\det M^{(x)}\equiv1$ as follows:

\begin{equation}\label{M_(x)}
M^{(x)}(x,\mu)=\begin{cases}

\left( \frac{\Phi_{\infty 2}^{(1)}(x,0,\mu)}{a(\mu)},\Phi_{\infty 3}^{(2)}(x,0,\mu)\right),\quad \mu\in\mathbb{C}^+,\\

\left( \Phi_{\infty 3}^{(1)}(x,0,\mu),\frac{\Phi_{\infty 2}^{(2)}(x,0,\mu)}{a^*( \mu)}\right),\quad \mu\in\mathbb{C}^-,\\

\end{cases}
\end{equation}
$\Phi_{\infty2}(x,0,\mu)$ and $\Phi_{\infty3}(x,0,\mu)$ in \eqref{M_(x)} are to be understood as determined by \eqref{inteq_inf2} and \eqref{inteq_inf3} for $t=0$, where $\tilde m(x,0)$ in $U_\infty(x,0)$ is replaced by $\tilde m_0(x)$.

Then the function $M^{(x)}(x,\mu)$ has the following properties (c.f. \cite{BKS20}):
\begin{enumerate}
    \item Jump relation across $\mathbb{R}$
    \begin{subequations}
        \label{jump_M_(x)}
        \begin{equation}
           M_-^{(x)}(x,\mu)=M_+^{(x)}(x,\mu)J(x,\mu),\quad\mu\in\mathbb{R}, 
        \end{equation}
        where 
         \begin{equation}
          J(x,\mu)=\eul^{-p(x,0,\mu)\sigma_3} J_0(\mu)\eul^{p(x,0,\mu)\sigma_3}
        \end{equation} 
with $p(x,0,\mu)= \frac{\ii}{4}\frac{\mu^2-1}{\mu}\int_0^x\tilde m_0(\xi)\dd\xi$ (see \eqref{p_mu}) and        
\begin{equation}\label{jump_M_(x)_0}
   J_0(\mu)=\begin{pmatrix}
       1&-r(\mu)\\r^*(\mu)&1-r(\mu)r^*(\mu)
   \end{pmatrix}, 
\end{equation}
where $r(\mu)=\frac{b(\mu)}{a^*(\mu)}$.
    \end{subequations}

\item $\det M^{(x)}(x,\mu)\equiv1$.

\item Behavior at $\infty$:
\begin{equation}\label{inf_M_(x)}
     M^{(x)}(x,\mu)=I-\frac{1}{\mu}\begin{pmatrix}
    \zeta_1&\eta_2\\
    \eta_1&-\zeta_2
\end{pmatrix}+O(\frac{1}{\mu^2}),\quad\mu\to\infty,
\end{equation}
where, in particular, $\eta_2(x)=1-\frac{1}{\tilde m_0(x)}$.

\item Behavior at $\pm 1$:

\begin{equation}\label{sing_M_(x)}
M^{(x)}(x,\mu)=\begin{cases}
\frac{\ii\alpha_+(x)}{2(\mu-1)}\begin{pmatrix} -c & 1 \\ -c & 1 \end{pmatrix}+\ord(1), &\mu\to 1,\ \ \Im\mu>0,\\
-\frac{\ii\alpha_+(x)}{2(\mu +1)}\begin{pmatrix} c & 1 \\ -c & -1 \end{pmatrix}+\ord(1), &\mu\to -1,\ \ \Im\mu>0,
\end{cases}
\end{equation}
with some $\alpha_+(x,t)\in\mathbb{R}$ and 
\begin{equation}\label{c-1}
c:=\begin{cases}
0,&\text{if }\gamma\neq 0 \text{ (generic case)},\\
\frac{a_1+b_1}{a_1},&\text{if }\gamma=0,
\end{cases}
\end{equation}
where $a_1=a(1)$, $b_1=b(1)$, and $\gamma:= -2\ii\lim\limits_{\mu\to 1}(\mu-1)a(\mu)$. 

\item Symmetry properties:
\begin{equation}\label{sym-M_(x)}
M^{(x)}(\bar\mu)=\sigma_1\overline{M^{(x)}(\mu)}\sigma_1,\qquad M^{(x)}(-\mu)=\sigma_2M^{(x)}(\mu)\sigma_2,\qquad M^{(x)}(\mu^{-1})=\sigma_1M^{(x)}(\mu)\sigma_1.
\end{equation}

\item Residue properties:

\begin{align}\label{res-M+_(x)}
\Res_{\mu_j}M^{(x)(1)}(x,\mu)&=\frac{e^{2p(x,0,\mu_j)}}{\dot a(\mu_j)b(\mu_j)}M^{(x)(2)}(x,\mu_j),\\
\label{res-M-_(x)}
\Res_{\bar\mu_j}M^{(x)(2)}(x,\mu)&=\frac{e^{-2p(x,0,\bar\mu_j)}}{\dot a^*(\bar\mu_j)b^*(\bar\mu_j)}M^{(x)(1)}(x,\bar\mu_j).
\end{align}

\item Behaviour at $\ii$:
\begin{equation}\label{i_beh-M_(x)}
M^{(x)}(x,\mu)= \eul^{-\left(\frac{1}{2}\int_0^x(\tilde m_0(\xi)-1)\dd\xi-\nu(0)\right)\sigma_3}+\frac{\mu-\ii}{2}\begin{pmatrix}
    \gamma_1(x)&a_2(x)\\
    a_3(x)&\gamma_2(x)
\end{pmatrix}+O((\mu-\ii)^2),
\end{equation}
where
\begin{align*}
   &a_2(x)= -\eul^{x}\int_x^\infty(\tilde m_0(\xi)-1)\dd\xi\eul^{\frac{1}{2}\int_0^x(\tilde m_0(\xi)-1)\dd\xi-\nu(0)}.
\end{align*}
Notice that $\eul^{x}\int_x^\infty(\tilde m_0(\xi)-1)\dd\xi=\tilde u_0(x)+\tilde u_{0x}(x)$.

\begin{proof}
    Expanding $\Phi_{03}^{(2)}(x,0,\mu)$ at $\mu=\ii$ using the Neumann series, we obtain
    \[
\Phi_{03}^{(2)}(x,0,\mu)=\begin{pmatrix}
    0\\1
\end{pmatrix}
 -\frac{\mu-\ii}{2} \begin{pmatrix}
    \eul^{x}\int_x^\infty(\tilde m_0(\xi)-1)\dd\xi\\0
\end{pmatrix} +O((\mu-\ii)^2).
    \]
In particular, 
\begin{align*}
    &\tilde b(\mu)=-\frac{\mu-\ii}{2}(\tilde u_0(0)+\tilde u_{0x}(0))+O((\mu-\ii)^2),\\
    &\tilde a(\mu)=1+O((\mu-\ii)^2).
\end{align*}

Likewise, expanding $\Phi_{02}^{(1)}(x,0,\mu)$ at $\mu=\ii$ using the Neumann series gives
    \[
\Phi_{02}^{(1)}(x,0,\mu)=\begin{pmatrix}
    1\\0
\end{pmatrix}
 -\frac{\mu-\ii}{2} \begin{pmatrix}
    0\\\eul^{x}\int_0^x(\tilde m_0(\xi)-1)\dd\xi
\end{pmatrix} +O((\mu-\ii)^2) 
    \]
Then the relation \eqref{Phi_0_inf_rel} between $\Phi_{0j}$ and $\Phi_{\infty j}$, $j=2,3$  and \eqref{s_via_til_s} imply \eqref{i_beh-M_(x)}.
    
\end{proof}
    
\end{enumerate}

The properties of $M^{(x)}$ stated above (that follow from the analysis of the direct problem) can be interpreted as a 
family of Riemann--Hilbert factorization problems parametrized by $(x,t)$.
However, the construction of jump matrix involves 
\[
p(x,0,\mu) = \frac{\ii (\mu^2-1)}{4\mu}\int_0^x\tilde m_0(\xi)\dd\xi
\]
which, in turn, involves $\tilde m_0(x)$.
This suggests the introduction of a new variable
\begin{equation}\label{y_(x)}
y(x)=\int_0^x\tilde m_0(\xi)\dd\xi,
\end{equation}
 which makes the RH problem  explicitly dependent
on $y$ as a parameter:

\textbf{The Riemann--Hilbert problem RH$^{x}$:} Given $a(\mu)$, $b(\mu)$ for $\mu\in\mathbb{C}_+$ and the set $\{\mu_j\}_1^N$ with $|\mu_j|=1$, find a  piece-wise meromorphic $2\times 2$ matrix valued function $\hat M^{(x)}(y,\mu)$ that satisfies the following conditions:

\begin{enumerate}
    \item Jump condition across $\mathbb{R}$
    \begin{subequations}
        \label{jump_hatM_(x)}
        \begin{equation}
           \hat M_-^{(x)}(y,\mu)=\hat M_+^{(x)}(y,\mu)\hat J(y,\mu),\quad\mu\in\mathbb{R} 
        \end{equation}
        where
         \begin{equation}
          \hat J(y,\mu)=\eul^{-\hat p(y,0,\mu)\sigma_3} J_0(\mu)\eul^{\hat p(y,0,\mu)\sigma_3}
        \end{equation} 
        with $\hat p(y,0,\mu)= \frac{\ii}{4}\frac{\mu^2-1}{\mu}y$ and $J_0(\mu)$ defined in \eqref{jump_M_(x)_0}.
    \end{subequations}

\item Behavior at $\infty$:
\begin{equation}\label{inf_hatM_(x)}
     \hat M^{(x)}(y,\mu)=I+O(\frac{1}{\mu}),\quad\mu\to\infty.
\end{equation}

\item Behavior at $\pm 1$:

\begin{equation}\label{sing_hatM_(x)}
\hat M^{(x)}(y,\mu)=\begin{cases}
\frac{\ii\hat \alpha_+(y)}{2(\mu-1)}\begin{pmatrix} -c & 1 \\ -c & 1 \end{pmatrix}+\ord(1), &\mu\to 1,\ \ \Im\mu>0,\\
-\frac{\ii\hat\alpha_+(y)}{2(\mu +1)}\begin{pmatrix} c & 1 \\ -c & -1 \end{pmatrix}+\ord(1), &\mu\to -1,\ \ \Im\mu>0,
\end{cases}
\end{equation}
with some $\hat\alpha_+(y,t)\in\mathbb{R}$, where $c$ is as in \eqref{c-1}.

\item Symmetry conditions:
\begin{equation}\label{sym-hatM_(x)}
\hat M^{(x)}(\bar\mu)=\sigma_1\overline{\hat M^{(x)}(\mu)}\sigma_1,\qquad \hat M^{(x)}(-\mu)=\sigma_2\hat M^{(x)}(\mu)\sigma_2,\qquad \hat M^{(x)}(\mu^{-1})=\sigma_1\hat M^{(x)}(\mu)\sigma_1,
\end{equation}

\item Residue conditions:  for $j=1,\dots,N$,

\begin{align}\label{res-hatM+_(x)}
\Res_{\mu_j}\hat M^{(x)(1)}(y,\mu)&=\frac{e^{2\hat p(y,0,\mu_j)}}{\dot a(\mu_j)b(\mu_j)}\hat M^{(x)(2)}(y,\mu_j),\\
\label{res-hatM-_(x)}
\Res_{\bar\mu_j}\hat M^{(x)(2)}(y,\mu)&=\frac{e^{-2\hat p(y,0,\bar\mu_j)}}{\dot a^*(\bar\mu_j)b^*(\bar\mu_j)}\hat M^{(x)(1)}(y,\bar\mu_j).
\end{align}   
\end{enumerate}

\begin{proposition}
    \begin{enumerate}
        \item $\det \hat M^{(x)}\equiv1$
    
        \item If the solution of the RH problem \eqref{jump_hatM_(x)}--\eqref{res-hatM-_(x)} exists, it is unique.

    \end{enumerate}
\end{proposition}

\begin{proof}
    See \cite{BKS20}.
\end{proof}

The uniqueness of the solution of the Riemann--Hilbert problem \textbf{RH$^{(x)}$} 
and properties of $M^{(x)}(x,\mu)$ justify the following procedure for the inverse mapping
\[
    \{a(\mu),\, b(\mu)\} \longrightarrow \{\tilde m_0(x)\}
\]
for the $x$-problem:

\begin{enumerate}[Step 1.]
    \item Given $a(\mu)$ and $b(\mu)$, construct the Riemann--Hilbert problem \textbf{RH$^{(x)}$};

    \item Solve the constructed RH problem \textbf{RH$^{(x)}$};

    \item Evaluate the solution $ \hat M^{(x)}(y,\mu)$ of this RH problem  at $\mu=\infty$:

    \begin{equation*}
     \hat M^{(x)}(y,\mu)=I-\frac{1}{\mu}\begin{pmatrix}
    \zeta_1(y)&\eta_2(y)\\
    \eta_1(y)&-\zeta_2(y)
\end{pmatrix}+O(\frac{1}{\mu^2}),\quad\mu\to\infty
\end{equation*}
and $\mu=\ii$:
    \begin{equation}
       \hat M^{(x)}(y,\mu)=\begin{pmatrix}
           \hat a(y)&0\\
           0&\hat a^{-1}(y)
       \end{pmatrix}+O(\mu-\ii).
    \end{equation}

\item

Then $\tilde m_0(x)$ is given implicitly as follows
\begin{subequations}
    \begin{align}
\label{m_0_via_RH_x_m_2}
&\hat{\tilde m}_0(y)=\frac{1}{1-\hat\eta_2(y)},\\\label{M_(x)_x_y}
&x(y)=\int_0^y (1-\hat\eta_2(\xi))\dd \xi=y+2 \ln \hat a(y)-\nu(0).
\end{align}
\end{subequations}

\end{enumerate}
(notice that the formula for the change of variable $x=x(y)$ 
can be obtained  from the behavior of $ \hat M^{(x)}(y,\mu)$ 
either at $\mu=\infty$ or at $\mu=i$).

Finally, having $\tilde m_0(x)$, we can reconstruct $\tilde u_0(x)$ by the Green formula:
\begin{equation}\label{Green_m0}
    \begin{aligned}
  \tilde u_0(x)&=\frac{1}{2}\left\{\int_0^x\eul^{y-x}(\tilde m(y,0)-1)dy+\int_x^\infty\eul^{x-y}(\tilde m(y,0)-1)dy\right\}\\
  &+\eul^{x}\left\{\tilde u_0(0)-\frac{1}{2}\int_0^\infty\eul^{-y}(\tilde m(y,0)-1)dy\right\}  
\end{aligned}
\end{equation}

\subsection{The Inverse $t$-Spectral Mapping}

The inverse $t$-spectral mapping
 \[ \{ A(\mu), B(\mu) \}\longrightarrow \{ v_0(t), v_1(t),v_2(t)  \}\]
 is described in terms of the solution of the RH problem, whose jump matrix is determined by $ A$, $ B$.
The construction of the RH problem follows from the relation
between the eigenfunctions $\Phi_{\infty1}(0,t,\mu)$, $\Phi_{\infty2}(0,t,\mu)$, $\Phi_{01}(0,t,\mu)$ and $\Phi_{02}(0,t,\mu)$ in the direct problem setting.

Let $\epsilon>0$ be sufficiently small so that all 
zeros of $A(\mu)$ and $A^*(\mu)$ appearing in the denominators below lie outside of $\{|\mu\pm\ii|\leq\epsilon\}$.

Notice that once $S(\mu)$ is known, $\eta(T)$ can be obtained by evaluating $A(\mu)$ as $\mu\to\ii$ in $\mathcal{C}_-$ (see \eqref{beh_A_B_at_i}). Consequently, 
$A(\mu)$ and $B(\mu)$ determine $\tilde A(\mu)$ and $\tilde B(\mu)$  via \eqref{tilde-A--A}.

Now having determined $\tilde A(\mu)$ and $\tilde B(\mu)$, consider the following matrix-valued function $M^{(t)}(t,\mu)$ with $\det M^{(t)}\equiv1$ defined in the domains separated by contour depicted in  Figure \ref{fig:contour_RH_t}:
\begin{equation}\label{M_(t)}
M^{(t)}(t,\mu)=\begin{cases}

\left( \frac{\Phi_{\infty 2}^{(1)}(0,t,\mu)}{A(\mu)},\Phi_{\infty 1}^{(2)}(0,t,\mu)\right),\quad \mu\in\left(\mathcal{C}_-\cup\mathcal{D}_-\right)\cap\{|\mu\pm\ii| >\epsilon\},\\

\left( \frac{\Psi_{0 2}^{(1)}(0,t,\mu)}{\tilde A(\mu)},\Psi_{0 1}^{(2)}(0,t,\mu)\right),\quad \mu\in\left(\mathcal{C}_-\cup\mathcal{D}_-\right)\cap\{|\mu\pm\ii| <\epsilon\},\\

\left(\Phi_{\infty 1}^{(1)}(0,t,\mu),\frac{\Phi_{\infty 2}^{(2)}(0,t,\mu)}{A^*(\mu)}\right),\quad \mu\in\left(\mathcal{C}_+\cup\mathcal{D}_+\right)\cap\{|\mu\pm\ii| >\epsilon\},\\

\left(\Psi_{0 1}^{(1)}(0,t,\mu),\frac{\Psi_{0 2}^{(2)}(0,t,\mu)}{\tilde A^*(\mu)}\right),\quad \mu\in\left(\mathcal{C}_+\cup\mathcal{D}_+\right)\cap\{|\mu\pm\ii| <\epsilon\}
\end{cases}
\end{equation}
where the functions $\Psi_{0j}$ are defined by
\begin{equation}\label{psi_0}
    \Psi_{0j}(x,t,\mu)\coloneqq\Phi_{0j}(x,t,\mu)\eul^{(p-p_0)(x,t,\mu)\sigma_3}.
\end{equation}

\begin{figure}[h]   
\begin{tikzpicture}[scale=3]
% --- outlines ---
\draw (0,0) circle (1);
\draw (1.5,0) circle (1);
\draw[red] (0.75,0.66) circle (0.3);
\draw[red] (0.75,-0.66) circle (0.3);

\draw[thick, postaction={decorate},
      decoration={markings, mark=at position 0.5 with {\arrow[scale=2]{>}}}]
      (-1.5,0) -- (-1,0);

\draw[thick, postaction={decorate},
      decoration={markings, mark=at position 0.95 with {\arrow[scale=2]{<}}}]
      (-1.5,0) -- (0,0);  

\draw[thick, postaction={decorate},
      decoration={markings, mark=at position 0.9 with {\arrow[scale=2]{>}}}]
      (-1.5,0) -- (1,0);  

\draw[thick, postaction={decorate},
      decoration={markings, mark=at position 0.9 with {\arrow[scale=2]{<}}}]
      (-1.5,0) -- (2,0);  

\draw[thick, postaction={decorate},
      decoration={markings, mark=at position 0.95 with {\arrow[scale=2]{>}}}]
      (-1.5,0) -- (3,0);        

\draw[thick, postaction={decorate},
      decoration={markings, mark=at position 0.25 with {\arrow[scale=2]{<}}}] 
      (0,0) circle (1);

\draw[thick, postaction={decorate},
      decoration={markings, mark=at position 0.75 with {\arrow[scale=2]{>}}}] 
      (0,0) circle (1);

\draw[thick, postaction={decorate},
      decoration={markings, mark=at position 0.095 with {\arrow[scale=2]{<}}}] 
      (0,0) circle (1);

\draw[thick, postaction={decorate},
      decoration={markings, mark=at position 0.14 with {\arrow[scale=2]{>}}}] 
      (0,0) circle (1);    

\draw[thick, postaction={decorate},
      decoration={markings, mark=at position -0.09 with {\arrow[scale=2]{>}}}] 
      (0,0) circle (1);

\draw[thick, postaction={decorate},
      decoration={markings, mark=at position -0.135 with {\arrow[scale=2]{<}}}] 
      (0,0) circle (1);

\draw[thick, postaction={decorate},
      decoration={markings, mark=at position 0.04 with {\arrow[scale=2]{>}}}] 
      (0,0) circle (1);

\draw[thick, postaction={decorate},
      decoration={markings, mark=at position -0.04 with {\arrow[scale=2]{<}}}] 
      (0,0) circle (1);  

      \draw[thick, postaction={decorate},
      decoration={markings, mark=at position 0.25 with {\arrow[scale=2]{<}}}] 
      (1.5,0) circle (1);

\draw[thick, postaction={decorate},
      decoration={markings, mark=at position 0.75 with {\arrow[scale=2]{>}}}] 
      (1.5,0) circle (1);

\draw[thick, postaction={decorate},
      decoration={markings, mark=at position 0.41 with {\arrow[scale=2]{<}}}] 
      (1.5,0) circle (1);

\draw[thick, postaction={decorate},
      decoration={markings, mark=at position 0.64 with {\arrow[scale=2]{<}}}] 
      (1.5,0) circle (1);    

\draw[thick, postaction={decorate},
      decoration={markings, mark=at position -0.4 with {\arrow[scale=2]{>}}}] 
      (1.5,0) circle (1);

\draw[thick, postaction={decorate},
      decoration={markings, mark=at position -0.63 with {\arrow[scale=2]{>}}}] 
     (1.5,0) circle (1);

\draw[thick, postaction={decorate},
      decoration={markings, mark=at position 0.55 with {\arrow[scale=2]{<}}}] 
      (1.5,0) circle (1);

\draw[thick, postaction={decorate},
      decoration={markings, mark=at position -0.53 with {\arrow[scale=2]{>}}}] 
      (1.5,0) circle (1);

\draw[thick,red, postaction={decorate},
      decoration={markings, mark=at position 0.25 with {\arrow[scale=2,red]{<}}}] 
      (0.75,0.66) circle (0.3);      
      
\draw[thick,red, postaction={decorate},
      decoration={markings, mark=at position 0.5 with {\arrow[scale=2,red]{>}}}] 
      (0.75,0.66) circle (0.3); 

\draw[thick,red, postaction={decorate},
      decoration={markings, mark=at position 0.75 with {\arrow[scale=2,red]{<}}}] 
      (0.75,0.66) circle (0.3);  

\draw[thick,red, postaction={decorate},
      decoration={markings, mark=at position 1 with {\arrow[scale=2,red]{>}}}] 
      (0.75,0.66) circle (0.3);

\draw[thick,red, postaction={decorate},
      decoration={markings, mark=at position 0.25 with {\arrow[scale=2,red]{>}}}] 
      (0.75,-0.66) circle (0.3);      
      
\draw[thick,red, postaction={decorate},
      decoration={markings, mark=at position 0.5 with {\arrow[scale=2,red]{<}}}] 
      (0.75,-0.66) circle (0.3); 

\draw[thick,red, postaction={decorate},
      decoration={markings, mark=at position 0.75 with {\arrow[scale=2,red]{>}}}] 
      (0.75,-0.66) circle (0.3);  

\draw[thick,red, postaction={decorate},
      decoration={markings, mark=at position 1 with {\arrow[scale=2,red]{<}}}] 
      (0.75,-0.66) circle (0.3);       
% axis
\draw[thick] (-1.5,0) -- (1.5,0);

% marks
\node at (0.85,0.68) {\textcolor{red}{$i$}};
\fill[red] (0.75,0.66) circle(0.02);

\node at (0,0.2) {\textcolor{blue}{$-1$}};
\fill[blue] (0,0) circle(0.02);

\node at (1.5,0.2) {\textcolor{blue}{$1$}};
\fill[blue] (1.5,0) circle(0.02);

% Signs
\node at (-0.5,0.3) {$\mathcal{D}_+$};
\node at (-0.5,-0.3) {$\mathcal{D}_-$};
\node at (2,0.3) {$\mathcal{D}_+$};
\node at (2,-0.3) {$\mathcal{D}_-$};
\node at (0.75,0.2) {$\mathcal{C}_-$};
\node at (0.75,-0.2) {$\mathcal{C}_+$};
\node at (-1,1.1) {$\mathcal{C}_-$};
\node at (-1,-1.1) {$\mathcal{C}_+$};
\end{tikzpicture}
\caption{The oriented contour $\check \Sigma$ for the  RH problem \textbf{RH$^{(t)}$}}
    \label{fig:contour_RH_t}
\end{figure}

\begin{remark}
    The contour depicted in Figure \ref{fig:contour_RH_t} is symmetric under the mappings $\mu\mapsto\frac{1}{\mu}$, $\mu\mapsto\bar\mu$ and $\mu\mapsto -\mu$.
\end{remark}

\begin{notations*}Let us introduce the following notations

\begin{align}\label{sigma}
  &\Sigma\coloneqq\mathbb{R}\cup\{|\mu-1|=\sqrt{2}\}\cup\{|\mu+1|=\sqrt{2}\},\\\label{check_sigma} 
  &\check\Sigma\coloneqq \Sigma\cup\{|\mu-\ii|=\epsilon\}\cup\{|\mu+\ii|=\epsilon\}.
\end{align}
    
\end{notations*}

Then the function $M^{(t)}(t,\mu)$ has the following properties:

\begin{enumerate}
    \item Jump relation across $\check\Sigma$
    \begin{subequations}
        \label{jump_M_(t)}
        \begin{equation}
           M_-^{(t)}(t,\mu)=M_+^{(t)}(t,\mu)J^{(t)}(t,\mu),\quad\mu\in\check\Sigma 
        \end{equation}
        where
         \begin{equation}
          J^{(t)}(t,\mu)=\eul^{-p(0,t,\mu)\sigma_3} J^{(t)}_0(\mu)\eul^{p(0,t,\mu)\sigma_3}
        \end{equation} 
        with
\[p(0,t,\mu)= \frac{\ii}{2}\frac{\mu^2-1}{\mu}\left(\frac{1}{2}\eta(t)-\frac{4\mu^2}{(\mu^2+1)^2}t\right),\]
 where $\eta(t)=-\int_0^t(v_0(\tau)-v_2(\tau)+1)(v_0^2(\tau)-v_1^2(\tau)+2v_0 (\tau))\dd\tau$, and       
\begin{equation}\label{jump_M_(t)_0}
   J^{(t)}_0(\mu)=\begin{cases}
       \begin{pmatrix}
          1&-R(\mu)\\
          R^*(\mu)&1-R(\mu)R^*(\mu)
       \end{pmatrix},\quad \mu\in\Sigma\cap\{|\mu\pm\ii|>\epsilon\},\\
       \begin{pmatrix}
          1- R(\mu) R^*(\mu)& R(\mu)\eul^{-\frac{\ii (\mu^2-1) }{2\mu}\eta(T)}\\
          - R^*(\mu)\eul^{\frac{\ii (\mu^2-1) }{2\mu}\eta(T)}&1
       \end{pmatrix},\quad \mu\in\Sigma\cap\{|\mu\pm\ii|<\epsilon\},\\
       
       \begin{pmatrix}
         \eul^{\frac{\ii (\mu^2-1) }{4\mu}\eta(T)}&0\\
          0&\eul^{-\frac{\ii (\mu^2-1) }{4\mu}\eta(T)}
       \end{pmatrix},\quad \mu\in(\mathcal{D}_-\cup\mathcal{C}_-)\cap\{|\mu\pm\ii|=\epsilon\},\\

              \begin{pmatrix}
          \eul^{-\frac{\ii (\mu^2-1) }{4\mu}\eta(T)}&0\\
         0&\eul^{\frac{\ii (\mu^2-1) }{4\mu}\eta(T)}
       \end{pmatrix},\quad \mu\in(\mathcal{D}_+\cup\mathcal{C}_+)\cap\{|\mu\pm\ii|=\epsilon\},       
   \end{cases}
\end{equation}
and $R(\mu):=\frac{B(\mu)}{A^*(\mu)}$.
    \end{subequations}

\item Behavior at $\infty$:
\begin{equation}\label{inf_M_(t)}
     M^{(t)}(t,\mu)=I-\frac{1}{\mu}\begin{pmatrix}
    \zeta_1(t)&\eta_2(t)\\
    \eta_1(t)&-\zeta_2(t)
\end{pmatrix}+O(\frac{1}{\mu^2}),\quad\mu\to\infty,
\end{equation}
where, in particular, $\eta_1(t)=\eta_2(t)=1-\frac{1}{v_0(t)-v_2(t)+1}$.

\item $\det M^{(t)}(t,\mu)\equiv1$

\item Behavior at $\pm 1$:

\begin{equation}\label{sing_M_(t)}
M^{(t)}(t,\mu)=\begin{cases}
\frac{\ii\beta_+(t)}{2(\mu-1)}\begin{pmatrix} -1 & C \\ -1 & C \end{pmatrix}+\ord(1), &\mu\to 1,\ \ \Im\mu>0,\\
-\frac{\ii\beta_+(t)}{2(\mu +1)}\begin{pmatrix} 1 & C \\ -1 & -C \end{pmatrix}+\ord(1), &\mu\to -1,\ \ \Im\mu>0,
\end{cases}
\end{equation}
with some $\beta_+(x,t)\in\mathbb{R}$ and 
\begin{equation}\label{C-1}
C:=\begin{cases}
0,&\text{if }\Gamma\neq 0 \text{ (generic case)},\\
\frac{ A_1+ B_1}{ A_1},&\text{if }\Gamma=0,
\end{cases}
\end{equation}
where $A_1=A(1)$, $B_1=B(1)$, and $\Gamma:= -2\ii\lim\limits_{\mu\to 1}(\mu-1)A(\mu)$. 

\item Symmetry properties:
\begin{equation}\label{sym-M_(t)}
M^{(t)}(\bar\mu)=\sigma_1\overline{M^{(t)}(\mu)}\sigma_1,\qquad M^{(t)}(-\mu)=\sigma_2M^{(t)}(\mu)\sigma_2,\qquad M^{(t)}(\mu^{-1})=\sigma_1M^{(t)}(\mu)\sigma_1.
\end{equation}

\item Residue properties. 
Let $\{\nu_j\}$ be zeros of $A(\mu)$  in $\mathcal{C}_-\cup\mathcal{D}_-$.
Assume that they are all simple. Then

\begin{align}\label{res-M+_(t)}
\Res_{\nu_j}M^{(t)(1)}(t,\mu)&=\frac{e^{2p(0,t,\nu_j)}}{\dot A(\nu_j)B(\nu_j)}M^{(t)(2)}(t,\nu_j),\\
\label{res-M-_(t)}
\Res_{\bar\nu_j}M^{(t)(2)}(t,\mu)&=\frac{e^{-2p(0,t,\bar\nu_j)}}{\dot A^*(\bar\nu_j)B^*(\bar\nu_j)}M^{(t)(1)}(t,\bar\nu_j).
\end{align}

\item Behavior at $\ii$. Expanding $\Phi_{02}^{(1)}(0,t,\mu)$ and $\Phi_{01}^{(2)}(0,t,\mu)$ at $\mu = i$ 
via the Neumann series, we obtain
\begin{equation}\label{i_beh-M_(t)}
M^{(t)}(t,\mu)=\left(I+M_1^{(t)}(t)(\mu-\ii)+O((\mu-\ii)^2)\right) \eul^{\frac{\ii(\mu^2-1)}{4\mu}\eta(t)\sigma_3}, 
\end{equation}
where
\begin{equation*}
    M_1^{(t)}=\begin{pmatrix}
    \gamma_1(t)&-\frac{1}{2}(v_0(t)+v_1(t))\\
    -\frac{1}{2}(v_0(t)-v_1(t))&\gamma_1(t)
\end{pmatrix}
\end{equation*}
Notice that multiplication by $\eul^{\frac{\ii(\mu^2-1)}{4\mu}\eta(t)\sigma_3}$ does not change elements $(12)$ and $(21)$ of $M_1^{(t)}$.

\end{enumerate}

As in the $x$-spectral mapping, the properties of $M^{(t)}(t,\mu)$ stated above 
can be seen as those specifying a 
family of Riemann--Hilbert factorization problems.
However, we again face the same problem: the construction of jump matrix involves $p(0,t,\mu)$
which in turn involves 
\[
\eta(t)=-\int_0^t(v_0(\tau)-v_2(\tau)+1)(v_0^2(\tau)-v_1^2(\tau)+2v_0 (\tau))\dd\tau. 
\]
Consequently, this suggests the introduction of a new parameter for RH problem,
$z$, in terms of which $\hat J^{(t)}$ can be explicitly written: \[\hat p(z,t,\mu)=\frac{\ii}{2}\frac{\mu^2-1}{\mu}\left(\frac{1}{2}z-\frac{4\mu^2}{(\mu^2+1)^2}t\right),\] so that,
substituting 
\begin{equation}\label{y_(t)}
z = z(t)=\eta(t),
\end{equation}
we have
$p(0,t,k)=\hat p(\eta(t),t,k)$.

%Recall that $\frac{\dd z}{d t}=-(v_0(t)-v_2(t)+1)(v_0^2(t)-v_1^2(t)+2v_0 (t))$.

Consider the following family of RH problems parametrized by $z$ and $t$:

\textbf{The Riemann--Hilbert problem RH$^{(t)}$:} Given $A(\mu)$, $B(\mu)$ for $\mu\in\mathcal{C}_-\cup\mathcal{D}_-$, and a set $\{\nu_j\}_1^K\subset (\mathcal{C}_-\cup\mathcal{D}_-)$, find a piece-wise meromorphic $2\times 2$  matrix valued function $\hat M^{(t)}(z,t,\mu)$ that satisfies the following conditions:

\begin{enumerate}
    \item Jump relation across $\check\Sigma$
    \begin{subequations}
        \label{jump_hatM_(t)}
        \begin{equation}
           \hat M_-^{(t)}(z,t,\mu)=\hat M_+^{(t)}(z,t,\mu)\hat J^{(t)}(z,t,\mu),\quad\mu\in\check\Sigma,
        \end{equation}
        where
         \begin{equation}
          \hat J^{(t)}(z,t,\mu)=\eul^{-\frac{\ii}{4}\frac{\mu^2-1}{\mu}\left(\frac{1}{2}z-\frac{4\mu^2}{(\mu^2+1)^2}t\right)\sigma_3}J^{(t)}_0(\mu)\eul^{\frac{\ii}{4}\frac{\mu^2-1}{\mu}\left(\frac{1}{2}z-\frac{4\mu^2}{(\mu^2+1)^2}t\right)\sigma_3}
        \end{equation} 
        with $J^{(t)}_0(\mu)$ defined in \eqref{jump_M_(t)_0}.
        \end{subequations}

\item Behavior at $\infty$:
\begin{equation}\label{inf_hatM_(t)}
     \hat M^{(t)}(z,t,\mu)=I+O(\frac{1}{\mu}),\quad\mu\to\infty.
\end{equation}

\item Behavior at $\pm 1$:

\begin{equation}\label{sing_hatM_(t)}
\hat M^{(t)}(z,t,\mu)=\begin{cases}
\frac{\ii\beta_+(z,t)}{2(\mu-1)}\begin{pmatrix} -1 & C \\ -1 & C \end{pmatrix}+\ord(1), &\mu\to 1,\ \ \Im\mu>0 ,\\
-\frac{\ii\beta_+(z,t)}{2(\mu +1)}\begin{pmatrix} 1 & C \\ -1 & -C \end{pmatrix}+\ord(1), &\mu\to -1,\ \ \Im\mu>0,
\end{cases}
\end{equation}
with some $\beta_+(z,t)\in\mathbb{R}$, where $C$ is as in \eqref{C-1}.

\item Symmetry properties:
\begin{equation}\label{sym-hatM_(t)}
\hat M^{(t)}(\bar\mu)=\sigma_1\overline{\hat M^{(t)}(\mu)}\sigma_1,\qquad \hat M^{(t)}(-\mu)=\sigma_2\hat M^{(t)}(\mu)\sigma_2,\qquad \hat M^{(t)}(\mu^{-1})=\sigma_1\hat M^{(t)}(\mu)\sigma_1,
\end{equation}

\item Residue conditions: for $j=1,\dots,K$,

\begin{align}\label{res-hatM+_(t)}
\Res_{\nu_j}\hat M^{(t)(1)}(z,t,\mu)&=\frac{e^{\ii\frac{\nu_j^2-1}{2\nu_j}\left(\frac{1}{2}z-\frac{4\nu_j^2}{(\nu_j^2+1)^2}t\right)}}{\dot A(\nu_j)B(\nu_j)}\hat M^{(t)(2)}(z,t,\nu_j),\\
\label{res-hatM-_(t)}
\Res_{\bar\nu_j}\hat M^{(t)(2)}(z,t,\mu)&=\frac{e^{-\ii\frac{\bar\nu_j^2-1}{2\bar\nu_j}\left(\frac{1}{2}z-\frac{4\bar\nu_j^2}{(\bar\nu_j^2+1)^2}t\right)}}{\dot A^*(\bar\nu_j)B^*(\bar\nu_j)}\hat M^{(t)(1)}(z,t,\bar\nu_j).
\end{align}

\end{enumerate}

\begin{proposition}
    \begin{enumerate}
        \item $\det \hat M^{(t)}\equiv1$
    
        \item If the solution of the RH problem \eqref{jump_hatM_(t)}--\eqref{res-hatM-_(t)} exists, it is unique.
 
    \end{enumerate}
\end{proposition}

\begin{proof}
    see \cite{BKS20}.
\end{proof}

The uniqueness of the solution of the Riemann--Hilbert problem \textbf{RH$^{(t)}$} justifies the following procedure for the inverse mapping
\[
    \{A(\mu),\, B(\mu) \} \longrightarrow \{v_0(t), v_1(t), v_2(t)\}
\]
for the $t$-problem:

\begin{enumerate}[Step 1.]
    \item Given $\tilde A(\mu)$, $\tilde B(\mu)$ construct the Riemann--Hilbert problem  \textbf{RH$^{(t)}$};

    \item Solve the constructed RH problem \textbf{RH$^{(t)}$};

    \item Evaluate the solution $\hat M^{(t)}(z,t,\mu)$ of this RH problem  at $\mu=\ii$:

    \begin{equation}
       \hat M^{(t)}(z,t,\mu)=\begin{pmatrix}
           \hat a_1(z,t)&0\\
           0&\hat a_1^{-1}(z,t)
       \end{pmatrix}+(\mu-\ii)\begin{pmatrix}
           0&\hat a_2(z,t)\\\hat a_3(z,t)&0
       \end{pmatrix}+O((\mu-\ii)^2)
    \end{equation}
and at $\mu=\infty$:

    \begin{equation}
        M^{(t)}(z,t,\mu)=I-\frac{1}{\mu}\begin{pmatrix}
    \hat \zeta_1(z,t)&\hat \eta_2(z,t)\\
    \hat \eta_1(z,t)&-\hat \zeta_2(z,t)
\end{pmatrix}+O(\frac{1}{\mu^2}).
    \end{equation}

    \item Define $\hat v_j(z,t)$ from this expansions in the following way (cf. \eqref{inf_M_(t)} and \eqref{i_beh-M_(t)}):
    \begin{align}\label{hat_v_j_via_RH_t}
     \hat v_0(z,t)&=-(\hat a_2(z,t)\hat a_1(z,t)+\hat a_3(z,t)\hat a_1^{-1}(z,t)), \\
     \hat v_1(z,t)&= -(\hat a_2(z,t)\hat a_1(z,t)-\hat a_3(z,t)\hat a_1^{-1}(z,t)),\\
     \hat v_2(z,t)&=-(\hat a_2(z,t)\hat a_1(z,t)+\hat a_3(z,t)\hat a_1^{-1}(z,t))-\frac{1}{1-\hat \eta_2(z,t)}+1.
    \end{align}

\begin{remark}
       Notice that $\hat{\tilde m}_0(z,t):=\hat v_0(z,t)-\hat v_2(z,t)+1= \frac{1}{1-\hat \eta_2(z,t)}$.
\end{remark}

    Define $z(t)$ as the solution of the differential equation (cf. \eqref{y_(t)}):

    \begin{align}\label{y_(t)_via_RH_t}
     &\frac{\dd z}{\dd t}=-\frac{1}{1-\hat \eta(z,t)}(4\hat a_2(z,t)\hat a_3(z,t)-2\hat a_2(z,t)\hat a_1(z,t)-2\hat a_3(z,t)\hat a_1^{-1}(z,t)),\\
     &z(0)=0.
    \end{align}
Then
\begin{equation}\label{v_j_via_RH_t}
     v_j(t)=
     \hat v_j(z(t),t), \quad j=0,1,2.
    \end{equation}

\end{enumerate}

\section{The main Riemann--Hilbert problem}\label{sec:5}

\subsection{The pre-Riemann--Hilbert problem} Assuming that  there exists a solution $\tilde u(x,t)$ of the IBV problem for \eqref{mCH-2}--\eqref{ic} such that $\tilde \omega(0,t)\le 0$, introduce 
\begin{align*}
    d(\mu)=a(\mu)A^*(\mu)-b(\mu)B^*(\mu),
\end{align*}
and consider the following piecewise meromorphic matrix-valued function $M(x,t,\mu)$ (depending on $(x,t)$ as parameters, for all $x\geq0$ and $t\in[0,T]$) defined in the domains separated by contour depicted 
in Figure \ref{fig:contour_RH_xt_}:

\begin{equation}\label{M_(xt)}
M^{(xt)}(x,t,\mu)=\begin{cases}

\left( \frac{\Phi_{\infty 2}^{(1)}(x,t,\mu)}{a(\mu)},\Phi_{\infty 3}^{(2)}(x,t,\mu)\right),\quad \mu\in \mathcal{C}_-,\\

\left( \frac{\Phi_{\infty 1}^{(1)}(x,t,\mu)}{d(\mu)},\Phi_{\infty 3}^{(2)}(x,t,\mu)\right),\quad \mu\in\mathcal{D}_+,\\

\left(\Phi_{\infty 3}^{(1)}(x,t,\mu),\frac{\Phi_{\infty 2}^{(2)}(x,t,\mu)}{a^*(\mu)}\right),\quad \mu\in\mathcal{C}_+,\\

\left(\Phi_{\infty 3}^{(1)}(x,t,\mu),\frac{\Phi_{\infty 1}^{(2)}(x,t,\mu)}{d^*(\mu)}\right),\quad \mu\in\mathcal{D}_-.
\end{cases}
\end{equation}
Notice that $d$ is chosen in such a way that $\det M^{(xt)}\equiv 1$. 

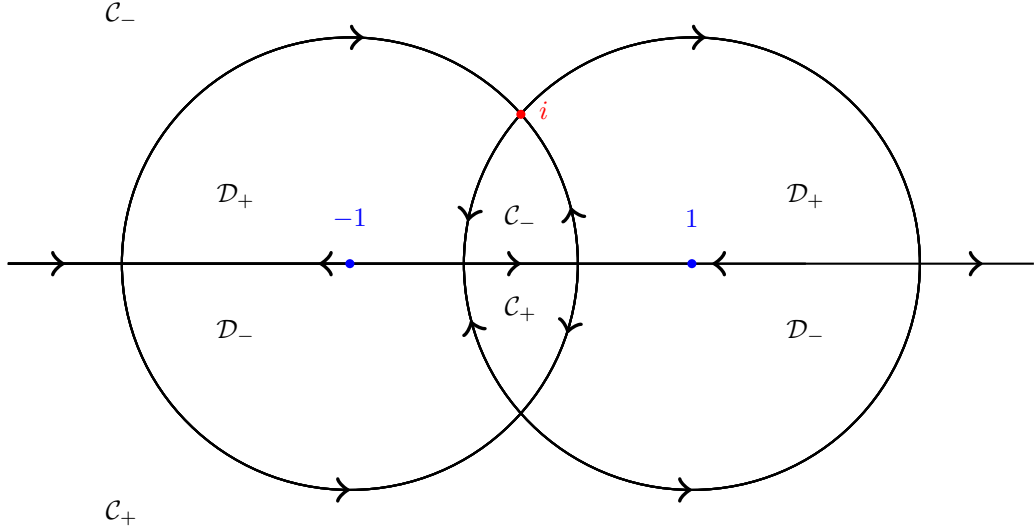
\begin{figure}[h]   
\begin{tikzpicture}[scale=3]
% --- outlines ---
\draw (0,0) circle (1);
\draw (1.5,0) circle (1);

\draw[thick, postaction={decorate},
      decoration={markings, mark=at position 0.5 with {\arrow[scale=2]{>}}}]
      (-1.5,0) -- (-1,0);

\draw[thick, postaction={decorate},
      decoration={markings, mark=at position 0.95 with {\arrow[scale=2]{<}}}]
      (-1.5,0) -- (0,0);  

\draw[thick, postaction={decorate},
      decoration={markings, mark=at position 0.9 with {\arrow[scale=2]{>}}}]
      (-1.5,0) -- (1,0);  

\draw[thick, postaction={decorate},
      decoration={markings, mark=at position 0.9 with {\arrow[scale=2]{<}}}]
      (-1.5,0) -- (2,0);  

\draw[thick, postaction={decorate},
      decoration={markings, mark=at position 0.95 with {\arrow[scale=2]{>}}}]
      (-1.5,0) -- (3,0);        

\draw[thick, postaction={decorate},
      decoration={markings, mark=at position 0.25 with {\arrow[scale=2]{<}}}] 
      (0,0) circle (1);

\draw[thick, postaction={decorate},
      decoration={markings, mark=at position 0.75 with {\arrow[scale=2]{>}}}] 
      (0,0) circle (1);

\draw[thick, postaction={decorate},
      decoration={markings, mark=at position 0.04 with {\arrow[scale=2]{>}}}] 
      (0,0) circle (1);

\draw[thick, postaction={decorate},
      decoration={markings, mark=at position -0.04 with {\arrow[scale=2]{<}}}] 
      (0,0) circle (1);  

      \draw[thick, postaction={decorate},
      decoration={markings, mark=at position 0.25 with {\arrow[scale=2]{<}}}] 
      (1.5,0) circle (1);

\draw[thick, postaction={decorate},
      decoration={markings, mark=at position 0.75 with {\arrow[scale=2]{>}}}] 
      (1.5,0) circle (1);

\draw[thick, postaction={decorate},
      decoration={markings, mark=at position 0.55 with {\arrow[scale=2]{<}}}] 
      (1.5,0) circle (1);

\draw[thick, postaction={decorate},
      decoration={markings, mark=at position -0.53 with {\arrow[scale=2]{>}}}] 
      (1.5,0) circle (1);

% axis
\draw[thick] (-1.5,0) -- (1.5,0);

% marks
\node at (0.85,0.68) {\textcolor{red}{$i$}};
\fill[red] (0.75,0.66) circle(0.02);

\node at (0,0.2) {\textcolor{blue}{$-1$}};
\fill[blue] (0,0) circle(0.02);

\node at (1.5,0.2) {\textcolor{blue}{$1$}};
\fill[blue] (1.5,0) circle(0.02);

% Signs
\node at (-0.5,0.3) {$\mathcal{D}_+$};
\node at (-0.5,-0.3) {$\mathcal{D}_-$};
\node at (2,0.3) {$\mathcal{D}_+$};
\node at (2,-0.3) {$\mathcal{D}_-$};
\node at (0.75,0.2) {$\mathcal{C}_-$};
\node at (0.75,-0.2) {$\mathcal{C}_+$};
\node at (-1,1.1) {$\mathcal{C}_-$};
\node at (-1,-1.1) {$\mathcal{C}_+$};
\end{tikzpicture}
\caption{The oriented contour $ \Sigma$ for the RH problem \textbf{RH$^{(xt)}$}}
    \label{fig:contour_RH_xt_}
\end{figure}

Then the function $M^{(xt)}(x,t,\mu)$ satisfy the following properties:

\begin{enumerate}
    \item Jump relation across $\Sigma$:
    \begin{subequations}
        \label{jump_M_(xt)}
        \begin{equation}
           M_-^{(xt)}(x,t,\mu)=M_+^{(xt)}(x,t,\mu)J^{(xt)}(x,t,\mu),\quad\mu\in\Sigma,
        \end{equation}
        where
         \begin{equation}
          J^{(xt)}(x,t,\mu)=\eul^{-p(x,t,\mu)\sigma_3} J^{(xt)}_0(\mu)\eul^{p(x,t,\mu)\sigma_3}
        \end{equation} 
        with
        
\begin{equation}\label{jump_M_(xt)_0}
   J^{(xt)}_0(\mu)=\begin{cases}
       \begin{pmatrix}
          1&-r(\mu)\\
          r^*(\mu)&1-r(\mu)r^*(\mu)
       \end{pmatrix},\quad \mu\in\mathbb{R}\setminus\left(\{|\mu-1|\leq\sqrt{2}\}\triangle\{|\mu+1|\leq\sqrt{2}\}\right),\\

      \begin{pmatrix}
          \frac{1}{dd^*}&-\frac{aB-bA}{d^*}\\
          \frac{a^*B^*-b^*A^*}{d}&1
       \end{pmatrix},\quad \mu\in\mathbb{R}\cap\left(\{|\mu-1|<\sqrt{2}\}\triangle\{|\mu+1|<\sqrt{2}\}\right),\\

      \begin{pmatrix}
          1&0\\
          \frac{B^*}{ad}&1
       \end{pmatrix},\quad \mu\in\mathbb{C}_+\cap\left(\{|\mu-1|=\sqrt{2}\}\cup\{|\mu+1|=\sqrt{2}\}\right),\\

      \begin{pmatrix}
          1&-\frac{B}{a^*d^*}\\
          0&1
       \end{pmatrix},\quad \mu\in\mathbb{C}_-\cap\left(\{|\mu-1|=\sqrt{2}\}\cup\{|\mu+1|=\sqrt{2}\}\right),\\   
   \end{cases}
\end{equation}
and $r(\mu):=\frac{b(\mu)}{a^*(\mu)}$.
    \end{subequations}

\begin{remark}
   Notice that the jump
 matrix is written in terms of the spectral functions $a$, $b$, $A$, and $B$. These
 spectral functions are determined by the initial values $\tilde u(x,0)$ and the boundary
 values $\tilde u(0,t)$, $\tilde u_x(0,t)$, and $\tilde u_{xx}(0,t)$.
\end{remark}

\item Behavior at $\infty$:
\begin{equation}\label{inf_M_(xt)}
     M^{(xt)}(x,t,\mu)=I-\frac{1}{\mu}\begin{pmatrix}
    \zeta_1&\eta_2\\
    \eta_1&-\zeta_2
\end{pmatrix}+O(\frac{1}{\mu^2}),\quad\mu\to\infty,
\end{equation}
where $\eta_1(x,t)=\eta_2(x,t)=1-\frac{1}{\tilde m(x,t)}$.

\item $\det M^{(xt)}(x,t,\mu)\equiv1$

\item Behavior at $\pm 1$.
Using $d(\mu)=\frac{\ii}{2(\mu-1)}d_{-1}+d_0+O(\mu-1)$ with $d_{-1}=\gamma(\bar A_1-\bar B_1)-\Gamma (a_1-b_1)$, we obtain

\begin{equation}\label{sing_M_(xt)}
M^{(xt)}(x,t,\mu)=\begin{cases}
\frac{\ii\beta_+(x,t)}{2(\mu-1)}\begin{pmatrix} -\tilde C & 1 \\ -\tilde C & 1 \end{pmatrix}+\ord(1), &\mu\to 1,\ \ \Im\mu>0 ,\\
-\frac{\ii\beta_+(x,t)}{2(\mu +1)}\begin{pmatrix} \tilde C & 1 \\ -\tilde C & -1 \end{pmatrix}+\ord(1), &\mu\to -1,\ \ \Im\mu>0,
\end{cases}
\end{equation}
with some $\beta_+(x,t)\in\mathbb{R}$ and 
\begin{equation}\label{tilC-1}
\tilde C:=\begin{cases}
0,&\text{if } d_{-1}\neq 0 \text{ (generic case)},\\
\frac{d_0-A_1b_1+B_1a_1}{d_0},&\text{if } d_{-1}=0,
\end{cases}
\end{equation}
where $A_1=A(1)$, $B_1=B(1)$, $a_1=a(1)$, $b_1=b(1)$, $d_0=d(1)$, and \[d_{-1}:= -2\ii\lim\limits_{\mu\to 1}(\mu-1)d(\mu).\] 

\item Symmetry properties:
\begin{equation}\label{sym-M_(xt)}
M^{(xt)}(\bar\mu)=\sigma_1\overline{M^{(xt)}(\mu)}\sigma_1,\qquad M^{(xt)}(-\mu)=\sigma_2M^{(xt)}(\mu)\sigma_2,\qquad M^{(xt)}(\mu^{-1})=\sigma_1M^{(xt)}(\mu)\sigma_1.
\end{equation}

\item Residue properties. Let $\{\kappa_j\}$ be zeros of $d(\mu)$ in $\mathcal{D}_+$. Assume that they are all simple. Then 

\begin{align}\label{res-M+_(xt)_2}
\Res_{\kappa_j}M^{(xt)(1)}(x,t,\mu)&=\frac{B^*(\kappa_j)e^{2p(x,t,\kappa_j)}}{\dot d(\kappa_j)a(\kappa_j)}M^{(xt)(2)}(x,t,\kappa_j),\\
\label{res-M-_(xt)_2}
\Res_{\bar\kappa_j}M^{(xt)(2)}(x,t,\mu)&=\frac{B(\bar \kappa_j)e^{-2p(x,t,\bar\kappa_j)}}{\dot d^*(\bar\kappa_j)}M^{(xt)(1)}(x,t,\bar\kappa_j).
\end{align}

\item Behavior at $\ii$: 
\begin{equation}\label{i_beh-M_(xt)}
    \begin{aligned}
    M^{(xt)}(x,t,\mu)=\left(I+ \right. & M_1^{(xt)}(x,t)(\mu-\ii)
    \left. + \ O((\mu-\ii)^2)\right) \times \\
    \times & \eul^{\frac{\ii(\mu^2-1)}{\mu}\left(\frac{1}{2}\int_0^x(\tilde m(\xi,t)-1)\dd\xi-\frac{1}{2}\int_0^t(\tilde m\tilde\omega)(0,\tau)\dd\tau-\frac{\nu(0)}{4}\right)\sigma_3}, 
    \end{aligned}
\end{equation}
where
\begin{equation*}
    M_1^{(xt)}=\begin{pmatrix}
    \gamma_1(x,t)&-\frac{1}{2}(\tilde u(x,t)+\tilde u_x(x,t))\\
    -\frac{1}{2}(\tilde u(x,t)-\tilde u_x(x,t))&\gamma_2(x,t)
\end{pmatrix}.
\end{equation*}

\begin{remark}
   This behavior of $M^{(xt)}(x,t,\mu)$ can be obtained by analyzing the integral equations \eqref{inteq_01}--\eqref{inteq_03} for the Jost solutions $\Phi_{0j}(x,t,\mu)$, $j=1,2,3$, at $\mu=\ii$ and invoking the relation \eqref{Phi_0_inf_rel} between $\Phi_{0j}$ and $\Phi_{\infty j}$.
\end{remark}

\end{enumerate}

\subsection{The main Riemann--Hilbert problem parametrized by $y$ and $t$}

The construction of jump matrix involves $p(x,t,\mu)$ (see \eqref{p_mu}), which in turn involves $\tilde u(x,t)$. This suggests the introduction of a new variable
\begin{equation}\label{y}
y(x,t)=\int_0^x\tilde m(\xi,t)\dd\xi-\int_0^t(\tilde m\tilde\omega)(0,\tau)\dd \tau, 
\end{equation}
in terms of which the jump matrix and the residue conditions can be written
explicitly. Notice that $\frac{\partial y}{\partial x}=\tilde m(x,t)$ and $\frac{\partial y}{\partial t}=-\tilde m(x,t)\tilde \omega(x,t)$.

\textbf{The Riemann--Hilbert problem RH$^{(xt)}$:} Given $a(\mu)$, $b(\mu)$, $A(\mu)$, $B(\mu)$ and the set $\{\kappa_j\}_1^J\subset \mathcal{D}_+$, find a piece-wise meromorphic $2\times 2$  matrix valued function $\hat M^{(xt)}(y,t,\mu)$ that satisfies the following conditions:

\begin{enumerate}
    \item Jump relation across $\check\Sigma$
    \begin{subequations}
        \label{jump_hatM_(xt)}
        \begin{equation}
           \hat M_-^{(xt)}(y,t,\mu)=\hat M_+^{(xt)}(y,t,\mu)\hat J^{(xt)}(y,t,\mu),\quad\mu\in\check\Sigma 
        \end{equation}
        where
         \begin{equation}
         \hat J^{(xt)}(y,t,\mu)=\eul^{-\hat p(y,t,\mu)\sigma_3} J^{(xt)}_0(\mu)\eul^{\hat p(y,t,\mu)\sigma_3}
        \end{equation} 
        with $J^{(xt)}_0(\mu)$ defined in \eqref{jump_M_(xt)_0} and $\hat p(y,t,\mu)=\frac{\ii(\mu^2-1)}{2\mu}\left(\frac{1}{2}y-\frac{4\mu^2}{(\mu^2+1)^2}t\right)$.

    \end{subequations}

\item Behavior at $\infty$:
\begin{equation}\label{inf_hatM_(xt)}
    \hat M^{(xt)}(x,t,\mu)=I+O(\frac{1}{\mu}),\quad\mu\to\infty.
\end{equation}

\item Behavior at $\pm 1$:

\begin{equation}\label{sing_hatM_(xt)}
\hat M^{(xt)}(y,t,\mu)=\begin{cases}
\frac{\ii\hat\beta_+(y,t)}{2(\mu-1)}\begin{pmatrix} -\tilde C & 1 \\ -\tilde C & 1 \end{pmatrix}+\ord(1), &\mu\to 1,\ \ \Im\mu>0,\\
-\frac{\ii\hat\beta_+(y,t)}{2(\mu +1)}\begin{pmatrix} \tilde C & 1 \\ -\tilde C & -1 \end{pmatrix}+\ord(1), &\mu\to -1,\ \ \Im\mu>0,
\end{cases}
\end{equation}
with some $\beta_+(x,t)\in\mathbb{R}$, where $\tilde C$ is as in \eqref{tilC-1}.

\item Symmetry properties:
\begin{equation}\label{sym-hatM_(xt)}
\hat M^{(xt)}(\bar\mu)=\sigma_1\overline{\hat M^{(xt)}(\mu)}\sigma_1,\qquad \hat M^{(xt)}(-\mu)=\sigma_2\hat M^{(xt)}(\mu)\sigma_2,\qquad \hat M^{(xt)}(\mu^{-1})=\sigma_1\hat M^{(xt)}(\mu)\sigma_1.
\end{equation}

\item Residue conditions:
\begin{align}\label{res-hatM+_(xt)_2}
\Res_{\kappa_j}\hat M^{(xt)(1)}(y,t,\mu)&=\frac{B^*(\kappa_j)e^{2\hat p(y,t,\kappa_j)}}{\dot d(\kappa_j)a(\kappa_j)}\hat M^{(xt)(2)}(y,t,\kappa_j),\\
\label{res-hatM-_(xt)_2}
\Res_{\bar\kappa_j}\hat M^{(xt)(2)}(y,t,\mu)&=\frac{B(\bar \kappa_j)e^{-2\hat p(y,t,\bar\kappa_j)}}{\dot d^*(\bar\kappa_j)}\hat M^{(xt)(1)}(y,t,\bar\kappa_j).
\end{align}

\end{enumerate}

\begin{proposition}
    \begin{enumerate}
        \item $\det \hat M^{(xt)}\equiv1$
    
        \item If the solution of the RH problem \eqref{jump_hatM_(xt)}--\eqref{res-hatM-_(xt)_2} exists, it is unique.
 
    \end{enumerate}
\end{proposition}

In such a way, we arrive at the following representational result:

\begin{proposition}\label{Prop:rep}
Let $\tilde u(x,t)$ be a solution of the mCH equation in the domain $x>0$, $0<t<T$ such that $\tilde\omega(0,t)\leq0$. Then $\tilde u(x,t)$ can be represented in terms of the unique solution of a Riemann--Hilbert problem \eqref{jump_hatM_(xt)}--\eqref{res-hatM-_(xt)_2}, for which the data (jump matrix and residue condition) are given in terms of the
initial values $\tilde u(x,0)$ and the boundary values $\tilde u(0,t)$, $\tilde u_x(0,t)$, $\tilde u_{xx}(0,t)$ via the associated spectral functions.

\end{proposition}

\begin{proof}
  The solution of the RH problem \eqref{jump_hatM_(xt)}--\eqref{res-hatM-_(xt)_2} is unique. Together with  \eqref{inf_M_(xt)} and \eqref{i_beh-M_(xt)}, this yields  the following procedure for representing $\tilde u(x,t)$ in terms of $\hat M^{(xt)}(y,t,\mu)$:

  \begin{enumerate}[Step 1.]
    \item Given $a(\mu)$, $b(\mu)$, $A(\mu)$, and $B(\mu)$,  construct the Riemann--Hilbert problem \textbf{RH$^{(xt)}$};

    \item Solve the constructed RH problem \textbf{RH$^{(xt)}$};

    \item Evaluate the solution of this RH problem at $\mu=\ii$:
    \begin{equation}\label{hatM_xt__i}
       \hat M^{(xt)}(y,t,\mu)=\begin{pmatrix}
           \hat a_1(y,t)&0\\
           0&\hat a_1^{-1}(y,t)
       \end{pmatrix}+(\mu-\ii)\begin{pmatrix}
           0&\hat a_2(y,t)\\\hat a_3(y,t)&0
       \end{pmatrix}+O((\mu-\ii)^2) 
    \end{equation}
and at $\mu=\infty$:
    \begin{equation}\label{hatM_xt__inf}
        M^{(xt)}(y,t,\mu)=I-\frac{1}{\mu}\begin{pmatrix}
    \hat \zeta_1(y,t)&\hat \eta_2(y,t)\\
    \hat \eta_1(y,t)&-\hat \zeta_2(y,t)
\end{pmatrix}+O(\frac{1}{\mu^2})
    \end{equation}

    \item Define $\hat u(y,t)$ and $\hat m(y,t)$ from this expansions in the following way (cf. \eqref{inf_M_(xt)} and \eqref{i_beh-M_(xt)}):
    \begin{align}\label{hat_u_via_RH_xt}
     &\hat{\tilde u}(y,t)=-(\hat a_2(y,t)\hat a_1(y,t)+\hat a_3(y,t)\hat a_1^{-1}(y,t)), \\\label{hat_u_x_via_RH_xt}
    & \hat{\tilde u}_x(y,t)= -(\hat a_2(y,t)\hat a_1(y,t)-\hat a_3(y,t)\hat a_1^{-1}(y,t)),\\\label{hat_m_via_RH_xt}
   & \hat{\tilde m}(y,t)= \frac{1}{1-\hat \eta_2(y,t)}.
    \end{align}
Define $x(y,t)$ as the solution of the problem (cf. \eqref{y})

    \begin{align}\label{y_via_RH_xt}
     &\frac{\partial x}{\partial y}=1-\hat \eta_2(y,t),\\
     &x(\eta(t),t)=0.
    \end{align}
Then
\begin{subequations}\label{v_j_via_RH_xt}
\begin{align}
   &\tilde u(x,t)=
     \hat{\tilde u}(y(x,t),t),\\
     &\tilde m(x,t)=
     \hat{\tilde m}(y(x,t),t).
    \end{align}
    \end{subequations}
    
\end{enumerate}
  
\end{proof}

\begin{remark}
    Notice that $x(y,t)$ can be introduced as well by
\begin{equation}\label{main_x_y}
    x(y,t)=y+2 \log\hat a_1(y,t)-\nu(0).
\end{equation}
\end{remark}

\section{The mCH equation,  associated Lax pair, and  RH formalism in $y,t$ variables}\label{sec:mCHinyt}

\subsection{The mCH equation and the  associated Lax pair in $y,t$ variables}

The new space variable $y=y(x,t)$ introduced in \eqref{y} is the solution of the system
of equations
\begin{subequations}\label{y-xt}
\begin{align}
\label{y_x} y_x & = \tilde m,\\
   \label{y_t} y_t &=-\tilde m\tilde \omega
\end{align}
    \end{subequations}
   (compatible due to \eqref{mCH-2}) fixed by the condition $y(0,0)=0$.
Then, differentiating the identity $x(y(x,t),t)=x$, we obtain the equations characterizing the 
inverse mapping $x=x(y,t)$:
\begin{subequations}\label{x-yt}
\begin{align}
\label{x_y} x_y(y,t) & = \frac{1}{\hat{\tilde m}(y,t) },\\
   \label{x_t} x_t(y,t) &=\hat{\tilde \omega}(y,t),
\end{align}
    \end{subequations}
where $\hat{\tilde m }(y,t) = \tilde m(x(y,t),t)$ and  $\hat{\tilde \omega }(y,t) =  \tilde \omega(x(y,t),t)$.

\begin{proposition} (mCH equation in the $(y,t)$ variables) Let $\tilde u(x,t)$, $\tilde \omega(x,t)$ and $\tilde m(x,t)$ satisfy \eqref{mCH2} and let $y(x,t)$ be 
such that \eqref{y-xt} hold.
%defined as in \eqref{y}. 
Then the mCH equation \eqref{mCH2} in the $(y,t)$ variables reads as the following system of equations
for  $\hat{\tilde m}(y,t):=\tilde m(x(y,t),t)$, $\hat {\tilde u}(y,t):=\tilde u(x(y,t),t)$, and 
$\hat {\tilde q}(y,t):=\tilde u_x(x(y,t),t)$:

\begin{subequations}\label{mCH_in_y}
\begin{align}\label{mCH_in_y-1}
&(\hat {\tilde m}^{-1})_t(y,t)= 2\hat {\tilde q}(y,t),\\\label{mCH_in_y-2}
&\hat {\tilde m}(y,t)=\hat {\tilde u}(y,t) - \left({\hat{\tilde q}}\right)_y(y,t) \hat{\tilde m}(y,t) +1,\\\label{mCH_in_y-3}
&\hat{\tilde q}(y,t)=\hat{\tilde u}_y(y,t)\hat {\tilde m}(y,t).
\end{align}
\end{subequations}
    
\end{proposition}

\begin{proof}
As we discussed above, \eqref{y-xt} implies \eqref{x-yt}.
 Substituting $\tilde m_t=-(\tilde m\tilde \omega)_x$ from \eqref{mCH2} and $x_t=\hat{\tilde \omega}(y,t)$ from \eqref{x_t} into the equality
\[
\hat{\tilde m}_t=\left(\tilde m_xx_t+\tilde m_t\right)\mid_{x=x(y,t)}
\]
and we get \eqref{mCH_in_y-1}.

Now, substituting $x_y=\frac{1}{\hat {\tilde m}}$ from \eqref{x_y} into the equality
$\hat {\tilde u}_y(y,t)=\left(\tilde u_x(x,t)x_y(y,t)\right)\mid_{x=x(y,t)}$
we get \eqref{mCH_in_y-3}.

Finally, $\hat {\tilde q}_y(y,t)=\left(\tilde u_{xx}(x,t)x_y(y,t)\right)\mid_{x=x(y,t)}$ implies \eqref{mCH_in_y-2}.
\end{proof}

\begin{remark}
    Introducing $\hat {\tilde \omega}(y,t):=\tilde\omega(x(y,t),t)=\hat {\tilde u}^2(y,t)-\hat {\tilde q}^2(y,t)+2\hat {\tilde u}(y,t)$, equation \eqref{mCH_in_y} implies

\begin{equation}\label{mCH_in_y-4}
(\hat {\tilde m}^{-1})_t(y,t)= \hat {\tilde \omega}_y(y,t).
\end{equation}

\end{remark}

Similarly, the reverse  change of variable $(y,t)\mapsto(x,t)$ 
reduces \eqref{mCH_in_y} to \eqref{mCH2} with 
$\tilde u(x,t):=\hat{\tilde u}(y(x,t),t)$, $\tilde m(x,t):=\hat{\tilde{ m}}(y(x,t),t)$ and $\tilde u_x(x,t):=\hat{\tilde q}(y(x,t),t)$. 
\begin{proposition}\label{prop:hatu_u}
    Let $\hat{\tilde u}(y,t)$, $\hat{\tilde m}(y,t)$,  and $\hat{\tilde q}(y,t)$  satisfy 
    \eqref{mCH_in_y} and let $x(y,t)$ be 
such that \eqref{x-yt} holds. Define $\tilde u(x,t):=\hat{\tilde u}(y(x,t),t)$, $\tilde m(x,t):=\hat{\tilde{ m}}(y(x,t),t)$ and $\tilde u_x(x,t):=\hat{\tilde q}(y(x,t),t)$.
Then \eqref{mCH2} holds for $\tilde u(x,t)$, $\tilde m(x,t)$ and $\tilde \omega(x,t)=\tilde u^2(x,t)-\tilde u_x^2(x,t)+2\tilde u(x,t)$.
\end{proposition}

Now, let us reformulate the original Lax pair equations in the $(y,t)$ variables.

\begin{proposition}
The Lax pair \eqref{Lax-Q-form} in the variables $(y,t)$ takes the form 
\begin{subequations}\label{Lax_y}
\begin{align}\label{Lax_y_y}
    &\hat\Psi_y+\frac{\ii(\mu^2-1)}{4\mu}\sigma_3\hat\Psi=\widetilde{U}\hat \Psi,\\\label{Lax_y_t}
   &\hat \Psi_t-\frac{2\ii(\mu^2-1)\mu}{(\mu^2+1)^2}\sigma_3\hat\Psi=\widetilde{V}\hat\Psi.
\end{align}
with
\begin{subequations}\label{tilde-UV}
\begin{align}\label{tilde-U}
\widetilde U(y,t,\mu)&=\frac{\ii f(y,t)}{\mu-1}
\begin{pmatrix}
1 & -1 \\ 1 & -1
\end{pmatrix}+\frac{\ii f(y,t)}{\mu+1}
\begin{pmatrix}
1 & 1 \\ -1 & -1
\end{pmatrix}+\ii f(y,t)
\begin{pmatrix}
0 & -1 \\ 1 & 0
\end{pmatrix}, \\ 
\label{tilde-V}
\widetilde V(y,t,\mu)&=\frac{\ii g_0(y,t)}{\mu-1}
\begin{pmatrix}
1 & -1 \\ 1 & -1
\end{pmatrix}+\frac{\ii g_0(y,t)}{\mu+1}
\begin{pmatrix}
1 & 1 \\ -1 & -1
\end{pmatrix}\notag\\
&\quad
+\frac{1}{\mu-\ii}
\begin{pmatrix}
0 & g_1(y,t) \\ g_2(y,t) & 0
\end{pmatrix}+\frac{1}{\mu+\ii}
\begin{pmatrix}
0 & g_2(y,t) \\ g_1(y,t) & 0
\end{pmatrix},						
\end{align}
\end{subequations}
with $f$, $q$, $g_1$, and $g_2$ as follows:
\begin{equation}\label{f-q-g}
f=-\frac{\hat {\tilde m}-1}{2\hat {\tilde m}},\quad g_0=\hat {\tilde u},\quad g_1 =-\hat {\tilde u}-\hat{\tilde q},\quad g_2=\hat {\tilde u}-\hat{\tilde q}.
\end{equation}
\end{subequations}

\end{proposition}

\begin{proof}
  Introducing $\hat\Psi(y,t) = \hat\Phi (x(y,t),t)$ and taking into account \eqref{x_y} and \eqref{x_t}, the Lax pair \eqref{Lax-Q-form} in the variables $(y,t)$ takes the form \eqref{Lax_y}.
\end{proof}

\subsection{The  RH formalism in $y,t$ variables}
In this subsection we discuss how to arrive at a (local) solution of the mCH equation in the 
$y,t$ variables starting from a RH problem parametrized by $y$ and $t$, suggested 
by the properties of Jost solutions.

Let $\Gamma$ be an oriented contour in the complex plane that is invariant under the transformations $\mu\mapsto\frac{1}{\mu}$, $\mu\mapsto\bar\mu$ and $\mu\mapsto -\mu$, and let ${\mathbb C}\setminus\Gamma = D_1\cup D_2$ so that $\Gamma$ is the counterclockwise boundary of $D_1$ (the clockwise boundary of $D_2$).

Consider the following \textbf{RH problem parametrized by $y$ and $t$}:
find a piece-wise meromorphic ($\mu\in {\mathbb C}\setminus\Gamma$), $2\times 2$-matrix-valued function $\hat M(y,t,\mu)$ satisfying the following conditions:
\begin{enumerate}[\textbullet]
\item
\emph{Jump} condition
\begin{equation}\label{jump-y_loc}
\hat M_+(y,t,\mu)=\hat M_-(y,t,\lambda) \hat J(y,t,\mu),\qquad \mu\in\Gamma,
\end{equation}
where 
$J(y,t,\mu)=\eul^{-\hat p(y,t,\mu)\sigma_3}J_0(k)\eul^{\hat p(y,t,\mu)\sigma_3}$ with $\hat p(y,t,\mu)$ defined in \eqref{y} and some 
$J_0(\mu)$ with $\det J_0(k)\equiv 1$ and  such that
\begin{equation}\label{jump_at_inf}
    J_0(\mu)=I+O(\frac{1}{\mu}),\quad \mu\to\infty.
\end{equation}
If $\ii\in\Gamma$, we additionally assume that
\begin{equation}\label{jump_at_ii}
   J_0(\mu)=I+O(\mu-\ii), \quad \mu\to\ii. 
\end{equation}

\item  \emph{Normalization} condition:
\begin{equation}\label{norm-m-hat_loc}
\hat M(y,t,\mu)=I+\ord(\frac{1}{\mu}), \quad \mu\to\infty.
\end{equation}

\item \emph{Symmetry} condition:
\begin{equation}\label{sym-M-hat}
\hat M(\bar\mu)=\sigma_1\overline{\hat M(\mu)}\sigma_1,\quad\hat M(-\mu)=\sigma_2\hat M(\mu)\sigma_2,\quad\hat M(\mu^{-1})=\sigma_1\hat M(\mu)\sigma_1
\end{equation}
where $\hat M(\mu)\equiv\hat M(y,t,\mu)$.

\item \emph{Singularity} conditions at $\pm1$:

\begin{equation}\label{sing_hatM}
\hat M(y,t,\mu)=\begin{cases}
\frac{\ii\alpha_+(y,t)}{2(\mu-1)}\begin{pmatrix} -c& 1 \\ -c & 1 \end{pmatrix}+\ord(1), &\mu\to 1,\ \ \Im\mu>0,\\
-\frac{\ii\alpha_+(y,t)}{2(\mu +1)}\begin{pmatrix} c & 1 \\ -c & -1 \end{pmatrix}+\ord(1), &\mu\to -1,\ \ \Im\mu>0,
\end{cases}
\end{equation}
with non specified $\alpha_+(y,t)\in\mathbb{R}$.

\item \emph{Residue} conditions:
\begin{align}\label{res_hatM_loc}
\Res_{ \mu_j}\hat M^{(1)}(y,t,\mu)&= c_j \hat M^{(2)}(y,t, \mu_j) \eul^{-\hat p(y,t, \mu_j)}, ~  \mu_j\in D_1, \\\label{res_hatM__loc}
\Res_{-\bar\mu_j}\hat M^{(2)}(y,t,\mu)&=\bar c_j \hat M^{(1)}(y,t,-\bar \mu_j)\eul^{ \hat p(y,t,-\bar \mu_j)} , ~ \bar \mu_j\in D_2
\end{align}
with some $\{ \mu_j\}$, $\{c_j\}$ such that the set $\{ \mu_j\}$ is symmetric under the mappings $\mu\mapsto -\frac{1}{\mu}$ ($c_{-\frac{1}{\mu_j}}= -\mu_j^2c_{\mu_j}$), $\mu\mapsto-\bar\mu$ ($c_{-\bar\mu_j}= \bar c_{\mu_j}$).
\end{enumerate}

\begin{remark}
    The symmetries \eqref{sym-M-hat} together with \eqref{jump_at_ii} imply that 
    $\hat M(y,t,\mu)=\begin{pmatrix}
        \phi(y,t)&0\\
        0& \phi^{-1}(y,t)
    \end{pmatrix}+O(\mu-\ii)$ as $\mu\to\ii$.
\end{remark}

The following propositions are similar to those in \cite{BKS20}.

\begin{proposition}
 Assume that $\hat M $ is a solution of RH problem \eqref{jump-y_loc}--\eqref{res_hatM__loc}. Then $\det \hat M=1$.
\end{proposition}

\begin{proposition}
     If a solution of the RH problem \eqref{jump-y_loc}--\eqref{res_hatM__loc} exists, it is unique.
\end{proposition}

Now, assume that the RH problem \eqref{jump-y_loc}--\eqref{res_hatM__loc} 
has a solution $\hat M(y,t,\mu)$ that is differentiable w.r.t. $y$ and $t$.

\begin{proposition}
    Evaluating $\hat M(y,t,\mu)$ at particular points of $\overline{\mathbb C}$ one can 
    obtain a solution of the CH equation in the $y,t$ variables.
\end{proposition}

We proceed as follows (see \cite{BKS20}):
\begin{enumerate}[(a)]
\item 
Starting from $\hat M(y,t,\mu)$, define $2\times 2$-matrix valued functions \[ \hat \Psi (y,t,\mu):= \hat M(y,t,\mu)\eul^{- p(y,t,\mu)\sigma_3}\] and show that $\hat\Psi(y,t,\mu)$ satisfies the system of differential equations:
\begin{equation}\label{Lax-hat-hat}
\begin{split}
\hat\Psi_y&=\doublehat{U}\hat\Psi, \\
\hat \Psi_t&=\doublehat{V}\hat\Psi,
\end{split}
\end{equation}
where $\doublehat{U}$ and $\doublehat{V}$ have the same (rational) dependence on $k$ as in \eqref{Lax_y_y} and \eqref{Lax_y_t}, with coefficients given in terms of $\hat M(y,t,\mu)$ evaluated at appropriate values of $k$.
\item
Show that the compatibility condition for \eqref{Lax-hat-hat}, i.e., the equality $\doublehat{U}_t - \doublehat{V}_y + [\doublehat{U},\doublehat{V}]=0$, reduces to 
\eqref{mCH_in_y}.
\end{enumerate}

\begin{proposition}\label{Prop_Lax_y}
    Let $\hat M(y,t,\mu)$ be the solution of the RH problem \eqref{jump-y_loc}--\eqref{res_hatM__loc}. 
     Define
\begin{equation}\label{hatpsiy}
\hat\Psi(y,t,\mu)\coloneqq\hat M(y,t,\mu)\eul^{-\hat p(y,t,\mu)\sigma_3},
\end{equation}
where $\hat p(y,t,\mu)\coloneqq-\frac{\ii(\mu^2-1)}{4\mu}\left(-y+\frac{8\mu^2}{(\mu^2+1)^2}t\right)$. Then $\hat\Psi(y,t,\mu)$ satisfies the differential equation
\[
\hat\Psi_y=\doublehat{U}\hat\Psi
\]
with $\doublehat{U}=-\frac{\ii(\mu^2-1)}{4\mu}\sigma_3+\widetilde{U}$, where $\widetilde{U}$ is as in \eqref{tilde-U} with $f$ given by 
\[
f(y,t)\coloneqq-\frac{\eta_2(y,t)}{2},
\] 
$\eta_2(y,t)$ being extracted from the large $\mu$ expansion of $\hat M(y,t,\mu)$:
\[
\hat M(y,t,\mu)=I+\frac{1}{\mu}
\begin{pmatrix}\xi(y,t) & \eta_2(y,t)\\\eta_2(y,t)&-\xi(y,t)\end{pmatrix}+\ord(\mu^{-2}),\qquad\mu\to\infty.
\]

\end{proposition}

\begin{proposition}\label{Prop_Lax_t}
$\hat\Psi(y,t,\mu)$ defined as in Proposition \ref{Prop_Lax_y} satisfies the differential equation
\begin{equation}\label{psi-t-i}
\hat\Psi_t=\doublehat{V}\hat\Psi
\end{equation}
with $\doublehat{V}=\frac{2\ii(\mu^2-1)\mu}{(\mu^2+1)^2}\sigma_3+\widetilde V$, where $\widetilde V$ is as in \eqref{tilde-V} with coefficients $g_1=\gamma_1$ and $g_2=\gamma_2$ determined by evaluating $\hat M(y,t,\mu)$ as $\mu\to\ii$:

\begin{subequations}\label{hat-M-i}
    \begin{align}
&\hat M(\mu)=\begin{pmatrix}
                  a_1 & 0 \\
                  0 & a_1^{-1}
\end{pmatrix}  
+\begin{pmatrix}
                  0 & a_2 \\
                  a_3 & 0
\end{pmatrix}(\mu-\ii)+\ord((\mu-\ii)^2), \qquad \mu\to\ii;\\
&\gamma_1:=2a_2a_1;\\
&\gamma_2:=-2a_3a_1^{-1},
\end{align}
\end{subequations}

and the coefficient $g_0=\beta_2$  determined by evaluating $\hat M(y,t,\mu)$ as $\mu\to 1$:
\begin{align*}
    &\hat M(y,t,\mu)=-\frac{\ii \alpha_+(y,t)}{2(\mu-1)}\begin{pmatrix} -c & 1 \\ -c & 1 \end{pmatrix}+\begin{pmatrix}  n(y,t) & m(y,t) \\  f(y,t) &  g(y,t) \end{pmatrix}+\ord(\mu+1)\text{ as }\mu\to -1,\ \mu \in\D{C}^+;\\
    &\beta_2:=\frac{1}{2}\left(\alpha_{+t}(-c n+ f)+\alpha_{+}(-c n_{t}+ f_{t})-c\alpha_+^2\right);\\
        &\beta_1:=\frac{1}{2}\left(\alpha_{+y}(-c n+f)+\alpha_{+}(-c n_{y}+ f_{y})-c\alpha_+^2\right).
\end{align*}

\end{proposition}

The next step is to demonstrate that the compatibility condition
\begin{equation}\label{compat}
\doublehat{U}_t - \doublehat{V}_y + [\doublehat{U},\doublehat{V}]=0
\end{equation}
yields the mCH equation in the $(y,t)$ variables. Now, evaluating the compatibility equation \eqref{compat} at the singular points for $ \doublehat{U}$ and  $\doublehat{V}$, we get algebraic and differential equations amongst the coefficients of $ \doublehat{U}$ and  $\doublehat{V}$, i.e., amongst $\beta_1$, $\beta_2$, $\gamma_1$, $\gamma_2$,  and  $\eta_2$ that can be reduced to \eqref{mCH_in_y}.

\begin{proposition}
Let $\beta_1(y,t)$, $\beta_2(y,t)$, $\gamma_1(y,t)$, $\gamma_2(y,t)$, and $\eta_2(y,t)$ be the functions determined in terms of $\hat M(y,t,\mu)$ as in Propositions \ref{Prop_Lax_y} and \ref{Prop_Lax_t}. Then they satisfy the following equations:
\begin{subequations}\label{rel}
\begin{align}\label{rel-a}
&\beta_{1 t}+\frac{\gamma_1+\gamma_2}{2}= 0;\\
\label{rel-b}
&\beta_2 - \frac{\gamma_2-\gamma_1}{2}=0;\\
\label{rel-c}
&(\gamma_1-\gamma_2)_y-(1+2\beta_1)(\gamma_1+\gamma_2)=0;\\
\label{rel-d}
&(\gamma_2+\gamma_1)_y+4\beta_1-(1+2\beta_1)(\gamma_1-\gamma_2)=0;\\
\label{rel-e}
&\beta_1=-\frac{\eta_2}{2}.
\end{align}
\end{subequations}
\end{proposition}

\begin{proposition}\label{reduce}
Let $\hat{\tilde m} (y,t)$, $\hat {\tilde u}(y,t)$, and $\hat {\tilde q}(y,t)$ be defined in terms of $\eta_2$, $\gamma_1$, and $\gamma_2$ as follows:
\begin{equation}\label{hmbbgg}
\hat {\tilde m}(y,t) = (1-\eta_2(y,t))^{-1},\quad \hat {\tilde u}(y,t)=\frac{\gamma_2(y,t)-\gamma_1(y,t)}{2},\quad \hat {\tilde q}(y,t)=\frac{\gamma_2(y,t)+\gamma_1(y,t)}{2}.
\end{equation}
Then the equations \eqref{rel} reduce to \eqref{mCH_in_y}.
\end{proposition}

Now, introduce the reverse change of variable $(y,t)\mapsto (x,t)$ by
\begin{equation}\label{change}
    x(y,t)=\int_{\eta(t)}^y(1-\hat\eta_2)(\xi,t)\dd \xi.
\end{equation}

Indeed, it is immediate from the definition that $x(\eta(t),t)=0$, and \eqref{x_y} is satisfied. Moreover, differentiating \eqref{change} with respect to 
$t$ and using \eqref{mCH_in_y-4} yields \eqref{x_t}.

\begin{remark}
   Substituting the expressions for $\gamma_1$ and $\gamma_2$ from \eqref{hat-M-i} into \eqref{hmbbgg} yields the formulas \eqref{hat_u_via_RH_xt} and \eqref{hat_u_x_via_RH_xt}.
\end{remark}

Then $u(x,t)=\hat u(y(x,t),t)$, $m(x,t)=\hat m(y(x,t),t)$ and $u_x(x,t)=\hat q(y(x,t),t)$ satisfy \eqref{mCH-2} by Proposition \ref{prop:hatu_u}.

\section{Initial Boundary Value Problem}\label{sec:7}

To fix ideas, we consider the generic case (with $\tilde C=0$ in \eqref{sing_hatM_(xt)}) and assume that zeros of $a(k)$, $A(k)$, and $d(k)$ 
are simple and mutually disjoint. Moreover, assume that the initial and boundary
data are smooth and that the  initial data decay to $0$ fast as $x\to\infty$.

\begin{proposition}\label{prop:7}
    Let smooth functions $\{\tilde u_0(x),~x\geq0; ~\{v_j(t)\}_0^2,~0\leq t\leq T\}$ be 
    such that the following conditions are satisfied:
    \begin{enumerate}
        \item 
Internal compatibility  conditions
\begin{enumerate}[(i)]
    \item $\partial_x^j \tilde u_0(0)=v_j(0)$, $j=0,1,2$;
        \item if $\tilde \omega_0(t)=0$ for  $t\in (T_1,T_2)$, then $v_{0t}(t)-v_{2t}(t)+2v_1(t)(v_0(t)-v_2(t)+1)^2=0$ for  $t\in (T_1,T_2)$;
\end{enumerate}
\item Sign-preserving condition 
\[
\tilde \omega_0(t)=v_0^2(t)-v_1^2(t)+2v_0(t)\leq0.
\]
\item 
Compatibility of boundary and initial data:
 the spectral functions 
$a(\mu)$, $b(\mu)$, $A(\mu)$, and $B(\mu)$
associated to the data via the direct $x$-problem and $t$-problem
satisfy the global relations
\eqref{relations_Phi_inf3} and \eqref{relations_Phi_03_i}.
\item
The Riemann--Hilbert problem \eqref{jump_hatM_(xt)}--\eqref{res-hatM-_(xt)_2} has a solution $\hat M^{(xt)}(y,t,\mu)$ for all $y\geq0$ and $0\leq t<T$ such that
\begin{enumerate}[(a)]
    \item it is differentiable w.r.t. $y$ and $t$;
\item 
functions 
$\hat\eta_2(y,t)$ and $\hat a_j (y,t)$ 
evaluated from $\hat M^{(xt)}(y,t,\mu)$, see \eqref{hatM_xt__i} and \eqref{hatM_xt__inf},
are differentiable w.r.t. $y$ and $t$ and decay fast as $y\to\infty$;
    
 \item 
 $\hat{\tilde m}(y,t):=\frac{1}{1-\hat \eta_2(y,t)}>0$. 
 
\end{enumerate}

 \end{enumerate}

Then the IBVP 
\begin{subequations}\label{SP_IBVP}
\begin{align}\label{SP-2_IBVP}
&\tilde m_t = -(\tilde \omega\tilde m)_x,\\
\label{tq_IBVP}
&\tilde m=\tilde u -\tilde u_{xx}+1,\\
\label{tw_IBVP}
&\tilde \omega= \tilde u^2 -\tilde u_{x}^2+2\tilde u;\\
\label{boundary_IBVP}
    &\tilde u(0,t) = v_0(t), \quad \tilde u_x(0,t) = v_1(t), \quad \tilde u_{xx}(0,t) = v_2(t), \quad
     0\leq t \le T;\\
     \label{ic_IBVP}
     &\tilde u(x,0) = \tilde u_0(x), \quad x \geq 0;\\
     \label{as_IBVP}
     &\tilde u(x,t)\to 0, \quad x \to\infty, \quad
     0\leq t \leq T.
\end{align}
\end{subequations}
has a unique solution, $\tilde u(x,t)$, which can be represented in terms of $\hat M^{(xt)}(y,t,\mu)$ in the parametric form \eqref{hat_u_via_RH_xt}, \eqref{y_via_RH_xt}, \eqref{v_j_via_RH_xt}. 

\end{proposition}

\begin{proof}
        The proof consists in the following steps:
\begin{enumerate}[Step 1.]
    \item Prove that $\tilde u(x,t)$ satisfies the mCH equation \eqref{SP-2_IBVP}--\eqref{tw_IBVP}

    \item  Prove that $\tilde u(x,0) =\tilde u_0(x)$.

    \item Prove that $\tilde  u(0,t) = v_0(t)$, $\tilde u_x(0,t) = v_1(t)$ and $\tilde u_{xx}(0,t) = v_2(t)$.
\end{enumerate}

\textbf{Step 1.} The statement follows from the considerations in Section \ref{sec:mCHinyt}.

\textbf{Step 2.} The idea of the proof of Step 2 is based on establishing a relation between the
 solution $\hat{ M}^{(xt)}(y,t,\mu)$ of the \textbf{RH$^{(xt)}$} problem evaluated at $t = 0$ and the solution $\hat M^{(x)}(y,\mu)$ of the \textbf{RH$^{(x)}$}. We present this relation in the form of the multiplication by an appropriate matrix factor:
 \begin{subequations}
     \begin{align}
         &P^{(x)}(y,\mu)=\eul^{-\hat p(y,0,\mu)\sigma_3}P_0^{(x)}(\mu)\eul^{\hat p(y,0,\mu)\sigma_3},\\
         & P_0^{(x)}(\mu)=\begin{cases}
        \begin{pmatrix}
           1&0\\
           -\frac{B^*(\mu)}{a(\mu)d(\mu)}&1
        \end{pmatrix},\quad \mathcal{D}_+,\\
         \begin{pmatrix}
           1&-\frac{B(\mu)}{a^*(\mu)d^*(\mu)}\\
           0&1
        \end{pmatrix},\quad \mathcal{D}_-,\\
       I, \quad \text{otherwise}.
     \end{cases}
     \end{align}
 \end{subequations}

Thus, the problem reduces to show that  $\hat{ N}^{(xt)}(y,\mu)$ defined by
 \begin{equation}
   \hat{ N}^{(xt)}(y,\mu)=\hat{ M}^{(xt)}(y,0,\mu)P^{(x)}(y,\mu)
 \end{equation}
 satisfies the Riemann--Hilbert problem \textbf{RH$^{(x)}$}.

\begin{enumerate}
    \item The jump conditions match by construction.

    \item Noticing that $P_0^{(x)}(\mu)=I$ in $\mathcal{C}_-\cup\mathcal{C}_+$, the normalization condition is satisfied.

    \item In generic case, we have 
    $\frac{B^*(\mu)}{a(\mu)d(\mu)}=O(\mu\pm 1)$ as $\mu\to\mp 1$. Consequently, $P_0^{(x)}(\pm 1)=I$, and the conditions at $\mu=\mp 1$ match.

    \item From the symmetries \eqref{sym_a_} and \eqref{sym_A_}, it follows that symmetry conditions match.

    \item Finally, let's check that the residue condition match.

For $\mu\in \mathcal{D}_+$, we have 
 \begin{align}
     \hat{ N}^{(xt)(1)}(y,0,\mu)&=\hat{M}^{(xt)(1)}(y,0,\mu)-\frac{B^*(\mu)\eul^{2\hat p(y,0,\mu)}}{a(\mu)d(\mu)}\hat{M}^{(xt)(2)}(y,0,\mu),\\
     \hat{ N}^{(xt)(2)}(y,0,\mu)&=\hat{M}^{(xt)(2)}(y,0,\mu).
 \end{align}
 Taking into account the residue condition \eqref{res-hatM+_(xt)_2}, the singularities of the first column at $\mu=\kappa_j$ cancel.

 On the other hand, at $\mu=\mu_j$, the first column is singular due to the singularity of $\frac{1}{a(\mu)}$, and, using $d(\mu_j)=-b(\mu_j)B^*(\mu_j)$, the corresponding residue condition takes the form \eqref{res-hatM+_(x)}.

 The singularities in $\mathcal{D}_-=\overline{\mathcal{D}_+}$  can be treated in a similar way.
\end{enumerate}
Thus, by uniqueness of solution of RH problem \textbf{RH$^{(x)}$}, we conclude that $\hat{ M}^{(xt)}(y,0,\mu)P^{(x)}(y,\mu)=\hat{ M}^{(x)}(y,\mu)$.

Since \(P^{(x)}_{0}(\mu)=I\) for sufficiently large \(\mu\), 
relations \eqref{m_0_via_RH_x_m_2} and \eqref{hat_m_via_RH_xt} imply that
\[
\hat {\tilde m}(y,0)=\hat{\tilde m}_{0}(y).
\]
Furthermore, combining \eqref{M_(x)_x_y} and \eqref{y_via_RH_xt} yields
\[
x(y)=x(y,0),
\]
where $x(y)$ in the left hand side is the change of the variable associated with the $x$-problem for $\tilde m_{0}$ whereas the right hand side is the change of the variable associated with the main RH problem at $t=0$.

\textbf{Step 3.} The idea of the proof of Step 3 is similar to that of the proof of Step 2. We establish a relation between $\hat{ M}^{(xt)}(y,t,\mu)$ evaluated at $y=\eta(t)$ and the solution $\hat{ M}^{(t)}(z,t,\mu)$ of the Riemann--Hilbert problem RH$^{(t)}$  associated with the boundary data $v_0(t)$, $v_1(t)$, $v_2(t)$. 
This relation can be expressed as multiplication by an appropriate piece-wise meromorphic matrix factor:
\begin{subequations}
    \begin{equation}
        P^{(t)}(t,\mu)=\eul^{-\hat p(\eta(t),t,\mu)\sigma_3}P^{(t)}_0(t,\mu)\eul^{\hat p(\eta(t),t,\mu)\sigma_3}
    \end{equation}
with
\begin{equation}
     P^{(t)}_0(t,k)=\begin{cases}

\begin{pmatrix}
    \frac{a}{A}&-(bA-Ba)\\
    0&\frac{A}{a}
\end{pmatrix},\quad \mu\in \mathcal{C}_-\cap\{|\mu\pm\ii| >\epsilon\},\\

\begin{pmatrix}
    \frac{ a}{A}&-(b A- B a)\\
    0&\frac{ A}{ a}
\end{pmatrix}\eul^{\frac{\ii(\mu^2-1)}{4\mu}\eta(T)\sigma_3},\quad \mu\in \mathcal{C}_-\cap\{|\mu\pm\ii| <\epsilon\},\\

\begin{pmatrix}
    d&-\frac{b}{A^*}\\
    0&\frac{1}{d}
\end{pmatrix},\quad \mu\in\mathcal{D}_+\cap\{|\mu\pm\ii| >\epsilon\},\\

\begin{pmatrix}
     d&-\frac{b}{A^*}\\
    0&\frac{1}{ d}
\end{pmatrix}\eul^{\frac{\ii(\mu^2-1)}{4\mu}\eta(T)\sigma_3},\quad \mu\in\mathcal{D}_+\cap\{|\mu\pm\ii| <\epsilon\},\\

\end{cases}
 \end{equation}

\begin{equation}
     P^{(t)}_0(t,k)=\begin{cases}

\begin{pmatrix}
    \frac{A^*}{a^*}&0\\
    -(b^*A^*-B^*a^*)&\frac{a^*}{A^*}
\end{pmatrix},\quad \mu\in\mathcal{C}_+\cap\{|\mu\pm\ii| >\epsilon\},\\

\begin{pmatrix}
    \frac{ A^*}{ a^*}&0\\
    -(b^* A^*- B^* a^*)&\frac{ a^*}{ A^*}
\end{pmatrix}\eul^{\frac{\ii(\mu^2-1)}{4\mu}\eta(T)\sigma_3},\quad \mu\in\mathcal{C}_+\cap\{|\mu\pm\ii| <\epsilon\},\\

\begin{pmatrix}
   \frac{1}{d^*} &0\\
    -\frac{b^*}{A}&d^*
\end{pmatrix},\quad \mu\in\mathcal{D}_-\cap\{|\mu\pm\ii| >\epsilon\},\\

\begin{pmatrix}
   \frac{1}{ d^*} &0\\
    -\frac{ b^*}{ A}&d^*
\end{pmatrix}\eul^{\frac{\ii(\mu^2-1)}{4\mu}\eta(T)\sigma_3},\quad \mu\in\mathcal{C}_+\cap\{|\mu\pm\ii| <\epsilon\}
\end{cases}
 \end{equation}
\end{subequations}

  Thus the problem reduces to show that  $\hat{\hat N}^{(xt)}(t,\mu)$ defined by
 \begin{equation}
   \hat{\hat N}^{(xt)}(y,\mu)=\hat{ M}^{(xt)}(\eta(t),t,\mu)P^{(t)}(t,\mu)
 \end{equation}
 satisfies the Riemann--Hilbert problem \textbf{RH$^{(t)}$} with $z=\eta(t)$.

 \begin{enumerate}
     \item The jump conditions match by construction.

     \item The global relation \eqref{relations_Phi_inf3} together with the asymptotic properties of $a(\mu)$ and $A(\mu)$ provides that 
     $P^{(t)}(\mu)\to I$ as $\mu\to\infty$ in $\mathbb{C}_+$ for all $t\leq T$. Therefore, the normalization condition is satisfied.

     \item Notice that in generic case, we have 
    \[d(\mu)=\mp\frac{\ii}{2}\frac{d_{-1}}{\mu\pm1}
+d_0+O(\mu \mp1), \quad\mu\to\mp 1, \quad\mu\in\mathbb{C}_+, \] 
and, by \eqref{ab_at_1} and \eqref{AB_at_1},
\[
\frac{b}{A^*}=-\frac{\gamma}{\Gamma}+O(\mu \mp1), \quad\mu\to\mp 1, \quad\mu\in\mathbb{C}_+.
\]
Consequently,
\[
P^{(t)}(y,\mu)=\mp\frac{\ii}{2(\mu\pm 1)}\begin{pmatrix}
   d_{-1}&0\\
   0&0
\end{pmatrix}+\begin{pmatrix}
    d_0&-\frac{\gamma}{\Gamma}\\
    0&0
\end{pmatrix}+O(\mu\pm 1),\quad\mu\to\mp 1, \quad\mu\in\mathbb{C}_+,
\]
and thus noticing that we do not specify  coefficients in the structural condition at  $\mu=\pm 1$, we conclude that the conditions at $\mu=\pm 1$ match.

     \item From the symmetries \eqref{sym_a_} and \eqref{sym_A_} it follows that symmetry conditions match.

     \item Finally, let us check that the residue conditions match.

     For $\mu\in\mathcal{D}_+$, we have
\[
\begin{aligned}
    \hat{ M}^{(xt)}(\eta(t),t,\mu)P^{(t)}(t,\mu)= & \left( d(\mu)\hat{ M}^{(xt)(1)}(\eta(t),t,\mu), -\frac{b(\mu)}{A^*(\mu)}\eul^{-2p(\eta(t),t,\mu)}\hat{ M}^{(xt)(1)}(\eta(t),t,\mu)\right. \\
    & +\left.\frac{1}{d(\mu)}\hat{ M}^{(xt)(2)}(\eta(t),t,\mu)\right).
\end{aligned}
\]
Taking into account the fact that $\kappa_j$ is a simple zero of $d(\mu)$ and the residue condition  \eqref{res-hatM+_(xt)_2}, the singularities of both columns at $\mu=\kappa_j$ cancel.

On the other hand, at $\mu=\bar\nu_j$,  the second column is singular due to the singularity of $\frac{1}{A^*(\mu)}$, and, using $d(\bar\nu_j)=-b(\bar{\nu}_j)B^*(\bar{\nu}_j)$, the corresponding residue condition takes the form \eqref{res-hatM-_(t)}.

 For $\mu\in\mathcal{D}_+$, we have
 \[\hat{ M}^{(xt)}(\eta(t),t,\mu)P^{(t)}(t,\mu)=\left(\frac{a}{A}\hat{ M}^{(xt)(1)}, -(bA-Ba)\eul^{-2p(\eta(t),t,\mu)}\hat{ M}^{(xt)(1)}+\frac{A}{a}\hat{ M}^{(xt)(2)}\right)\]
Taking into account the first column is singular due to the singularity of $\frac{1}{A(\mu)}$, the corresponding residue condition takes the form \eqref{res-hatM+_(t)}.
The singularities in $\mathcal{D}_-=\overline{\mathcal{D}_+}$ and $\mathcal{C}_+=\overline{\mathcal{C}_-}$  can be treated in a similar way.
     
 \end{enumerate}
 
Then, by uniqueness of solution of RH problem \textbf{RH$^{(x)}$}, we conclude that $\hat{\hat N}^{(xt)}(t,\mu)=\hat{ M}^{(t)}(\eta(t),t,\mu)$.

Since $v_j(t)$, $j=0,1$ are determined by the first two coefficients in the asymptotic expansion of $\hat{ M}^{(t)}(\eta(t),t,\mu)$ as $\mu\to\ii$, it is essential to control the behavior of $P^{(t)}_0(t,\mu)$ near $\mu=\ii$.

First of all, notice that
\[
\frac{a(\mu)}{A(\mu)}\eul^{\frac{\ii(\mu^2-1)}{4\mu}\eta(T)}=\eul^{-\frac{1}{2}\nu(0)}+O(\mu-\ii), \quad\mu\to\ii, \quad\mu\in\mathcal{C}_-
\]
and thus the global relation
\eqref{relations_Phi_03_i} yields,  for all $0\leq t<T$,
\begin{equation}\label{P_t_i}
    P^{(t)}(t,k)= \eul^{-\frac{1}{2}\nu(0)\sigma_3}+(\mu-\ii)\begin{pmatrix}
         p_1(t)&0\\
         0&p_2(t)
     \end{pmatrix}+O((\mu-\ii)^2),\quad \mu\to\ii,\quad \mu\in\mathcal{C}_-.
\end{equation}
Consequently, as $\mu\to\ii$ in $\mathcal{C}_-$
\[
\begin{aligned}
    &  \hat{\hat N}^{(xt)}(t,\mu) =  \\ 
    & = \begin{pmatrix}
           \hat a_1(\eta(t),t)\eul^{-\frac{1}{2}\nu(0)}&0\\
           0&\hat a_1^{-1}(\eta(t),t)\eul^{\frac{1}{2}\nu(0)}
       \end{pmatrix}  +(\mu-\ii) \begin{pmatrix}
           \tilde p_1(t)&\hat a_2(\eta(t),t)\eul^{\frac{1}{2}\nu(0)}\\\hat a_3(\eta(t),t)\eul^{-\frac{1}{2}\nu(0)}& \tilde p_2(t)
       \end{pmatrix}\\
       & +O((\mu-\ii)^2)
\end{aligned}
\]
and, in particular, the expressions for $\hat u(\eta(t),t)$ and $\hat u_x(\eta(t),t)$ given in \eqref{hat_u_via_RH_xt} and \eqref{hat_u_x_via_RH_xt} remain unchanged.

For $\mu\in \mathcal{D}_+$, \eqref{P_t_i} holds as well  (no use of the global relation is required in this case), therefore the same arguments as above apply.

Furthermore, the global relation \eqref{relations_Phi_inf3} ensures that $(b(\mu)A(\mu)-B(\mu)a(\mu))\eul^{-2\hat p(\eta(t),t,\mu)}$ decays exponentially fast  as $\mu\to\infty$ in $\mathbb{C}_+$ for $t< T$. Therefore, the expression for $\hat m(\eta(t),t)$ from \eqref{hat_m_via_RH_xt} remains unchanged.

\end{proof}

\section{Concluding remarks}

The mCH equation, similarly to the original CH equation,
is characterized by the singularities  at $\lambda=0$ in the original Lax pair
formulations. On the other hand, 
in contrast with the original CH equation, it turns to be possible 
to construct the pre-RH problem for the mCH
 from the eigenfunctions $\Phi_{\infty j}$
only whereas the  eigenfunctions $\Phi_{0 j}$ are used  as a tool for 
 evaluation of  $\Phi_{\infty j}$ at $\mu=\pm i$
corresponding to $\lambda=0$ in the original Lax pair. 

Concerning the validity of assumptions in Proposition \ref{prop:7}, 
we notice that following the ideas from \cite{BKS22}, the RH problem \textbf{RH$^{(xt)}$} can be transformed to a regular one \textbf{RH$^{R}$} (i.e., a problem having only a jump condition and a normalization at $\infty$), at least in the case 
without residue conditions. Then the solvability 
of the latter RHP can be  ensured under a small-norm assumption on the initial and boundary data (see, i.e., \cite{Z89}). As for smoothness and rapid decay of 
its solution, a standard approach is to assume that the data are from the Schwarz
class (for the initial data) and smooth (for the boundary data).
Therefore, it is the positivity assumption (4c)
which can be considered as a specific feature in the implementation of the RHP
approach to the characterization of initial boundary problems for the mCH equation.

%-------------------%

\section{Acknowledgment}

IK acknowledges the support from 
 the Austrian Science Fund (FWF), grant no.
10.55776/ESP691. 
IK and DS acknowledge the support from 
Ukrainian-Austrian Joint Programme of Scientific and Technological Cooperation Project
(supported by OeAD, Project WTZ No. UA 08/2025 and Ministry of Education and Science of Ukraine, project no. 0125U003485).

\appendix

\section[The Riemann--Hilbert problem formalism in the case $\tilde{\omega}(0,t)\ge 0$]{\texorpdfstring{$\tilde{\omega}(0,t)\ge 0$}{tilde-omega-positive}}
\label{sec:geq}

As we already mentioned, the analytic properties of distinguished solutions of the Lax 
pair (used in the construction of the associated RH problem) depends significantly 
on the sign of the boundary values $\tilde \omega(0,t)$. Here we briefly discuss the case $\tilde \omega(0,t)\geq 0$.

In this case we have the following properties of $\Phi_{\infty i}^{(j)}$:

\begin{enumerate}
    \item $\Phi_{\infty 1}^{(1)}(x,t,\mu)$ is analytic in $\mathbb{C}\setminus\{\pm1,\pm\ii\}$.
 Moreover, $\Phi_{\infty 1}^{(1)}(x,t,\mu)$=$\begin{pmatrix}
        1\\0
    \end{pmatrix}+O(\frac{1}{\mu})$ as $\mu\to\infty$ in $\{\Im\mu>0\}$.

    \item $\Phi_{\infty 1}^{(2)}(x,t,\mu)$ is analytic in $\mathbb{C}\setminus\{\pm1,\pm\ii\}$.
 Moreover, $\Phi_{\infty 1}^{(2)}(x,t,\mu)$=$\begin{pmatrix}
        0\\1
    \end{pmatrix}+O(\frac{1}{\mu})$ as $\mu\to\infty$ in $\{\Im\mu<0\}$.

     \item $\Phi_{\infty 2}^{(1)}$ is analytic in $\mathbb{C}\setminus\{\pm1,\pm\ii\}$.
 Moreover, $\Phi_{\infty 2}^{(1)}(0,t,\mu)$=$\begin{pmatrix}
        1\\0
    \end{pmatrix}+O(\frac{1}{\mu})$ as $\mu\to\infty$ in $\{\Im\mu<0\}$; $\Phi_{\infty 2}^{(1)}(x,0,\mu)$=$\begin{pmatrix}
        1\\0
    \end{pmatrix}+O(\frac{1}{\mu})$ as $\mu\to\infty$ in $\{\Im\mu>0\}$

    \item $\Phi_{\infty 2}^{(2)}(x,t,\mu)$ is analytic in $\mathbb{C}\setminus\{\pm1,\pm\ii\}$.
 Moreover, $\Phi_{\infty 2}^{(2)}(0,t,\mu)$=$\begin{pmatrix}
        0\\1
    \end{pmatrix}+O(\frac{1}{\mu})$ as $\mu\to\infty$ in $\{\Im\mu<0\}$; $\Phi_{\infty 2}^{(2)}(x,0,\mu)$=$\begin{pmatrix}
        0\\1
    \end{pmatrix}+O(\frac{1}{\mu})$ as $\mu\to\infty$ in $\{\Im\mu<0\}$

      \item $\Phi_{\infty 3}^{(1)}(x,t,\mu)$ is analytic in $\mathbb{C}_-\setminus\{-\ii\}$ and continuous up to the boundary except at $\mu=\pm 1, -\ii$.
 Moreover, $\Phi_{\infty 3}^{(1)}(x,t,\mu)$=$\begin{pmatrix}
        1\\0
    \end{pmatrix}+O(\frac{1}{\mu})$ as $\mu\to\infty$.

    \item $\Phi_{\infty 3}^{(2)}(x,t,\mu)$ is analytic in $\mathbb{C}_+\setminus\{\ii\}$ and continuous up to the boundary except at $\mu=\pm 1, \ii$.
 Moreover, $\Phi_{\infty 3}^{(2)}(x,t,\mu)$=$\begin{pmatrix}
        0\\1
    \end{pmatrix}+O(\frac{1}{\mu})$ as $\mu\to\infty$.   

    \item In the corresponding half-planes

    \begin{subequations}\label{Phi_inf_pm1_geq}
        \begin{align}
            \Phi_{\infty j}&=\frac{\ii}{2(\mu-1)}\alpha_j(x,t)
            \begin{pmatrix}
                -1&1\\-1&1
            \end{pmatrix}+O(1),\quad \mu\to 1,\\
             \Phi_{\infty j}&=-\frac{\ii}{2(\mu+1)}\alpha_j(x,t)
            \begin{pmatrix}
                1&1\\-1&-1
            \end{pmatrix}+O(1),\quad \mu\to -1.
        \end{align}
    \end{subequations}
\end{enumerate}
In particular, this implies that in this case, $A(\mu)$ and $B(\mu)$ are analytic in $\mathbb{C}\setminus\{\pm1,\pm\ii\}$. Moreover, \[ A(\mu)=1+O(\frac{1}{\mu}),\quad B(\mu)=O(\frac{1}{\mu}),\quad \mu\to\infty,\quad \mu\in\mathbb{C}_-.\]
And the \textit{global relation} is as follows

\begin{equation}\label{relations_Phi_03_i__geq}
     \eul^{\frac{\eta(T)+\nu(0)}{2}} ( A(\mu)b(\mu)- B(\mu) a(\mu))\eul^{\frac{4\ii\mu(\mu^2-1)}{(\mu^2+1)^2}T}=O(\mu-\ii),\quad\mu\to\ii,
\end{equation} 

The formulation of the Riemann–Hilbert problem \textbf{RH$^{(x)}$} is independent of the sign of $\tilde \omega(0,t)$. On the other hand, the constructions of \textbf{RH$^{(t)}$} and \textbf{RH$^{(xt)}$} depend on it. For brevity, we omit the intermediate steps and state only the final result.

\subsection{The Riemann--Hilbert problem \textbf{RH$^{(t)}$}}

Recalling the analytic properties of eigenfunctions and scattering coefficients, we introduce the matrix-valued function:
\begin{equation}\label{M_(t)_geq}
M^{(t)}(t,\mu)=\begin{cases}

\left(\Phi_{\infty 1}^{(1)}(0,t,\mu),\frac{\Phi_{\infty 2}^{(2)}(0,t,\mu)}{A^*(\mu)}\right),\quad \mu\in\left(\mathcal{C}_-\cup\mathcal{D}_-\right)\cap\{|\mu\pm\ii| >\epsilon\},\\

\left( \frac{\Psi_{0 2}^{(1)}(0,t,\mu)}{\tilde A(\mu)},\Psi_{0 1}^{(2)}(0,t,\mu)\right),\quad \mu\in\left(\mathcal{C}_-\cup\mathcal{D}_-\right)\cap\{|\mu\pm\ii| <\epsilon\},\\

\left( \frac{\Phi_{\infty 2}^{(1)}(0,t,\mu)}{A(\mu)},\Phi_{\infty 1}^{(2)}(0,t,\mu)\right),\quad \mu\in\left(\mathcal{C}_+\cup\mathcal{D}_+\right)\cap\{|\mu\pm\ii| >\epsilon\},\\

\left(\Psi_{0 1}^{(1)}(0,t,\mu),\frac{\Psi_{0 2}^{(2)}(0,t,\mu)}{\tilde A^*(\mu)}\right),\quad \mu\in\left(\mathcal{C}_+\cup\mathcal{D}_+\right)\cap\{|\mu\pm\ii| <\epsilon\}.

\end{cases}
\end{equation}

Proceeding as in case $\tilde \omega(0,t)\leq0$, after introducing a new parameter $z(t)$, we conclude that $\hat M^{(t)}(z,t,\mu)$ can be characterized as the solution 
of the  following Riemann-Hilbert problem:

Given $ A(\mu)$, $ B(\mu)$ for $\mu\in\mathbb{C}\setminus\{\pm\ii,\pm 1\}$, and a set $\{\nu_j\}_1^K\subset (\mathcal{C}_+\cup\mathcal{D}_+)\cap\{|\mu\pm\ii| >\epsilon\}$, find a piece-wise meromorphic $2\times 2$  matrix valued function $\hat M^{(t)}(z,t,\mu)$ that satisfies the following conditions:

\begin{enumerate}
    \item Jump relation across $\Sigma\cup\{|\mu-\ii|=\epsilon\}\cup\{|\mu+\ii|=\epsilon\}$
    \begin{subequations}
        \label{jump_hatM_(t)_}
        \begin{equation}
           \hat M_-^{(t)}(z,t,\mu)=\hat M_+^{(t)}(z,t,\mu)\hat J^{(t)}(z,t,\mu),\quad\mu\in\mathbb{R} 
        \end{equation}
        where
         \begin{equation}
          \hat J^{(t)}(t,\mu)=\eul^{-\frac{\ii}{4}\frac{\mu^2-1}{\mu}\left(\frac{1}{2}z-\frac{4\mu^2}{(\mu^2+1)^2}t\right)\sigma_3}J^{(t)}_0(\mu)\eul^{\frac{\ii}{4}\frac{\mu^2-1}{\mu}\left(\frac{1}{2}z-\frac{4\mu^2}{(\mu^2+1)^2}t\right)\sigma_3}
        \end{equation} 
        with
\begin{equation}
   J^{(t)}_0(\mu)=\begin{cases}
       \begin{pmatrix}
          1-R(\mu)R^*(\mu)&R(\mu)\\
          -R^*(\mu)&1
       \end{pmatrix},\quad \mu\in\Sigma\cap\{|\mu\pm\ii|>\epsilon\},\\
       \begin{pmatrix}
          1- R(\mu)R^*(\mu)& R(\mu)\eul^{-\frac{\ii (\mu^2-1) }{2\mu}\eta(T)}\\
          - R^*(\mu)\eul^{\frac{\ii (\mu^2-1) }{2\mu}\eta(T)}&1
       \end{pmatrix},\quad \mu\in\Sigma\cap\{|\mu\pm\ii|<\epsilon\},\\

           \begin{pmatrix}
          1-R(\mu)R^*(\mu)&R(\mu)\\
          -R^*(\mu)&1
       \end{pmatrix}\eul^{\frac{\ii (\mu^2-1) }{4\mu}\eta(T)\sigma_3},\quad \mu\in(\mathcal{D}_-\cup\mathcal{C}_-)\cap\{|\mu\pm\ii|=\epsilon\},\\

           \eul^{-\frac{\ii (\mu^2-1) }{4\mu}\eta(T)\sigma_3}\begin{pmatrix}
          1-R(\mu)R^*(\mu)&R(\mu)\\
          -R^*(\mu)&1
       \end{pmatrix},\quad \mu\in(\mathcal{D}_+\cup\mathcal{C}_+)\cap\{|\mu\pm\ii|=\epsilon\}.
       
   \end{cases}
\end{equation}
and $R(\mu):=\frac{B(\mu)}{A^*(\mu)}$.
    \end{subequations}

\item Behavior at $\infty$:
\begin{equation}\label{inf_hatM_(t)_}
     \hat M^{(t)}(z,t,\mu)=I+O(\frac{1}{\mu}),\quad\mu\to\infty.
\end{equation}

\item Behavior at $\pm 1$:

\begin{equation}\label{sing_hatM_(t)_}
\hat M^{(t)}(z,t,\mu)=\begin{cases}
\frac{\ii\beta_+(z,t)}{2(\mu-1)}\begin{pmatrix} -C & 1 \\ -C & 1 \end{pmatrix}+\ord(1), &\mu\to 1,\ \ \Im\mu>0,\\
-\frac{\ii\beta_+(z,t)}{2(\mu +1)}\begin{pmatrix} C & 1 \\ -C & -1 \end{pmatrix}+\ord(1), &\mu\to -1,\ \ \Im\mu>0,
\end{cases}
\end{equation}
with some $\beta_+(z,t)\in\mathbb{R}$ and 
\begin{equation}\label{C-1_hat_}
C:=\begin{cases}
0,&\text{if }\Gamma\neq 0~ \text{(generic case)},\\
\frac{A_1+B_1}{A_1},&\text{if }\Gamma=0,
\end{cases}
\end{equation}
where $A_1=A(1)$, $B_1=B(1)$, and $\Gamma:= -2\ii\lim\limits_{\mu\to 1}(\mu-1)A(\mu)$.

\item Symmetry properties:
\begin{equation}\label{sym-hatM_(t)_}
\hat M^{(t)}(\bar\mu)=\sigma_1\overline{\hat M^{(t)}(\mu)}\sigma_1,\qquad \hat M^{(t)}(-\mu)=\sigma_2\hat M^{(t)}(\mu)\sigma_2,\qquad \hat M^{(t)}(\mu^{-1})=\sigma_1\hat M^{(t)}(\mu)\sigma_1,
\end{equation}

\item Residue conditions: 

\begin{align}\label{res-hatM+_(t)_}
\Res_{\eta_j}\hat M^{(t)(1)}(z,t,\mu)&=\frac{e^{2\left(\frac{1}{2}z-\frac{4\eta_j^2}{(\eta_j^2+1)^2}t\right)}}{\dot A(\eta_j)B(\eta_j)}\hat M^{(t)(2)}(z,t,\eta_j),\\
\label{res-hatM-_(t)_}
\Res_{\bar\eta_j}\hat M^{(t)(2)}(z,t,\mu)&=\frac{e^{-2\left(\frac{1}{2}z-\frac{4\bar\eta_j^2}{(\bar\eta_j^2+1)^2}t\right)}}{\dot A^*(\bar\eta_j)B^*(\bar\eta_j)}\hat M^{(t)(1)}(z,t,\bar\eta_j).
\end{align}

\end{enumerate}

\subsection{The Riemann--Hilbert problem \textbf{RH$^{(xt)}$}}

Recalling the analytic properties of eigenfunctions and scattering coefficients, we introduce the matrix-valued function:

\begin{equation}
M^{(xt)}(x,t,\mu)=\begin{cases}

\left( \frac{\Phi_{\infty 1}^{(1)}(x,t,\mu)}{d(\mu)},\Phi_{\infty 3}^{(2)}(x,t,\mu)\right),\quad \mu\in \mathcal{C}_-\cap\{|\mu\pm\ii| >\epsilon\},\\

\left( \frac{\Psi_{0 2}^{(1)}(x,t,\mu)}{\tilde a(\mu)},\Psi_{0 3}^{(2)}(x,t,\mu)\right),\quad \mu\in \mathcal{C}_-\cap\{|\mu\pm\ii| <\epsilon\},\\

\left( \frac{\Phi_{\infty 2}^{(1)}(x,t,\mu)}{a(\mu)},\Phi_{\infty 3}^{(2)}(x,t,\mu)\right),\quad \mu\in\mathcal{D}_+\cap\{|\mu\pm\ii| >\epsilon\},\\

\left( \frac{\Psi_{0 1}^{(1)}(x,t,\mu)}{\tilde d(\mu)},\Psi_{0 3}^{(2)}(x,t,\mu)\right),\quad \mu\in\mathcal{D}_+\cap\{|\mu\pm\ii| <\epsilon\},\\

\left(\Phi_{\infty 3}^{(1)}(x,t,\mu),\frac{\Phi_{\infty 1}^{(2)}(x,t,\mu)}{d^*(\mu)}\right),\quad \mu\in\mathcal{C}_+\cap\{|\mu\pm\ii| >\epsilon\},\\

\left(\Psi_{0 3}^{(1)}(x,t,\mu),\frac{\Psi_{0 2}^{(2)}(x,t,\mu)}{\tilde a^*(\mu)}\right),\quad \mu\in\mathcal{C}_+\cap\{|\mu\pm\ii| <\epsilon\},\\

\left(\Phi_{\infty 3}^{(1)}(x,t,\mu),\frac{\Phi_{\infty 2}^{(2)}(x,t,\mu)}{a^*(\mu)}\right),\quad \mu\in\mathcal{D}_-\cap\{|\mu\pm\ii| >\epsilon\},\\

\left(\Psi_{0 3}^{(1)}(x,t,\mu),\frac{\Psi_{0 1}^{(2)}(x,t,\mu)}{\tilde d^*(\mu)}\right),\quad \mu\in\mathcal{C}_+\cap\{|\mu\pm\ii| <\epsilon\}.
\end{cases}
\end{equation}

Again, proceeding as in case $\tilde \omega(0,t)\leq0$, after introducing a new variable $y(x,t)$, we conclude that $\hat M^{(t)}(z,t,\mu)$ can be characterized as the solution 
of the  following Riemann-Hilbert problem:

\textbf{The Riemann--Hilbert problem for $\hat M^{(t)}(z,t,\mu)$ (RH$^{(t)}$):} Given $a(\mu)$, $b(\mu)$ for $\mu\in\mathbb{C}_+$, $A(\mu)$, $B(\mu)$ for $\mu\in\mathbb{C}\setminus\{\pm1,\pm\ii\}$ and the sets $\{\kappa_j\}_1^J\subset \mathcal{C}_-\cap\{|\mu+\ii| >\epsilon\}$ and $\{\mu_j\}_1^N\subset\mathbb{C}_+$ with $|\mu_j|=1$, find a $2\times 2$ meromorphic matrix $\hat M^{(t)}(z,t,\mu)$ that satisfies the following conditions:

\begin{enumerate}
    \item Jump relation across $\check\Sigma$ depicted in Figure \ref{fig:contour_RH_t}
    \begin{subequations}
        \label{jump_hatM_(xt)_}
        \begin{equation}
           \hat M_-^{(xt)}(y,t,\mu)=\hat M_+^{(xt)}(y,t,\mu)\hat J^{(xt)}(y,t,\mu),\quad\mu\in\check\Sigma 
        \end{equation}
        where
         \begin{equation}
         \hat J^{(xt)}(y,t,\mu)=\eul^{-\hat p(y,t,\mu)\sigma_3} J^{(xt)}_0(\mu)\eul^{\hat p(y,t,\mu)\sigma_3}
        \end{equation} 
        with $\hat p(y,t,\mu)=\frac{\ii(\mu^2-1)}{2\mu}\left(\frac{1}{2}y-\frac{4\mu^2}{(\mu^2+1)^2}t\right)$,
\begin{equation}
   J^{(xt)}_0(\mu)=\begin{cases}
      \begin{pmatrix}
          1&\frac{aB-bA}{d^*}\\
          -\frac{a^*B^*-b^*A^*}{d}&\frac{1}{dd^*}
       \end{pmatrix},\quad \mu\in\mathbb{R}\setminus\left(\{|\mu-1|\leq\sqrt{2}\}\triangle\{|\mu+1|\leq\sqrt{2}\}\right),\\

       \begin{pmatrix}
          1-r(\mu)r^*(\mu)&r(\mu)\\
          -r^*(\mu)&1
       \end{pmatrix},\quad \mu\in\mathbb{R}\cap\left(\{|\mu-1|<\sqrt{2}\}\triangle\{|\mu+1|<\sqrt{2}\}\right),\\

      \begin{pmatrix}
          1&0\\
          -\frac{B^*}{ad}&1
       \end{pmatrix},\quad \mu\in\mathbb{C}_+\cap\{|\mu-\ii|>\epsilon\}\cap\left(\{|\mu-1|=\sqrt{2}\}\cup\{|\mu+1|=\sqrt{2}\}\right),\\

      \begin{pmatrix}
          1&\frac{B}{a^*d^*}\\
          0&1
       \end{pmatrix},\quad \mu\in\mathbb{C}_-\cap\{|\mu-\ii|>\epsilon\}\cap\left(\{|\mu-1|=\sqrt{2}\}\cup\{|\mu+1|=\sqrt{2}\}\right),\\

             \begin{pmatrix}
          1&0\\
          -\frac{ B^*}{ a d}\eul^{\frac{\ii (\mu^2-1) }{2\mu}\nu(0)}&1
       \end{pmatrix},\quad \mu\in\mathbb{C}_+\cap\{|\mu-\ii|<\epsilon\}\cap\left(\{|\mu-1|=\sqrt{2}\}\cup\{|\mu+1|=\sqrt{2}\}\right),\\

   \end{cases}
\end{equation}  
        
\begin{equation}
   J^{(xt)}_0(\mu)=\begin{cases}
             
      \begin{pmatrix}
          1&\frac{ B}{ a^* d^*}\eul^{-\frac{\ii (\mu^2-1) }{2\mu}\nu(0)}\\
          0&1
       \end{pmatrix},\quad \mu\in\mathbb{C}_-\cap\{|\mu-\ii|<\epsilon\}\cap\left(\{|\mu-1|=\sqrt{2}\}\cup\{|\mu+1|=\sqrt{2}\}\right),\\

        \begin{pmatrix}
          1&0\\
          -\frac{B^*}{ad}&1
\end{pmatrix}\eul^{\frac{\ii}{4}\frac{\mu^2-1}{\mu}\nu(0)\sigma_3}
,\quad \mu\in\mathcal{C}_-\cap\{|\mu-\ii|=\epsilon\},\\

        \eul^{-\frac{\ii}{4}\frac{\mu^2-1}{\mu}\nu(0)\sigma_3}\begin{pmatrix}
          1&\frac{B}{a^*d^*}\\
          0&1
       \end{pmatrix},\quad \mu\in\mathcal{C}_+\cap\{|\mu+\ii|=\epsilon\},\\

                \eul^{-\frac{\ii}{4}\frac{\mu^2-1}{\mu}\nu(0)\sigma_3} \begin{pmatrix}
          1&0\\
          -\frac{B^*}{ad}&1
       \end{pmatrix},\quad \mu\in\mathcal{D}_+\cap\{|\mu-\ii|=\epsilon\},\\
       
        \begin{pmatrix}
          1&\frac{B}{a^*d^*}\\
          0&1
\end{pmatrix}\eul^{\frac{\ii}{4}\frac{\mu^2-1}{\mu}\nu(0)\sigma_3},\quad \mu\in\mathcal{D}_-\cap\{|\mu+\ii|=\epsilon\}     
   \end{cases}
\end{equation}

and $r(\mu):=\frac{b(\mu)}{a^*(\mu)}$.

    \end{subequations}

\item Behavior at $\infty$:
\begin{equation}\label{inf_hatM_(xt)_}
    \hat M^{(xt)}(x,t,\mu)=I+O(\frac{1}{\mu}),\quad\mu\to\infty.
\end{equation}

\item Behavior at $\pm 1$:

\begin{equation}\label{sing_hatM_(xt)_}
\hat M^{(xt)}(y,t,\mu)=\begin{cases}
\frac{\ii\alpha_+(x,t)}{2(\mu-1)}\begin{pmatrix} -c & 1 \\ -c & 1 \end{pmatrix}+\ord(1), &\mu\to 1,\ \ \Im\mu>0,\\
-\frac{\ii\alpha_+(x,t)}{2(\mu +1)}\begin{pmatrix} c & 1 \\ -c & -1 \end{pmatrix}+\ord(1), &\mu\to -1,\ \ \Im\mu>0,
\end{cases}
\end{equation}
with some $\alpha_+(x,t)\in\mathbb{R}$ and 
\begin{equation}\label{tilC-1__}
c:=\begin{cases}
0,&\text{if }\gamma\neq 0 \text{ (generic case)},\\
\frac{b_1+a_1}{a_1},&\text{if }\gamma=0,
\end{cases}
\end{equation}
where $a_1=a(1)$, $b_1=b(1)$ and $\gamma:= -2\ii\lim\limits_{\mu\to 1}(\mu-1)a(\mu)$. 

\item Symmetry properties:
\begin{equation}\label{sym-hatM_(xt)_}
\hat M^{(xt)}(\bar\mu)=\sigma_1\overline{\hat M^{(xt)}(\mu)}\sigma_1,\qquad \hat M^{(xt)}(-\mu)=\sigma_2\hat M^{(xt)}(\mu)\sigma_2,\qquad \hat M^{(xt)}(\mu^{-1})=\sigma_1\hat M^{(xt)}(\mu)\sigma_1,
\end{equation}

\item Residue properties:

\begin{align}\label{res-hatM+_(xt)_1_}
\Res_{\mu_j}\hat M^{(xt)(1)}(y,t,\mu)&=\frac{e^{2\hat p(y,t,\mu_j)}}{\dot a(\mu_j)b(\mu_j)}\hat M^{(xt)(2)}(y,t,\mu_j),\\\label{res-hatM+_(xt)_2_}
\Res_{\kappa_j}\hat M^{(xt)(1)}(y,t,\mu)&=\frac{B^*(\kappa_j)e^{2\hat p(y,t,\kappa_j)}}{\dot d(\kappa_j)a(\kappa_j)}\hat M^{(xt)(2)}(y,t,\kappa_j),\\
\label{res-hatM-_(xt)_1_}
\Res_{\bar\mu_j}\hat M^{(xt)(2)}(y,t,\mu)&=\frac{e^{-\hat 2p(y,t,\bar\mu_j)}}{\dot a^*(\bar\mu_j)b^*(\bar\mu_j)}\hat M^{(xt)(1)}(y,t,\bar\mu_j),\\
\label{res-hatM-_(xt)_2_}
\Res_{\bar\kappa_j}\hat M^{(xt)(2)}(y,t,\mu)&=\frac{B(\bar \kappa_j)e^{-2\hat p(y,t,\bar\kappa_j)}}{\dot d^*(\bar\kappa_j)}\hat M^{(xt)(1)}(y,t,\bar\kappa_j).
\end{align}

\end{enumerate}

\begin{proposition}
Let $\tilde u(x,t)$ be a solution of the mCH equation in the domain $x>0$, $0<t<T$ such that $\tilde\omega(0,t)\leq0$. Then $\tilde u(x,t)$ can be represented (as in \eqref{hat_u_via_RH_xt} -- \eqref{v_j_via_RH_xt})  in terms of the unique solution of a Riemann--Hilbert problem \eqref{jump_hatM_(xt)_}--\eqref{res-hatM-_(xt)_2_}, for which the data (jump matrix and residue condition) are given in terms of the
initial values $\tilde u(x,0)$ and the boundary values $\tilde u(0,t)$, $\tilde u_x(0,t)$, $\tilde u_{xx}(0,t)$ via the associated spectral functions.
 
\end{proposition}

\bibliographystyle{RS}
\bibliography{shepelsky_etal}

\end{document}